\documentclass[12pt,a4paper]{article}
\usepackage[hmargin=2.5cm,vmargin=2.5cm]{geometry}
\usepackage{amssymb,latexsym,amsmath,amsfonts,amsthm,color,bbm}
\usepackage{graphicx,psfrag}
\usepackage{hyperref}
\hypersetup{colorlinks=true}

\DeclareMathOperator{\diag}{diag}
\DeclareMathOperator{\Res}{Res}
\DeclareMathOperator{\Ai}{Ai}
\DeclareMathOperator{\BK}{BK}
\DeclareMathOperator{\KMW}{KMW}
\DeclareMathOperator{\Pea}{P}

\newcommand{\R}{\mathbb{R}}
\newcommand{\id}{\mathbbm 1}
\renewcommand{\P}{\mathbf P}
\renewcommand{\d}{\mathrm d}
\newcommand{\Or}{\mathcal O}

\newcommand{\uu}{u}
\newcommand{\UU}{\bf U}
\newcommand{\vv}{v}
\newcommand{\er}{\mathbb{R}}
\newcommand{\cee}{\mathbb{C}}

\newcommand{\zet}{\mathbb{Z}}
\newcommand{\ii}{\mathbf{i}}
\newcommand{\lam}{\lambda}

\newcommand{\al}{\alpha}

\newcommand{\wtil}{\widetilde}
\newcommand{\what}{\widehat}
\renewcommand{\Re}{\mathrm{Re}\,}

\newcommand{\ud}{\,\mathrm{d}}

\newcommand{\pa}{\partial}
\newcommand{\pee}{\mathbf{p}}

\newcommand{\Kaa}{K}
\newcommand{\res}{\mathrm{res}}
\newcommand{\wt}{\widetilde}
\newcommand{\wh}{\widehat}
\newcommand{\si}{\sigma}
\newcommand{\Aa}{\mathbb{A}}
\newcommand{\Bb}{\mathbb{B}}
\newcommand{\Cc}{\mathbb{C}}
\newcommand{\Dd}{\mathbb{D}}
\newcommand{\Ss}{\mathbb{S}}
\newcommand{\Mm}{\mathbb{M}}
\newcommand{\vbf}{\mathbf{v}}
\newcommand{\alt}{\mathrm{alt}}
\newcommand{\DD}[1]{\mathcal D_{#1}}
\newcommand{\EE}{\mathbb E}
\newcommand{\FF}{\mathbb F}
\newcommand{\GG}{\mathbb G}
\newcommand{\HH}{\mathbb H}

\newtheorem{theorem}{Theorem}[section]
\newtheorem{lemma}[theorem]{Lemma}
\newtheorem{proposition}[theorem]{Proposition}

\newtheorem{corollary}[theorem]{Corollary}

\newtheorem{definition}[theorem]{Definition}

\newtheorem{remark}[theorem]{Remark}

\newtheorem{conj}[theorem]{Conjecture}

\numberwithin{equation}{section}

\allowdisplaybreaks

\author{Steven Delvaux\thanks{Department of Mathematics, University of Leuven (KU Leuven),
Celestijnenlaan 200B, B--3001 Leuven, Belgium. E-mail: {\tt steven.delvaux\symbol{'100}wis.kuleuven.be}} \and
B\'alint Vet\H o\thanks{Institute for Applied Mathematics, Bonn University, Endenicher Allee 60, 53115 Bonn, Germany;
MTA--BME Stochastics Research Group, Egry J.\ u.\ 1, 1111 Budapest, Hungary. E-mail: {\tt vetob@math.bme.hu}}}
\title{The hard edge tacnode process and the hard edge Pearcey process with non-intersecting squared Bessel paths}

\begin{document}

\maketitle

\begin{abstract}
A system of non-intersecting squared Bessel processes is considered which all start from one point and they all return to another point.
Under the scaling of the starting and ending points when the macroscopic boundary of the paths touches the hard edge,
a limiting critical process is described in the neighbourhood of the touching point which we call the hard edge tacnode process.
We derive its correlation kernel in an explicit new form which involves Airy type functions and operators that act on the direct sum of $L^2(\R_+)$ and a finite dimensional space.
As the starting points of the squared Bessel paths are set to $0$, a cusp in the boundary appears.
The limiting process is described near the cusp and it is called the hard edge Pearcey process.
We compute its multi-time correlation kernel which extends the existing formulas for the single-time kernel.
Our pre-asymptotic correlation kernel involves the ratio of two Toeplitz determinants which are rewritten using a Borodin--Okounkov type formula.
\end{abstract}

\section{Introduction}

In recent years, the investigations of non-intersecting Brownian motion and random walk paths
focused on the description of a critical process called the tacnode process.
This process appears when two groups of trajectories are asymptotically supported in two ellipses in the time-space plane
such that the ellipses touch each other creating a tacnode (self-touching point) of the macroscopic boundary.
The aim is to describe the behaviour of the paths near the touching point.

A series of recent papers by different groups of authors studied the tacnode process in parallel using various methods.
The first result in this direction is due to Adler, Ferrari and van Moerbeke:
in \cite{AFvM11}, a model of non-intersecting random walk paths is considered such that the paths form a symmetric tacnode.
In \cite{DKZ}, Delvaux, Kuijlaars and Zhang studied the non-symmetric case of non-intersecting Brownian trajectories
and they expressed the critical correlation kernel in terms of the solution of a $4\times4$ Riemann--Hilbert problem.
A different approach to the Brownian case is due to Johansson who gave a formula for the tacnode kernel
in terms of the resolvent of the Airy kernel in \cite{Joh11} in the symmetric case.
The latter approach was extended by Ferrari and Vet\H o in \cite{FV} to the general non-symmetric case.
The tacnode process was also obtained in the tiling problem of the double Aztec diamond by Adler, Johansson and van Moerbeke in \cite{AJvM}.

It was not a priori clear that the various formulas for the tacnode kernel give rise to the same limit process,
since the results in \cite{AFvM11}, \cite{Joh11} and \cite{FV} contain Airy resolvent type formulas whereas the kernel in \cite{DKZ} is expressed with the solution of a Riemann--Hilbert problem.
It was shown in \cite{AJvM} that the formulation of \cite{AFvM11} and \cite{Joh11} are equivalent.
A more recent result \cite{D} gives the equivalence of the Riemann--Hilbert formulas in \cite{DKZ}
and the Airy resolvent formulas in \cite{Joh11} and \cite{FV}.

\begin{figure}
\begin{center}
\def\svgwidth{200pt}
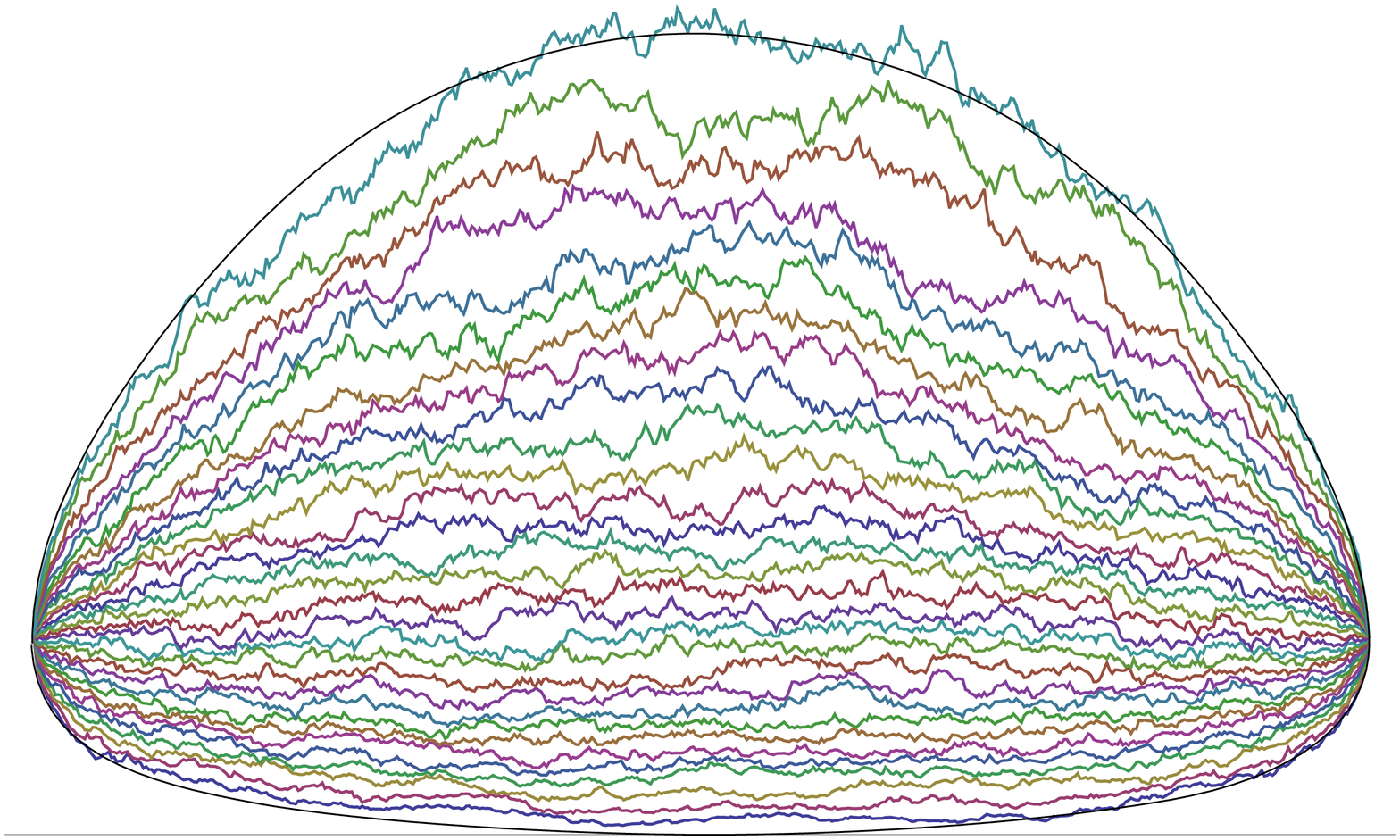\qquad
\def\svgwidth{200pt}
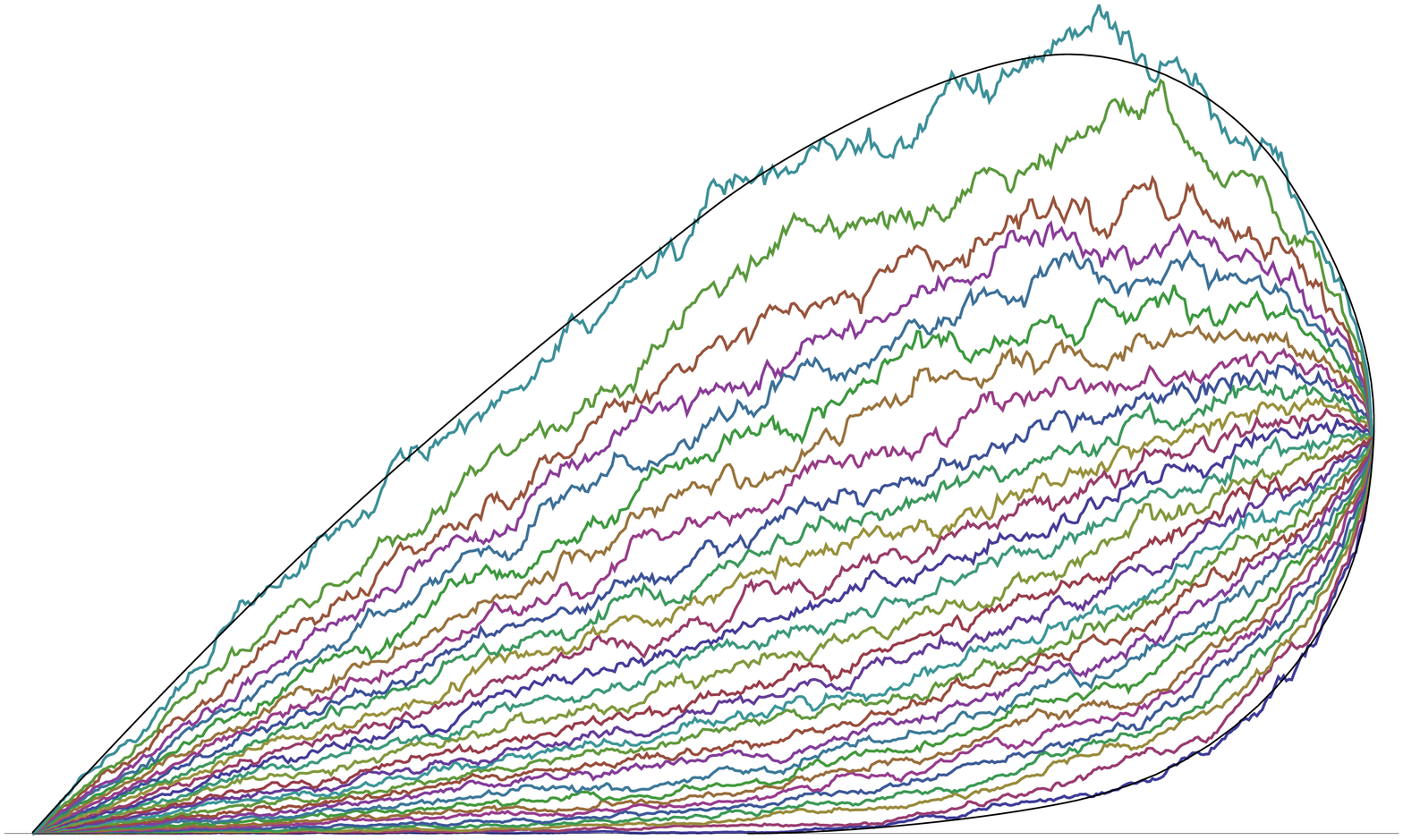
\end{center}
\caption{Simulation picture in the time-space plane of $n=30$ non-intersecting squared Bessel paths with the tacnode (left-hand side) and the Pearcey (right-hand side) scaling of the starting and endpoints.
The horizontal axis in both cases denotes the time $t$, running from time $t=0$ (starting time) to time $t=1$ (ending time).
The paths are conditioned to be non-intersecting throughout the time interval $t\in (0,1)$, and they have a fixed starting position $a$ and ending position $b$ at times $t=0$ and $t=1$ respectively.
In the left hand picture, $a,b>0$ are fine-tuned so that the limiting hull of the paths touches the $x=0$ line (horizontal axis) at a critical point, the hard edge tacnode.
In the right hand picture, we have $a=0$ and the limiting hull of the paths forms a cusp with the horizontal axis;
in the neighborhood of the cusp point, the hard edge Pearcey process is observed.\label{fig:simulation}}
\end{figure}

Apart from systems of non-intersecting Brownian motions and random walk paths,
the study of non-intersecting squared Bessel paths is also natural due to their representation as eigenvalues of the Laguerre process, see~\cite{KO01}.
This process is a positive definite matrix valued Brownian bridges.
Since the transition kernel of the squared Bessel paths can be given explicitly via modified Bessel functions, see \eqref{transitionprob:sqB},
a Karlin--McGregor type formula like \eqref{extendedkernel:start} can be applied.
Non-intersecting squared Bessel paths were also studied by Katori and Tanemura in \cite{KT11}, but the first description of the hard edge tacnode process is due to Delvaux \cite{Dhet}.
The formation of the tacnode in this case is slightly different: instead of two touching groups of trajectories in the time-space plane,
only one group of paths is considered and its boundary macroscopically touches the hard edge, i.e.\ the $x=0$ line.
This configuration is referred to as the hard edge tacnode and the critical process in the neighbourhood of the touching point is called the hard edge tacnode process.
In \cite{Dhet}, the correlation kernel of the hard edge tacnode process is expressed
in terms of the solution of a $4\times4$ Riemann--Hilbert problem which is different from the one that appears in the case of the Brownian trajectories.
See the left-hand side of Figure~\ref{fig:simulation} for non-intersecting squared Bessel paths with the tacnode scaling of the starting and endpoints.
For further figures, see also~\cite{D}.

In Theorem~\ref{thm:hardedgetacnode} which is the main result of the present paper
we provide an explicit Airy type formula for the multi-time correlation kernel of the hard edge tacnode process, under the assumption that the parameter $\al$ of the Bessel process is a (non-negative) integer.
The kernel is expressed as a double complex integral, see \eqref{deflimkernel}.
The formula is completely new in the literature.
It is reminiscent to the usual tacnode kernel, but the Airy resolvent operator on $L^2(\R_+)$ is replaced by an operator which acts on the direct sum space of $L^2(\R_+)$ and a finite dimensional space.

A different phenomenon appears when the starting points of the non-intersecting squared Bessel paths are taken to be $0$ and the endpoints are scaled linearly with the number of paths.
In this case, there is a critical time such that for any earlier time the lowest path stays close to $0$ whereas for any later time the distance of the lowest path from $0$ is macroscopic.
After rescaling around the critical time, one obtains the hard edge Pearcey process, see the right-hand side of Figure~\ref{fig:simulation}.
The single time hard edge Pearcey kernel was already described by Desrosiers and Forrester \cite{DF}
and later in a different formulation by Kuijlaars, Mart\'\i nez-Finkelshtein and Wielonsky \cite{KMW2}.
We give the multi-time correlation kernel of the hard edge Pearcey process in different formulations in Theorems~\ref{thm:hardedgePearcey} and~\ref{thm:Pearcey:alternative}
which also shows the equivalence of the formulas in \cite{DF} and \cite{KMW2}.
In Corollary~\ref{cor:BK}, we obtain the kernel by Borodin and Kuan \cite{BK} in the special case of one-dimensional squared Bessel paths, i.e.\ of absolute values of one-dimensional Brownian motions.

The main steps how the convergence of the kernel of $n$ non-intersecting squared Bessel paths to that of the hard edge tacnode is proved are the following.
First we write the kernel as a double complex contour integral using the representation of the modified Bessel function.
The integrand can be transformed into the ratio of two Toeplitz determinants of sizes $n-1$ and $n$ respectively.
The Toeplitz determinants have symbols with non-zero winding numbers around $0$, hence we apply a generalized Borodin--Okounkov formula due to B\"ottcher and Widom in \cite{BW} to obtain Theorem~\ref{theorem:toeplitz}
if the parameter $\alpha$ of the squared Bessel paths is a (non-negative) integer.
Then we obtain the ratio of two Fredholm type determinants where the two operators are rank one perturbations of each other.
This yields a resolvent type formula in Theorem~\ref{theorem:kernel:n} which is suitable for asymptotic analysis.
The functions that appear in the finite $n$ kernel can be rewritten in terms of Bessel functions, hence the asymptotic analysis relies on the convergence of Bessel functions which is proved separately in the appendix.

Now we introduce the model of non-intersecting squared Bessel paths that we consider.
The squared Bessel process depends on a parameter $\alpha>-1$.
The transition probability of the squared Bessel process for any time $t>0$ is defined by
\begin{equation}\label{transitionprob:sqB}
p_t(x,y) = \frac{1}{2t}
\left(\frac{y}{x}\right)^{\alpha/2}\exp\left(-\frac{x+y}{2t}\right)I_{\alpha}\left(\frac{\sqrt{xy}}{t}\right)
\end{equation}
for $x>0$ and $y\ge0$ where $I_\alpha$ is the modified Bessel function which can be given by the series
\begin{equation}\label{BesselI:def}
I_{\alpha}(z) = \sum_{k=0}^{\infty}
\frac{(z/2)^{2k+\alpha}}{k!\Gamma(k+\alpha+1)}.
\end{equation}
The transition probability $p_t(0,y)$ is obtained by taking the limit $x\to 0$.
If $d=2(\alpha+1)$ is an integer, then the squared Bessel process can be obtained as the squared absolute value of a $d$-dimensional Brownian motion.
In this case, we call $d$ the dimension of the squared Bessel process.

In the present paper, we consider non-intersecting squared Bessel paths which start from one fixed point and end at another fixed point.
In order to construct this system of paths, we first take $n$ non-intersecting squared Bessel paths with fixed different starting points $a_1>a_2>\ldots>a_n>0$ at time $t=0$
and fixed different ending points $b_1>b_2>\ldots>b_n>0$ at time $t=1$.
The paths are conditioned to be non-intersecting in the time interval $(0,1)$.
It is well known that this defines an extended determinantal point process.
That is, the joint probability at a sequence of times $0<t_1<t_2<\ldots<t_m<1$ can be expressed via the determinant of an extended correlation kernel $K_n(s,x;t,y)$.
The kernel $K_n$ is defined in terms of the transition probability $p_t(x,y)$ in \eqref{transitionprob:sqB}.
Namely as given also in \cite[Eq.~(1.1)]{Joh11}, we have
\begin{equation}\label{extendedkernel:start}
K_n(s,x;t,y) = -p_{t-s}(x,y) \id_{t>s}+ \sum_{j,k=1}^n p_{1-s}(x,b_k)(A^{-1})_{k,j} p_t(a_j,y)
\end{equation}
for any positions $x,y>0$ and times $s,t\in (0,1)$ with $\id_{t>s}$ denoting the characteristic function of $t>s$ and with $A$ defined as the $n\times n$ matrix
\begin{equation}\label{extendedkernel:Amatrix}
A = \left(p_1(a_j,b_k)\right)_{j,k=1}^n.
\end{equation}

Next we take the confluent limit of the starting and ending points $a_j\to a\geq 0$ and $b_j\to b> 0$.
If $a\geq 0$ and $b>0$ are suitably scaled with the number $n$ of the paths, we can create a picture in the time-space plane with a cusp or a tacnode at the hard edge.
In this paper, we obtain the limiting extended correlation kernel of non-intersecting squared Bessel paths near the tacnode
and the cusp which we call the hard edge tacnode process and the hard edge Pearcey process.

The paper is organized as follows.
We first state our main results in Section~\ref{s:mainres}.
The correlation kernel $K_n$ is expressed as a ratio of two Toeplitz determinants and in a pre-asymptotic form for finite $n$ in Subsection~\ref{ss:finiten}.
Then two different scalings and the corresponding limit processes are considered:
the hard edge tacnode process is introduced and discussed along with our results on the convergence in Subsections~\ref{ss:hetacnode} and \ref{ss:alt_tacnode}.
The hard edge Pearcey process with different formulations of its correlation kernel and our results are given in Subsection~\ref{ss:hePearcey}.
The finite $n$ formulas for the kernel $K_n$ are proved in Section~\ref{s:derivation_finite}.
The asymptotic analysis for the hard edge tacnode process is performed in Section~\ref{s:analysis_tacnode}, the one for the hard edge Pearcey process is in Section~\ref{s:analysis_Pearcey}.
The alternative hard edge Pearcey formulas are proved in Section~\ref{s:alternative_Pearcey}.
The proof of Proposition~\ref{prop:limSchur} which is needed for the existence of the given formulation of the hard edge tacnode kernel is postponed to Section~\ref{s:asymptinv}.
Section~\ref{s:Bessel_limit} of the appendix contains a statement about the convergence of the derivatives of Bessel functions to those of the Airy function and useful tail bounds.

\section{Main results}\label{s:mainres}

We first report our formulas for $n$ non-intersecting squared Bessel paths for $n$ finite.
Next we describe the limits under the tacnode and the Pearcey scaling.

\subsection{Correlation kernel of non-intersecting squared Bessel paths}\label{ss:finiten}

In what follows, we will assume that the parameter $\alpha$ of the squared Bessel process is a (non-negative) integer.
See however Remark~\ref{remark:nonint} where we state an extension of one of our results for finite $n$ to the case of non-integer $\alpha$.

Define the weight function
\begin{equation}\label{weight:wxyt}
w(z;x,y,t) = \frac{1}{4\pi \ii t}z^{\alpha}\exp\left(\frac{z-1}{2t}x\right)\exp\left(\frac{z^{-1}-1}{2t}y\right)
\end{equation}
on the unit circle $|z|=1$ with parameters $x,y,t>0$.
We abbreviate
\begin{equation}\label{weight:w}
w(z):=w(z;a,b,1)
\end{equation}
where $a,b>0$ denote the starting and ending point of the non-intersecting squared Bessel paths.
We also define the weight function
\begin{equation}\label{weight:wtil}
\wtil w(z) := (1-\xi z)(1-\eta^{-1}z^{-1})w(z)
\end{equation}
where $\xi,\eta$ are free parameters for the moment, but they will depend on the integration variables in \eqref{extendedkernel:toeplitz} as given by \eqref{xi:eta} below.

First, we express the kernel of $n$ non-intersecting squared Bessel paths as a double integral where the integrand is a ratio of two Toeplitz determinants.
The proof of the theorem is given in Section~\ref{subsection:proof:theorem:Toeplitz}.

\begin{theorem}[Toeplitz determinant formula for $K_n$]\label{theorem:toeplitz}
Let $\alpha$ be an integer and consider $n$ non-intersecting squared Bessel paths of parameter $\alpha$ with starting point $a>0$ and ending point $b>0$.
Then the extended correlation kernel can be written as
\begin{multline}\label{extendedkernel:toeplitz} K_n(s,x;t,y) = -p_{t-s}(x,y)
\id_{t>s} + \int_{S_1}\int_{S_1} C_{n}(\xi,\eta)w(u;a,y,t)w(v;x,b,1-s)\frac{\ud u}{u}\frac{\ud v}{v}
\end{multline}
where $S_1$ denotes the unit circle in the complex plane, oriented counterclockwise and
\begin{equation}\label{Toeplitz:ratio}
C_{n}(\xi,\eta) := \frac{\eta^{n-1}}{\xi^{n-1}}\det\left(\int_{S_1} z^{j-k}\wtil w(z)\frac{\ud z}{z}\right)_{j,k=1}^{n-1} /
\det\left(\int_{S_1} z^{j-k}w(z)\frac{\ud z}{z}\right)_{j,k=1}^{n}
\end{equation}
is a ratio of two Toeplitz determinants with weight functions $w(z)$ and $\wtil w(z)$ defined in \eqref{weight:wxyt}--\eqref{weight:wtil} with
\begin{equation}\label{xi:eta}
\xi:=\left(\frac{u-1+t}{t}\right)^{-1},\qquad \eta :=\frac{v^{-1}-s}{1-s}.
\end{equation}
\end{theorem}

\begin{remark}[Non-integer $\al$]\label{remark:nonint}
Theorem~\ref{theorem:toeplitz} can be extended to arbitrary real values of the parameter $\alpha$ ($\alpha>-1$), not necessarily integer.
To that end, it suffices to replace both integration contours $S_1$ in \eqref{Toeplitz:ratio} by a contour $C$
where $C$ is a contour encircling the origin in counterclockwise direction, beginning at and returning to $-\infty$, and never intersecting the negative real line except at $-\infty$.
The contour $C$ is the negative of the standard `Hankel contour'.
We then assume that all powers in the formulas have a branch cut along the negative real axis.
Note that in the special case where $\alpha$ is integer, the integration over the contour $C$ can be replaced by an integration over
the unit circle $S_1$ and then we retrieve the formula stated in Theorem~\ref{theorem:toeplitz} above.
The proof of the extension to non-integer $\alpha$ is similar to the one of Theorem~\ref{theorem:toeplitz} given in Section~\ref{subsection:proof:theorem:Toeplitz},
by noting that \eqref{intrepr:unitcircle:1} holds for non-integer values of $\al$ provided that the integration contour $S_1$ is replaced by the contour $C$ described above.
We omit the details.

Unfortunately, we were unable to apply an asymptotic analysis to the Toeplitz determinants for this generalized setting with $\al$ non-integer,
therefore we will always assume below that $\al$ is integer.
\end{remark}

In order to calculate the asymptotics of the Toeplitz determinants in \eqref{Toeplitz:ratio}, we apply a Borodin--Okounkov type formula.
Note that the symbols of our Toeplitz matrices have winding number $\alpha$ around the origin, which is in general non-zero.
A Borodin--Okounkov type formula in this setting is described by B\"ottcher--Widom \cite{BW}.

The formulas in \cite{BW} allow to write each Toeplitz determinant as the Fredholm determinant of an operator on $l^2(\zet_{\geq 0})$ times a determinant of size $\al\times \al$.
Alternatively, we will see that the Toeplitz determinant can be expressed via the Fredholm determinant of a \emph{single} operator acting on the direct sum space
\begin{equation}\label{direct:sum:space} L:=l^2(\zet_{\geq 0})\oplus\cee^\alpha.\end{equation}

Fix $n\in\zet_{>0}$. We define a block matrix operator of size $(\infty+\al)\times (\infty+\al)$
\begin{equation}\label{block:matrix:def}
\begin{array}{rl} & \hspace{-4mm}\begin{array}{rr} \infty & \al\end{array} \\ \begin{array}{r} \infty\\ \al\end{array} & \hspace{-5mm}\begin{pmatrix} A & C \\ B & D\end{pmatrix}\end{array}
\end{equation}
which we view as the matrix representation of a linear operator acting on the space $L$ \eqref{direct:sum:space} with respect to the natural basis of $L$.
The block matrix decomposition is compatible with the direct sum decomposition of $L$.

Let us describe the four blocks of \eqref{block:matrix:def}. We will denote by $S_\rho$ the circle with radius $\rho>1$.
\begin{itemize}
\item $A$ is a semi-infinite matrix with $(k,l)$th entry
\begin{equation}\label{A:def}
A_{k,l}=\delta_{k-l}- \frac{1}{(2\pi \ii)^2}\int_{S_{\rho^{-1}}}\ud z\int_{S_{\rho}}\ud w\frac{z^{l+n+\al}}{w^{k+n+\al+1}} \frac{1}{w-z}
\exp\left(\frac{w-z}{2}b+\frac{z^{-1}-w^{-1}}{2}a\right)
\end{equation}
for $k,l\in\zet_{\geq 0}$.

\item $B$ is a matrix of size $\al\times\infty$ with $(k,l)$th entry given by
\begin{equation}\label{B:def}
B_{k,l}=- \frac{1}{(2\pi \ii)^2}\int_{S_{\rho^{-1}}}\ud z\int_{S_{\rho}}\ud w\frac{z^{l+n+\al}}{w^{k+n+1}} \frac{1}{w-z}
\exp\left(\frac{w-z}{2}b+\frac{z^{-1}-w^{-1}}{2}a\right)
\end{equation}
for $k=0,\ldots,\al-1,\quad l\in\zet_{\geq 0}$.

\item $C$ is a Toeplitz matrix of size $\infty\times\al$  with $(k,l)$th entry
\begin{equation}\label{C:def}
C_{k,l} = \frac{1}{2\pi \ii}\int_{S_\rho}\ud w\,w^{l-k-n-\al-1}\exp\left(\frac{bw-aw^{-1}}{2}\right),\qquad k\in\zet_{\geq 0},\quad l=0,\ldots,\al-1.
\end{equation}

\item $D$ is a Toeplitz matrix of size $\al\times\al$ with $(k,l)$th entry
given by
\begin{equation}\label{D:def} D_{k,l} = \frac{1}{2\pi \ii}\int_{S_\rho}\ud w\,w^{l-k-n-1}\exp\left(\frac{bw-aw^{-1}}{2}\right),\qquad k,l=0,\ldots,\al-1.
\end{equation}
\end{itemize}

Next we define two vectors in the space $L$.
The vectors will depend on parameters $\xi,\eta\in\cee$ which we consider for the moment to be fixed numbers such that $|\xi|<1$ and $|\eta|>1$.
Denote again by $S_\rho$ the circle of radius $\rho$ which we now take such that $\rho\in(1,\min\{|\xi|^{-1},|\eta|\})$.
We define $\mathbf{h},\mathbf{\what h}$ to be the column vectors of length $\infty$ and $\al$ respectively with $k$th entry
\begin{align}
h_k &= \frac{1}{2\pi\ii}\int_{S_{\rho}}\ud w\, \frac{w^{-(k+n+\al+1)}}{w+\eta}\exp\left(\frac{bw-aw^{-1}}{2}\right),& k&\in\zet_{\geq 0},\label{hk:def}\\
\what h_k &= \frac{1}{2\pi \ii}\int_{S_{\rho}}\ud w\,\frac{w^{-(k+n+1)}}{w+\eta} \exp\left(\frac{bw-aw^{-1}}{2}\right),& k&=0,\ldots,\al-1.\label{hktil:def}
\end{align}
Let $\mathbf{g}$ be the row vector of length $\infty$ with $l$th entry
\begin{equation}\label{gk:def}
g_l = \frac{1}{2\pi \ii}\int_{S_{\rho^{-1}}}\ud z\,\frac{z^{l+n+\al}}{z+\xi}\exp\left(\frac{az^{-1}-bz}{2}\right),\qquad l\in\zet_{\geq 0}
\end{equation}
and we define $\boldsymbol{\xi}$ as the row vector of length $\al$ given by
\begin{equation}\label{xibold:def}
\boldsymbol{\xi} := \begin{pmatrix} 1 & -\xi & \ldots & (-\xi)^{\al-1}\end{pmatrix}.
\end{equation}
With these notations, let
\begin{equation}\label{M:xi:eta}
M(\xi,\eta)=\frac{1}{\eta-\xi}- \begin{pmatrix} \mathbf{g} & \boldsymbol{\xi} \end{pmatrix}\begin{pmatrix} A & C \\ B & D\end{pmatrix}^{-1}\begin{pmatrix} \bf h\\ \bf\what h\end{pmatrix}
\end{equation}
where the matrix is the same as the one given in \eqref{block:matrix:def}
and the entries of the two vectors are defined in \eqref{gk:def}--\eqref{xibold:def} and \eqref{hk:def}--\eqref{hktil:def} respectively.
The block matrix and block vector notation is again compatible with the direct sum decomposition of the space $L$ given in \eqref{direct:sum:space}.

\begin{remark}
The inverse of the block matrix in the above formula can again be partitioned as a block matrix of size $(\infty+\al)\times (\infty+\al)$:
\begin{equation}\label{schur:inversion}
\begin{pmatrix} A & C \\ B & D\end{pmatrix}^{-1} = \begin{pmatrix} A^{-1}+A^{-1}CS^{-1}BA^{-1} & -A^{-1}CS^{-1} \\ -S^{-1}BA^{-1} & S^{-1} \end{pmatrix}
\end{equation}
where $S:=D-BA^{-1}C$ is the Schur complement of $A$.
Hence the invertibility of the block matrix in \eqref{schur:inversion} follows from the invertibility of $A$ and $S$.

It is not clear a priori that the block matrix in the definition of $M(\xi,\eta)$ in \eqref{M:xi:eta} is invertible,
but as it will be seen later in Proposition~\ref{prop:limSchur} for the tacnode scaling and in \eqref{detDDalpha} for the Pearcey scaling, we get invertible matrices in these two limits.
Therefore, $M(\xi,\eta)$ is certainly well-defined if the parameters are close enough to any of these limits.
\end{remark}

\begin{theorem}\label{theorem:inner:product}
The ratio of Toeplitz determinants in \eqref{Toeplitz:ratio} can be written as
\begin{equation}\label{Cn:inner:product}
C_{n}(\xi,\eta) = 2\frac{\eta^{n+\al}}{\xi^{n-1}}\exp\left(\frac{1-\eta^{-1}}{2}a+\frac{1-\xi}{2}b\right) M(\xi,\eta)
\end{equation}
where $M(\xi,\eta)$ is given above by \eqref{M:xi:eta}.
\end{theorem}

The proof of the theorem is given in Section~\ref{subsection:proof:theorem:BO}.

\begin{definition}[Contours $\Gamma_\xi,\Gamma_\eta$]\label{def:contour:xi:eta}
We denote by $\Gamma_\xi$ and $\Gamma_\eta$ counterclockwise oriented closed contours as follows.
$\Gamma_\xi$ is a loop encircling $0$ but not the points $-s/(1-s)$ and $-t/(1-t)$.
$\Gamma_\eta$ consists of two pieces: a loop encircling $0$ lying at the inside of $\Gamma_\xi$, and a small loop surrounding the point at $\eta=-s/(1-s)$ lying at the outside of $\Gamma_\xi$.
A particular choice is shown later in Figure~\ref{fig:xietacontours}.
\end{definition}

\begin{theorem}\label{theorem:kernel:n}
Under the same assumptions as in Theorem~\ref{theorem:toeplitz}, the correlation kernel $K_n$ for $n$ non-intersecting squared Bessel paths can be written as
\begin{multline}\label{K_nkernelxieta}
K_n(s,x;t,y)=-p_{t-s}(x,y)\id_{t>s}+\frac1{2(2\pi\ii)^2}\int_{\Gamma_\eta}\d\eta\int_{\Gamma_\xi}\d\xi
\frac{((1-t)\xi+t)^{\alpha-1}}{((1-s)\eta+s)^{\alpha+1}}\\
\times\frac{\eta^{n+\alpha}}{\xi^{n+\al}}
\exp\left(\frac{\eta-\xi}2b+\frac{\xi^{-1}-\eta^{-1}}2a-\frac{x(\eta-1)}{2(1-s)\eta+2s}+\frac{y(\xi-1)}{2(1-t)\xi+2t}\right)M(\xi,\eta)
\end{multline}
where $M(\xi,\eta)$ is defined by \eqref{M:xi:eta} and the contours $\Gamma_\eta$ and $\Gamma_\xi$ are given in Definition~\ref{def:contour:xi:eta}.
\end{theorem}

\subsection{Hard edge tacnode process}\label{ss:hetacnode}

It turns out that it is more convenient for our purposes to work with
\begin{equation}\label{defN}
N=\frac n2
\end{equation}
when we consider the hard edge tacnode process.
The free parameter $q\in(0,\infty)$ will parametrize the location of the hard edge tacnode process, however $q$ will disappear from the limit process.

\begin{definition}\label{def:tacscale}
Let the \emph{tacnode scaling} be the scaling when time and space are scaled according to
\begin{equation}\label{timespacescaling}
\frac q{1+q}+\frac q{(1+q)^2}tN^{-1/3},\qquad\frac {2q}{(1+q)^2}xN^{-1/3}
\end{equation}
and the starting and ending points are rescaled by
\begin{align}
a&=2qN\left(1-\frac\sigma{2N^{2/3}}\right),\label{defa}\\
b&=2q^{-1}N\left(1-\frac\sigma{2N^{2/3}}\right)\label{defb}
\end{align}
where $\sigma\in\R$ is the \emph{temperature parameter}.
\end{definition}

For $\sigma\in\R$, we define the shifted Airy kernel by
\begin{equation}\label{def:shiftedAi}
\Kaa_{\Ai,\sigma}(x,y)=\int_\sigma^\infty \d\lambda \Ai(x+\lambda)\Ai(y+\lambda).
\end{equation}
To state the result about the limiting kernel, we define the function
\begin{equation}\label{h:repr:0}
h(x;\vv)=\int_\sigma^{\infty} \d\lambda\Ai(x+\lambda)\exp\left(-\vv\lambda\right)
\end{equation}
and the two-variate functions (with $\Ai^{(k)}$ denoting the $k$th derivative of the Airy function)
\begin{align}
\Aa_\sigma(x,y)&=(\id-\Kaa_{\Ai,\sigma})(x,y), && x,y\in\R_+,\label{defAa}\\
\Bb_\sigma(k,y)&=-\int_\sigma^\infty\d\lambda\Ai^{(k)}(\lambda)\Ai(y+\lambda), && k\in\{0,1,\dots,\alpha-1\},y\in\R_+,\label{defBb}\\
\Cc_\sigma(x,l)&=\Ai^{(l)}(x+\sigma), && x\in\R_+,l\in\{0,1,\dots,\alpha-1\},\label{defCc}\\
\Dd_\sigma(k,l)&=\Ai^{(k+l)}(\sigma), && k,l\in\{0,1,\dots,\alpha-1\}.\label{defDd}
\end{align}
The matrix component $\Cc_\sigma$ always appears with a subscript throughout this paper in order to avoid confusions with the set of complex numbers.
Let
\begin{equation}\label{defXi}
{\bf\UU}=\begin{pmatrix} 1 & \uu & \uu^2 & \dots & \uu^{\alpha-1} \end{pmatrix}
\end{equation}
and for any function $f(x)$ or $f(x;\vv)$, we define the column vector on the space $L^2(\R_+)\oplus\mathbb C^\alpha$
\begin{equation}\label{vf:def}
\vbf[f]=\begin{pmatrix} (f(x))_{x\in\R_+} & f(x)|_{x=0} & \frac{\pa}{\pa x} f(x)|_{x=0} & \ldots & \frac{\pa^{\al-1}}{\pa x^{\al-1}} f(x)|_{x=0} \end{pmatrix}^T.
\end{equation}
Finally define
\begin{equation}\label{defM}
\Mm(\uu,\vv)=\frac{\exp\left(\sigma(\uu-\vv)\right)}{\vv-\uu}+\begin{pmatrix} (h(y;-\uu))_{y\in\R_+} & -{\bf\UU}\exp(\sigma\uu) \end{pmatrix}
\begin{pmatrix} \Aa_\sigma & \Cc_\sigma \\ \Bb_\sigma & \Dd_\sigma\end{pmatrix}^{-1} \vbf[h(\cdot;\vv)].
\end{equation}

\begin{definition}[Contours $\Gamma_\uu$, $\Gamma_\vv$]\label{def:contours}
The contours depend on real parameters $s,t$ but we will not indicate this dependence in the notation.
Let $\Gamma_\uu$ be a Jordan arc in the complex plane which comes from $e^{-\ii2\pi/3}\infty$, goes to $e^{\ii2\pi/3}\infty$ and it crosses the real axis on the right of $-s$ and $-t$.
Let $\Gamma_\vv$ consist of two pieces: a Jordan arc in the complex plane coming from $e^{-\ii\pi/3}\infty$ and going to $e^{\ii\pi/3}\infty$, lying to the right of $\Gamma_\uu$,
and a clockwise oriented small loop around $-s$, lying to the left of $\Gamma_\uu$.
A possible choice of these contours can be seen in Figure~\ref{fig:HXicontours}.
\end{definition}

\begin{figure}
\begin{center}
\def\svgwidth{200pt}
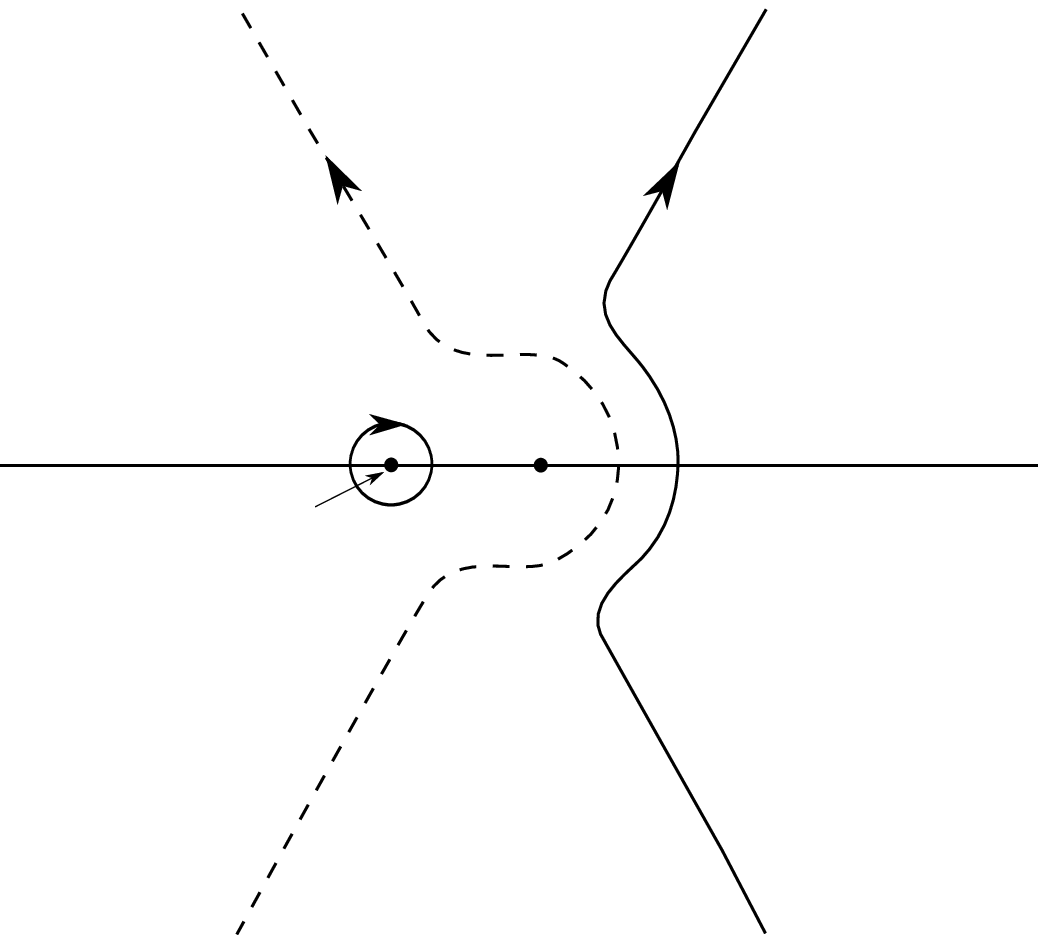
\end{center}
\caption{The integration contours $\Gamma_\uu$ and $\Gamma_\vv$.\label{fig:HXicontours}}
\end{figure}

Let
\begin{equation}\begin{aligned}\label{deflimkernel}
K^\alpha(s,x;t,y)=&-p_{\frac{t-s}{2}}(x,y)\id_{t>s}\\
&+\frac1{(2\pi\ii)^2}\int_{\Gamma_\vv}\d\vv\int_{\Gamma_\uu}\d\uu \frac{\exp\left(\frac{\vv^3}3+\frac x{\vv+s}\right)}{\exp\left(\frac{\uu^3}3+\frac y{\uu+y}\right)}
\frac{(\uu+t)^{\alpha-1}}{(\vv+s)^{\alpha+1}}\Mm(\uu,\vv)
\end{aligned}\end{equation}
be the hard edge tacnode kernel where the integration contours are given in Definition~\ref{def:contours} and the function $\Mm(\uu,\vv)$ is defined by \eqref{defM}.
The next theorem will be proved in Section~\ref{s:analysis_tacnode}.

\begin{theorem}\label{thm:hardedgetacnode}
Let $\alpha$ be a non-negative integer.
\begin{enumerate}
\item
Under the tacnode scaling given by Definition~\ref{def:tacscale}, the sequence of rescaled kernels converges
\begin{equation}\label{kernelconv}
2rK_{2N}\left(\frac q{1+q}+rs,2rx;\frac q{1+q}+rt,2ry\right)\to K^\alpha(s,x;t,y)
\end{equation}
uniformly as $x$ and $y$ are in a compact subset of $\R$ where
\begin{equation}
r=q(1+q)^{-2}N^{-1/3}
\end{equation}
and the kernel on the right-hand side of \eqref{kernelconv} is defined by \eqref{deflimkernel}.

\item
As a consequence, the hard edge tacnode process $\mathcal T^\alpha$ exists as the limit of $n$ non-intersecting squared Bessel processes
under the tacnode scaling.
It is characterized by the following gap probabilities.
For any fixed $k$ and $t_1,\dots,t_k\in\R$ and for any compact set $E\subset\{t_1,\dots,t_k\}\times\R$,
\begin{equation}
\P(\mathcal T^\alpha(\id_E)=\emptyset)=\det(\id-K^\alpha)_{L^2(E)}.
\end{equation}
\end{enumerate}
\end{theorem}

In the definition of the kernel $K^\alpha(s,x,t,y)$ in \eqref{defXi}--\eqref{defM}, we implicitly assume that the limiting block matrix is invertible.
Equivalently, the limiting Schur complement
\begin{equation}\label{schur:lim}
\Ss_\sigma= \Dd_\sigma-\Bb_\sigma{\Aa}_\sigma^{-1}{\Cc}_\sigma
\end{equation}
should be invertible.
This follows from Proposition~\ref{prop:limSchur} below for $\si$ large enough.

\subsection{Alternative formulations of the hard edge tacnode kernel}\label{ss:alt_tacnode}

Let $A_\si$ be the integral operator on $L^2(0,\infty)$ with kernel
\begin{equation}\label{defAsi}
A_\si(x,y)=\Ai(x+y+\si).
\end{equation}
Note that $K_{\Ai,\si}=A_\si^2$ and the matrices in \eqref{defAa}--\eqref{defDd} can also be easily expressed in terms of the operator $A_\si$ as
\begin{align}
\Aa_\si(x,y)&=(\id-A_\si^2)(x,y),&
\Cc_\si(x,l)&=\frac{\pa^l}{\pa y^l}A_\si(x,y)\bigm|_{y=0},\label{AC}\\
\Bb_\si(k,y)&=-\frac{\pa^k}{\pa x^k}A_\si^2(x,y)\bigm|_{x=0},&
\Dd_\si(k,l)&=\frac{\pa^k}{\pa x^k}\frac{\pa^l}{\pa y^l}A_\si(x,y)\bigm|_{x=y=0}.\label{BD}
\end{align}

We will use the following general notation in the sequel: if $f$ is a smooth function on $(-1,1)$, we denote the kernel of the integral operator $f(A_\si)$ on $L^2(0,\infty)$ by $f(A_\si)(x,y)$.
In particular, we have $f(x)=x/(1-x^2)$ in \eqref{schur:repr} below.
Straightforward substitution of \eqref{AC}--\eqref{BD} into \eqref{schur:lim} yields the following.

\begin{lemma}\label{lemma:schur:repr}
\begin{enumerate}
\item The Schur complement $\Ss_\si$ in \eqref{schur:lim} can be represented as
\begin{equation}\label{schur:repr}
\Ss_\si = \left(\frac{\pa^{k+l}}{\pa x^k\pa y^l}\left(A_\si(\id-A_\si^2)^{-1}(x,y)\right)\Bigm|_{x=y=0} \right)_{k,l=0}^{\alpha-1}.
\end{equation}
\item The Schur complement $\Ss_\si$ is a symmetric matrix.
\end{enumerate}
\end{lemma}

In particular, for $\al=1$ the Schur complement takes the scalar form
\[\Ss_\si = (\id-A_\si^2)^{-1}A_\si(0,0).\]
It is well-known \cite{TW1} that this is equal to $q(\si)$ where $q$ is the Hastings--McLeod solution to the Painlev\'e~II equation.
The invertibility of the Schur complement then amounts to the fact that $q(\si)\neq 0$ for all $\si\in\er$.
We conjecture that it holds more generally.

\begin{conj}\label{conj:schur:inv}
The Schur complement \eqref{schur:repr} is invertible for all $\si\in\er$.
\end{conj}

Let us prove Conjecture \ref{conj:schur:inv} under the assumption that $\si\gg0$ is large enough.
Note that for $\si\to\infty$ the operator $(\id-\Kaa_{\Ai,\si})^{-1}$ tends exponentially fast to the identity operator.
Hence the $(k,l)$th entry of the matrix \eqref{schur:repr} is approximated by
\[\frac{\pa^{k+l}}{\pa x^{k+l}}\Ai(x+\si)|_{x=0}.\]
We will prove the following proposition in Section~\ref{s:asymptinv}.

\begin{proposition}\label{prop:limSchur}
\begin{enumerate}
\item Let $n$ be fixed. We have the asymptotics
\begin{equation}\label{Airy:der:matrix}
\det\left(\frac{\pa^{k+l}}{\pa x^{k+l}}\Ai(x)\right)_{k,l=0}^{n-1}
\sim \frac{(-1)^{\binom n2}}{2^{\binom{n+1}2}\pi^{n/2}}\left(\prod_{j=0}^{n-1} j!\right) x^{-n^2/4}\exp\left(-\frac 23 n x^{3/2}\right)
\end{equation}
for $x\to\infty$ where $\sim$ means that the ratio of the two sides goes to $1$.

\item The Schur complement $\Ss_\si$ is invertible for all $\si\gg0$ large enough.
\end{enumerate}
\end{proposition}

We take the derivative of the kernel~\eqref{deflimkernel} with respect to the temperature parameter $\si$.
It turns out that this derivative has a convenient form: it is a rank 1 kernel.
The proof of this fact is a rather lengthy calculation which we omit here, hence we give the next theorem without proof.

\begin{theorem}[Temperature derivative of the kernel]\label{theorem:temp:der}
The derivative of the hard edge tacnode kernel \eqref{deflimkernel} with respect to $\si$ is a rank 1 kernel, that is
\begin{equation}\label{tempder:rank1}
\frac{\pa}{\pa \si} K^\al(s,x;t,y) = y^{\al} g_\si(s,x) h_\si(-t,y).
\end{equation}
If $\alpha=0$, then $g_\sigma(s,x)=h_\sigma(s,x)$ and they both are given by the contour integral
\begin{equation}
\frac1{2\pi\ii}\int_{\Gamma_\vv}\d\vv\,\frac{e^{\frac{\vv^3}3+\frac x{\vv+s}}}{\vv+s}\left(e^{-\vv\sigma}+\int_0^\infty\d w\int_0^\infty\d z\,(\id-K_{\Ai,\sigma})^{-1}(w,z)\Ai(w+\sigma)h(z;\vv)\right).
\end{equation}
If $\alpha\ge1$, then the function $g_\si$ is given by
\begin{equation}\label{f1:def}
g_\si(s,x) = \frac{1}{2\pi\ii}\int_{\Gamma_\vv}\ud\vv\,\frac{\exp\left(\frac{\vv^3}3+\frac x{\vv+s}\right)}{(\vv+s)^{\alpha+1}}
\begin{pmatrix}0& \dots& 0& 1\end{pmatrix} \begin{pmatrix} \Aa_\si & \Cc_\si\\ \Bb_\si & \Dd_\si\end{pmatrix}^{-1} \vbf[h(\cdot;\vv)]
\end{equation}
and $h_\si$ can be represented as
\begin{multline}\label{f2:def}
h_\si(s,x) = -x^{-\al}\frac{1}{2\pi\ii}\int_{\Gamma_\uu}\ud \uu\, \frac{(\uu-s)^{\alpha-1}}{\exp\left(\frac{\uu^3}3+\frac x{\uu-s}\right)}\\
\times\left[\uu^{\al}\exp\left(\si\uu\right)+\begin{pmatrix} (h(y;-\uu))_{y\in\R_+} & -\boldsymbol{\UU}\exp\left(\si\uu\right) \end{pmatrix}\begin{pmatrix} \Aa_\si & \Cc_\si\\
\Bb_\si & \Dd_\si\end{pmatrix}^{-1} \vbf[\Ai^{(\al)}(\cdot+\si)]\right].
\end{multline}
\end{theorem}

\begin{conj}
The functions $g_\si$ and $h_\si$ are the same.
\end{conj}

The conjecture holds for $\alpha=0$ trivially, the $\alpha=1$ case can still be checked by rewriting the functions in Schur complement forms.
For general $\alpha$, it is more complicated.
The equality of the two functions $g_\sigma$ and $h_\sigma$ follows from the hidden symmetry of the kernel.
Integration of \eqref{tempder:rank1} with respect to the temperature yields the following.

\begin{corollary}
The hard edge tacnode kernel can be written as
\begin{equation}
K^\al(s,x;t,y) = -y^{\al} \int_\si^\infty\ud\si' g_{\si'}(s,x) g_{\si'}(-t,y).
\end{equation}
\end{corollary}

Finally, we show how the block matrix notations in the above formulas can be avoided.
We show this first for the function $\Mm(\uu,\vv)$.
Let us use the notation
\begin{equation}\label{defexp}
e_s(x)=\exp(s(x+\sigma))
\end{equation}
for the exponential function for any $s\in\mathbb C$ and let
\begin{equation}
\langle f,g\rangle=\int_0^\infty f(x)g(x)\,\d x
\end{equation}
be the usual scalar product of functions $f$ and $g$ in $L^2(\R_+)$.
Observe that with this notation
\begin{equation}\label{hrewrite}
h(x;\vv)=A_\si e_{-\vv}(x).
\end{equation}

\begin{proposition}\label{lemma:Mschur} We can rewrite the function $\Mm(\uu,\vv)$ in \eqref{defM} with the notation of \eqref{defAsi} and \eqref{defexp} as
\begin{multline}\label{defM:schur}
\Mm(\uu,\vv)=\langle e_\uu,(\id-\Kaa_{\Ai,\si})^{-1}e_{-\vv}\rangle
\\ -\left(\frac{\pa^k}{\pa x^k}\left((\id-\Kaa_{\Ai,\si})^{-1}e_\uu\right)(x)\Bigm|_{x=0}\right)_{k=0}^{\al-1} \Ss_\si^{-1}
\left(\frac{\pa^k}{\pa x^k}\left(A_\si(\id-\Kaa_{\Ai,\si})^{-1}e_{-\vv}\right)(x)\Bigm|_{x=0}\right)_{k=0}^{\al-1}.
\end{multline}
Here the two expressions between large parentheses $(\dots)_{k=0}^{\al-1}$ denote a row and a column vector respectively, both of length $\al$.
\end{proposition}

In a very similar way, one can express the functions $g_\sigma$ and $h_\sigma$ without the block matrix notation.

\begin{remark}\label{rem:cafasso}
In the $\alpha=0$ case after the change of variables $\eta=(\vv+s)^{-1}$ and $\xi=(\uu+t)^{-1}$, the kernel in \eqref{deflimkernel} restricted to single time reads
\begin{equation}\label{alpha0singletimekernel}
K^\alpha(t,x;t,y)=\frac1{(2\pi\ii)^2}\int_{\gamma_\eta}\d\eta\int_{\gamma_\xi}\d\xi
\frac{\exp\left(\frac{(\eta^{-1}-t)^3}3+\frac x\eta\right)}{\exp\left(\frac{(\xi^{-1}-t)^3}3+\frac y\xi\right)}\frac1{\xi-\eta}.
\end{equation}
The contour $\gamma_\xi$ is a clockwise curve that has a cusp at the origin such that it leaves the origin in the $e^{\ii2\pi/3}$ direction,
it crosses the positive real axes and it returns to the origin from the $e^{-\ii2\pi/3}$ direction.
The other contour $\gamma_\eta$ consists of two parts: the first one is a clockwise loop inside $\gamma_\xi$ and it has a cusp at the origin
such that it leaves the origin in the $e^{\ii\pi/3}$ direction and it returns to the origin from the $e^{-\ii\pi/3}$ direction;
the other part is a counterclockwise circle around $\gamma_\xi$.

The kernel in \eqref{alpha0singletimekernel} with a further change of variables is a special case of the single time kernels of the form
\begin{equation}\label{singletimekernel}
\frac1{(2\pi\ii)^2}\int_{\gamma_\mu}\d\mu\int_{\gamma_\lambda}\d\lambda\frac{e^{\Theta_x(\mu)-\Theta_y(\lambda)}}{\lambda-\mu}
\end{equation}
where $\Theta_x(\mu)$ is a rational function in $\mu$ which depends on the parameter $x$ linearly, i.e.\ $\Theta_x(\mu)=\Theta_x(0)+x\mu$.
In this case, $\Theta_x$ is a third order rational function.
In the hard edge Pearcey case, it is a second order function, see Corollary~\ref{cor:KMW} below and also \cite{KMW2}.
As discussed in \cite{BC11}, to the type of kernels \eqref{singletimekernel} with some further restrictions on $\Theta_x$, one can associate a Riemann--Hilbert problem
which was used in \cite{BC11} to express the gap probabilities in terms of the tau function of the Riemann--Hilbert problem.
It is natural to ask about the existence of a corresponding Riemann--Hilbert problem in the present case.
\end{remark}

\subsection{Hard edge Pearcey process}\label{ss:hePearcey}

Here we consider the non-intersecting squared Bessel paths which all start from $0$.
More precisely, let
\begin{equation}\label{KnPea}
K_n^{\Pea}(s,x;t,y):=\lim_{a_i\to 0} K_n(s,x;t,y)
\end{equation}
be the confluent limit of the kernel \eqref{extendedkernel:start} of $n$ non-intersecting squared Bessel processes as the starting points tend to $0$.
This means that $K_n^{\Pea}(s,x;t,y)$ is the correlation kernel for $n$ non-intersecting squared Bessel processes
conditioned to start at $0$ at time $0$ and end at positions $b_j,\ j=1,\dots,n$ at time $1$.
This is similar to the situation considered by Kuijlaars, Mart\'\i nez-Finkelshtein and Wielonsky \cite[Fig.\ 1]{KMW2}.

\begin{proposition}\label{prop:confluent:a0}
For any real $\alpha>-1$, the kernel $K_n^{\Pea}$ in \eqref{KnPea} has the double integral representation
\begin{multline}\label{doubleintegral:confluent:a0}
K_n^{\Pea}(s,x;t,y) = -p_{t-s}(x,y) \id_{t>s} \\
- \frac1{2\pi\ii}\int_{-\infty}^0\d w\int_{\Gamma_z}\d z\ p_{t-1}(w,y)
p_{1-s}(x,z)\left(\frac{w}{z}\right)^{\alpha} \exp\left( \frac{z-w}{2} \right)
\frac{1}{w-z}\prod_{j=1}^n \frac{w-b_j}{z-b_j}
\end{multline}
where the $z$-contour $\Gamma_z$ is an counterclockwise loop surrounding the points $b_1,\ldots,b_n$ and being disjoint from the $w$-contour $(-\infty,0]$.
The branch for $\left(w/z\right)^{\alpha}$ is defined by $w^\alpha=e^{\ii\alpha\pi}|w|^\alpha$ and along the contour $\Gamma_z$, we use the extension of $z^{-\alpha}$ from the positive real axis.
\end{proposition}

This proposition is closely related to a result by Katori--Tanemura \cite[Theorem 2.1]{KT11}.
It can also be considered as the hard edge analogue of Tracy--Widom \cite[eq.\ (2.11)]{TW3}.
For completeness we will provide a proof of the proposition in Section~\ref{s:analysis_Pearcey}.
The confluent limit of the endpoints $b_j\to b$ for $j=1,\dots,n$ is simply obtained by substituting it in \eqref{doubleintegral:confluent:a0}.
If we choose $b=2qn$ where $q\in(0,\infty)$ is a free parameter, then the region filled by the trajectories of the Bessel processes behaves as follows.
For any time $t<1/(1+q)$, the lowest path stays close to $0$ whereas for $t>1/(1+q)$, its distance from $0$ is macroscopic.
Our main focus is the neighbourhood of $t=1/(1+q)$ under the following scaling where the hard edge Pearcey process is observed.

\begin{definition}\label{def:peascaling}
Let the \emph{Pearcey scaling} be the following scaling of the parameters.
We consider $n$ non-intersecting squared Bessel paths, take the confluent limit as all the starting points are $0$ and the endpoints are equal to $b=2qn$ with a free parameter $q\in(0,\infty)$
and we rescale time and space variables $t$ and $y$ as
\begin{equation}\label{spacetimescale}
\frac1{1+q}+\frac q{(1+q)^2}tn^{-1/2},\qquad \frac1{2q}yn^{-1/2}
\end{equation}
respectively. (We make this rescaling for all the considered time and space variables.)
\end{definition}

We define the hard edge Pearcey kernel by
\begin{multline}\label{Lalphadef}
L^\alpha(s,x;t,y)=-p_{\frac{t-s}{2}}(x,y)\id_{t>s}\\
+\left(\frac yx\right)^{\alpha/2}\frac2{\pi\ii}\int_{C}\d v\int_{0}^\infty\d u\, \left(\frac uv\right)^{\alpha}
\frac{uv}{v^2-u^2}\frac{e^{v^4/2+sv^2}}{e^{u^4/2+tu^2}}J_\alpha\left(2\sqrt{y}u\right)J_\alpha\left(2\sqrt{x}v\right)
\end{multline}
where $C$ is the contour which consists of two rays: one from $e^{\ii\pi/4}\infty$ to $0$ and one from $0$ to $e^{-\ii\pi/4}\infty$.
In the single time case $s=t$, the above kernel essentially reduces to the one of Desrosiers--Forrester~\cite[Sec.~5]{DF} which was also discussed in the last paragraph of \cite[Sec.~1]{KMW2}.
Applying an asymptotic analysis to the double integral formula in Proposition~\ref{prop:confluent:a0} similarly to \cite[Sec.\ III]{TW3}, we obtain the multi-time extension of the hard edge Pearcey process.

\begin{theorem}\label{thm:hardedgePearcey}
Let $\alpha>-1$.
The hard edge Pearcey process $\mathcal P^\alpha$ is obtained as the limit of $n$ non-intersecting squared Bessel processes starting at $0$ at time $0$ and ending at $b=2qn$ at time $1$
under the time-space scaling \eqref{spacetimescale}.
It is defined by the following gap probabilities.
For any fixed $k$ and $t_1,\dots,t_k\in\R$ and for any compact set $E\subset\{t_1,\dots,t_k\}\times\R$,
\begin{equation}
\P(\mathcal P^\alpha(\id_E)=\emptyset)=\det(\id-L^\alpha)_{L^2(E)}
\end{equation}
where $L^\alpha$ is the hard edge Pearcey kernel given by \eqref{Lalphadef}.
\end{theorem}

\begin{remark}\label{rem:Pearceysigmadep}
It is not hard to see from the proof of Theorem~\ref{thm:hardedgePearcey} that if we set $b=2qn(1+\si n^{-1/2})$ for some fixed $\si\in\R$,
then the correlation kernel of the hard edge Pearcey process becomes $L^\alpha(s+\si,x;t+\si,y)$.
Similarly as before, one can think of $\si$ as a \emph{temperature} parameter.
\end{remark}

For non-negative integer values of $\alpha$, we can use our methods involving Toeplitz determinants to obtain an alternative formulation of the kernel $L^{\alpha}$.
The following result is proved in Section~\ref{s:alternative_Pearcey}.

\begin{proposition}[Alternative formula for the hard edge Pearcey kernel]\label{thm:Pearcey:alternative}
If $\alpha$ is a non-negative integer, then we can write $L^\alpha$ in \eqref{Lalphadef} as
\begin{multline}\label{Lalpha:KMW:step1}
L^\alpha(s,x;t,y) = -p_{\frac{t-s}{2}}(x,y)\id_{t>s}\\
+\left(\frac yx\right)^\alpha \frac1{(2\pi\ii)^2}\int_{\Gamma_{-s}}\d w\, \int_{\delta+\ii\R}\d z\,
\frac{1}{w-z} \frac{e^{-\frac{w^2}2+\si w+\frac x{w+s}}}{e^{-\frac{z^2}2+\si z+\frac y{z+t}}}\frac{(w+s)^{\al-1}}{(z+t)^{\al+1}}
\end{multline}
where $\Gamma_{-s}$ is a clockwise oriented circle surrounding the singularity at $-s$, and $\delta>0$ is chosen such that the contour $\delta+\ii\R$
passes to the right of the singularity at $-t$ and to the right of the contour $\Gamma_{-s}$.
\end{proposition}

\begin{remark}\label{rem:Pearcy:nonint}
Although we derive it only if $\al$ is a non-negative integer, formula~\eqref{Lalpha:KMW:step1} makes sense for any real $\al>-1$ with a few minor modifications.
In that case, we replace $\Gamma_{-s}$ by a clockwise oriented closed loop which intersects the real line at a point to the right of $-s$, and also at $-s$ itself where it has a cusp at angle $\pi$.
We then take the principal branches of the powers $(w+s)^{\al-1}$, $(z+t)^{\al+1}$, i.e., with a branch cut along the negative half-line. See also Remark~\ref{remark:KMW:noninteger}.
\end{remark}

In the single time case $t=s$, the kernel $L^\al$ in \eqref{Lalpha:KMW:step1} can be further reduced to the kernel of Kuijlaars, Mart\'inez--Finkelshtein and Wielonsky,
as stated in~\cite[Eq.\ (1.19)]{KMW2}:
\begin{equation}\label{KMW:kernel}
K_\al^{\KMW}(x,y,t) = \frac{1}{(2\pi\ii)}\int_{\Gamma_v}\ud v\int_{\Gamma_u}\ud u \frac{1}{u-v}\frac{e^{t v^{-1}+v^{-2}/2+xv}}{e^{t u^{-1}+u^{-2}/2+yu}}\frac{v^\al}{u^\al}
\end{equation}
where $\Gamma_v$ is a clockwise circle in the left half-plane touching zero along the imaginary axis, and $\Gamma_u$ is a counterclockwise loop surrounding $\Gamma_v$.

\begin{corollary}\label{cor:KMW}
In the single time case $t=s$, the kernel $L^\al$ in \eqref{Lalpha:KMW:step1} reduces to
\begin{equation}\label{KMW:reduction}
L^\al(t,x;t,y) = \left(\frac yx\right)^\al K_\al^{\KMW}(y,x,t+\si).
\end{equation}
\end{corollary}

As a consequence of Theorem~\ref{thm:hardedgePearcey} in the special case of $\alpha=-1/2$, we recover the process described by Borodin and Kuan \cite{BK} via the kernel
\begin{multline}
K^{\BK}(\sigma_1,\eta_1;\sigma_2,\eta_2)=\frac2{\pi^2\ii}\int_C\d v\int_{\R_+}\d u\frac v{v^2-u^2}\frac{e^{v^4+\eta_2v^2}}{e^{u^4+\eta_1u^2}}\cos(\sigma_1u)\cos(\sigma_2v)\\
-\frac1{2\sqrt{\pi(\eta_1-\eta_2)}}\left(\exp\frac{(\sigma_1+\sigma_2)^2}{4(\eta_2-\eta_1)}+\exp\frac{(\sigma_1-\sigma_2)^2}{4(\eta_2-\eta_1)}\right)\id_{\eta_1>\eta_2}
\end{multline}
as follows.

\begin{corollary}\label{cor:BK}
For $\alpha=-1/2$, let us choose
\begin{equation}\label{kernelchangeofvariable}
s=\frac{\eta_2}{\sqrt2},\qquad t=\frac{\eta_1}{\sqrt2},\qquad x=\frac{\sigma_2^2}{4\sqrt2},\qquad y=\frac{\sigma_1^2}{4\sqrt2}.
\end{equation}
Then the hard edge Pearcey process with the above change of the space and time variables is the same as the Borodin--Kuan process, that is, their kernels satisfy the relation
\begin{equation}
L^{-1/2}\left(\frac{\eta_2}{\sqrt2},\frac{\sigma_2^2}{4\sqrt2};\frac{\eta_1}{\sqrt2},\frac{\sigma_1^2}{4\sqrt2}\right)\frac{\d y}{\d\sigma_1}=K^{\BK}(\sigma_1,\eta_1;\sigma_2,\eta_2).
\end{equation}
\end{corollary}

\begin{remark}
We assume above that the non-intersecting paths start at $a=0$ and end at $b=2qn$. Instead we could let them start at $a=2qn$ and end at $b=0$.
This leads to an alternate hard-edge Pearcey process that we could denote by the superscript $^\alt$. The corresponding kernels are related by
\begin{equation}\label{duality:Pearcey}
L^{\alpha,\alt}(s,x;t,y) = \left(\frac yx\right)^{\al} L^\alpha(t,y;s,x).
\end{equation}
Note that the prefactor $(y/x)^\alpha$ is just a conjugation of the kernel, hence the two kernels give rise to essentially the same correlation functions.
Note that this prefactor also appears in the right-hand side of \eqref{KMW:reduction}.
\end{remark}

\begin{remark}
The equivalence of the two formulas \eqref{Lalphadef} and \eqref{Lalpha:KMW:step1} for the kernel $L^\alpha$ can be seen directly as follows.
In both formulations, it can be shown that the derivative of the kernel with respect to the parameter $\si$ factorizes as
\begin{equation}\label{rank1:Pearcey}
\frac\pa{\pa\si}L^\alpha(s,x;t,y) = \left(\frac yx\right)^\al f(s,x)g(t,y)
\end{equation}
where the functions $f,g$ solve the third-order ODE's
\begin{align}
xf'''+(2-\al)f''-(s+\si) f'+f&=0,\label{f:diff}\\
xg'''+(2+\al)g''-(s+\si) g'-g&=0\label{g:diff}
\end{align}
where the prime denotes the derivative with respect to the space variable $x$ or $y$.
By using this fact and also some extra information to determine the relevant solutions to the ODE's,
one can prove that the functions $f,g$ in \eqref{rank1:Pearcey} are the same for both formulations of the kernel $L^\al$,
and hence the kernels \eqref{Lalphadef} and \eqref{Lalpha:KMW:step1} themselves are equal as well.
We leave the details to the interested reader.

We note that if $\al=0$ the functions $f,g$ in \eqref{rank1:Pearcey}  also satisfy the PDE's
\[\frac\pa{\pa s}f+\frac\pa{\pa x}f+x\frac{\pa^2}{\pa x^2}f=0,\qquad
\frac\pa{\pa t}g-\frac\pa{\pa y}g-y\frac{\pa^2}{\pa y^2}g=0.\]
\end{remark}

\section{Derivation of the kernel $K_n$}\label{s:derivation_finite}

In this section, we prove the different formulations of the kernel $K_n$ stated in Theorems~\ref{theorem:toeplitz}, \ref{theorem:inner:product} and \ref{theorem:kernel:n}.

\subsection{Proof of Theorem~\ref{theorem:toeplitz}}\label{subsection:proof:theorem:Toeplitz}

The proof of Theorem~\ref{theorem:toeplitz} consists of two steps.
First we express the quantity $C_n(\xi,\eta)$ as a Christoffel--Darboux kernel, then we rewrite it as the ratio of two Toeplitz determinants.

\paragraph{Step 1.}
Let us take the confluent limit of the starting and ending points $a_j\to a> 0$ and
$b_j\to b> 0$. The expression
\eqref{extendedkernel:start} reduces to
\begin{equation}\label{extendedkernel:conf}
K_n(s,x;t,y) = -p_{t-s}(x,y) \id_{t>s}+ \sum_{j,k=0}^{n-1} \left(\frac{\pa^{k}}{\pa b^{k}}p_{1-s}(x,b)\right)(A^{-1})_{k,j} \left(\frac{\pa^{j}}{\pa a^{j}}p_t(a,y)\right)
\end{equation}
where
\begin{equation}\label{extendedkernel:Amatrix:conf}
A = \left(\frac{\pa^{j}}{\pa a^{j}}\frac{\pa^{k}}{\pa b^{k}}p_1(a,b)\right)_{j,k=0}^{n-1}.
\end{equation}

Since we assumed that $\alpha$ is an integer, the integral representation for the modified Bessel function
\begin{equation}\label{intrepr:unitcircle:1}
I_{\alpha}(z) = \frac{1}{2\pi \ii} \int_{S_1} u^{\alpha-1} \exp\left(\frac{u+u^{-1}}{2}z\right)\ud u
\end{equation}
holds where $S_1$ denotes the unit circle.
Therefore, the transition probability \eqref{transitionprob:sqB} can be written as
\[p_t(x,y) = \int_{S_1} w(u;x,y,t) \frac{\ud u}{u}\]
with the weight function \eqref{weight:wxyt} on the unit circle $|u|=1$.
We can then evaluate the quantities in \eqref{extendedkernel:conf}--\eqref{extendedkernel:Amatrix:conf}:
\begin{align}
\frac{\pa^{j}}{\pa a^{j}} p_t(a,y) &= \int_{S_1} \left(\frac{u-1}{2t}\right)^{j} w(u;a,y,t) \frac{\ud u}{u},\label{intrepr:unitcircle:diff1}\\
\frac{\pa^{k}}{\pa b^{k}} p_{1-s}(x,b) &= \int_{S_1} \left(\frac{u^{-1}-1}{2(1-s)}\right)^{k} w(u;x,b,1-s) \frac{\ud u}{u},\label{intrepr:unitcircle:diff2}
\end{align}
and
\begin{equation}\label{intrepr:unitcircle:diff3}
\frac{\pa^{j}}{\pa a^{j}}\frac{\pa^{k}}{\pa b^{k}} p_1(a,b)
= \int_{S_1}\left(\frac{u-1}{2}\right)^{j}\left(\frac{u^{-1}-1}{2}\right)^{k}w(u;a,b,1) \frac{\ud u}{u}.
\end{equation}
Inserting these expressions in the kernel \eqref{extendedkernel:conf} and using the linearity of matrix multiplication, we find
\begin{multline}\label{extendedkernel:conf:2}
K_n(s,x;t,y) = -p_{t-s}(x,y) \id_{t>s} \\
+ \sum_{j,k=0}^{n-1} \left(\int_{S_1} Q_{k}\left(\frac{u^{-1}-1}{1-s}\right)w(u;x,b,1-s)\frac{\ud u}{u}\right)\\
\times (A^{-1})_{k,j} \left(\int_{S_1} P_{j}\left(\frac{u-1}{t}\right)w(u;a,y,t)\frac{\ud u}{u}\right)
\end{multline}
with
\begin{equation}\label{extendedkernel:Amatrix:conf:2}
A = \left(\int_{S_1} P_{j}(u-1)Q_{k}(u^{-1}-1)w(u;a,b,1)\frac{\ud u}{u}\right)_{j,k=0}^{n-1}
\end{equation}
where $P_{j}$ and $Q_{k}$ are arbitrary polynomials of degree $j$ and $k$ respectively.
Note that the factors $2$ in the denominators in \eqref{intrepr:unitcircle:diff1}--\eqref{intrepr:unitcircle:diff3} cancel each other.
So we obtain the formula \eqref{extendedkernel:toeplitz}, but with $C_{n}(\xi,\eta)$ there replaced by the Christoffel--Darboux kernel
\begin{equation}\label{CD:kernel:0}
C_{n}(\xi,\eta) := \sum_{j,k=0}^{n-1} P_j(\xi^{-1}-1) Q_k(\eta-1) (A^{-1})_{k,j}
\end{equation}
with again the notation \eqref{xi:eta} and \eqref{extendedkernel:Amatrix:conf:2}.
By taking $P_j(x)=(x+1)^{j}$ and $Q_k(y)=(y+1)^k$, this reduces to
\begin{equation}\label{CD:kernel:00}
C_{n}(\xi,\eta) =  \sum_{j,k=0}^{n-1} \xi^{-j} \eta^k (A^{-1})_{k,j}
\end{equation}
where
\begin{equation}\label{moment:matrix:A}
A = \left(\int_{S_1} z^{j-k} w(z)\frac{\ud z}{z}\right)_{j,k=0}^{n-1}.
\end{equation}
Note that $A$ is a Toeplitz matrix.

\paragraph{Step 2.}
Next we express the Christoffel--Darboux kernel $C_{n}(\xi,\eta)$ in \eqref{CD:kernel:00} as a ratio of two Toeplitz determinants.
We discuss this in a broader context.
Let $w(z)$ be an arbitrary weight function on the unit circle $S_1$.
Define the moments
\begin{equation}\label{moment:mjk}
m_{j,k} = \int_{S_1} z^{j-k}w(z)\frac{\ud z}{z}
\end{equation}
for $j,k\in\zet_{\geq 0}$.
Obviously the moments $m_{j,k}$ depend only on the difference $j-k$.
So they satisfy the Toeplitz property
\[m_{j+1,k+1}=m_{j,k},\]
but we will not use this for the moment.
Let $A=\left(m_{j,k}\right)_{j,k=0}^{n-1}$ be the moment matrix, i.e., the Toeplitz matrix \eqref{moment:matrix:A}.

\begin{lemma}
With the setting of the last paragraph, consider the Christoffel--Darboux kernel \eqref{CD:kernel:00}.
It can be written as a ratio of two Toeplitz determinants
\begin{equation}
C_n(\xi,\eta) = \frac{\eta^{n-1}}{\xi^{n-1}} \begin{vmatrix}
\wtil m_{0,0} & \wtil m_{0,1} & \ldots & \wtil m_{0,n-2} \\
\wtil m_{1,0} & \wtil m_{1,1} & \ldots & \wtil m_{1,n-2} \\
\vdots & \vdots & & \vdots \\
\wtil m_{n-2,0} & \wtil m_{n-2,1} & \ldots & \wtil m_{n-2,n-2}
\end{vmatrix}/\begin{vmatrix}
m_{0,0} & m_{0,1} & \ldots & m_{0,n-1} \\
m_{1,0} & m_{1,1} & \ldots & m_{1,n-1} \\
\vdots & \vdots & & \vdots \\
m_{n-1,0} & m_{n-1,1} & \ldots & m_{n-1,n-1}
\end{vmatrix}
\end{equation}
with the moments $m_{j,k}$ given in \eqref{moment:mjk} and
\[ \wtil m_{j,k} = m_{j,k}-\xi m_{j+1,k}-\eta^{-1}m_{j,k+1}+\xi\eta^{-1}m_{j+1,k+1}\]
which also satisfy the Toeplitz property $\wtil m_{j+1,k+1}=\wtil m_{j,k}$.
\end{lemma}

\begin{proof}
The definition of the Christoffel--Darboux kernel \eqref{CD:kernel:00} is equivalent with
\begin{equation}\label{CD:kernel:product}
C_n(\xi,\eta) = \begin{pmatrix}1 & \eta & \ldots & \eta^{n-1}\end{pmatrix}
\begin{pmatrix}
m_{0,0} & m_{0,1} & \ldots & m_{0,n-1}  \\
m_{1,0} & m_{1,1} & \ldots & m_{1,n-1}  \\
\vdots & \vdots & & \vdots \\
m_{n-1,0} & m_{n-1,1} & \ldots & m_{n-1,n-1}  \\
\end{pmatrix}^{-1}\begin{pmatrix}1 \\ \xi^{-1} \\ \vdots \\ \xi^{-(n-1)}\end{pmatrix}
\end{equation}
which in turn yields the determinantal formula
\begin{multline}\label{CD:moment:determinant}
C_n(\xi,\eta) = -\begin{vmatrix}
m_{0,0} & m_{0,1} & \ldots & m_{0,n-1} & 1 \\
m_{1,0} & m_{1,1} & \ldots & m_{1,n-1} & \xi^{-1} \\
\vdots & \vdots & & \vdots & \vdots \\
m_{n-1,0} & m_{n-1,1} & \ldots & m_{n-1,n-1} & \xi^{-(n-1)} \\
1 & \eta & \ldots & \eta^{n-1} & 0
\end{vmatrix}\\
\times\begin{vmatrix}
m_{0,0} & m_{0,1} & \ldots & m_{0,n-1} \\
m_{1,0} & m_{1,1} & \ldots & m_{1,n-1} \\
\vdots & \vdots & & \vdots \\
m_{n-1,0} & m_{n-1,1} & \ldots & m_{n-1,n-1}
\end{vmatrix}^{-1}.
\end{multline}
The denominator of \eqref{CD:moment:determinant} is a Toeplitz determinant.
We want to achieve the same for the determinant in the numerator.
To this end we apply a series of elementary row and column operations.
From the $j$th row we subtract $\xi$ times the $(j+1)$st row for $j=0,\ldots,n-2$, and from the $k$th column we subtract $\eta^{-1}$ times the $(k+1)$st column for $k=0,\ldots,n-2$.
These operations kill all the entries in the last row and column except for the entries $\xi^{-(n-1)}$ and $\eta^{n-1}$.
Hence the determinant factorizes as $-\xi^{-(n-1)}\eta^{n-1}$ times the determinant obtained by skipping the last two rows and columns.
The lemma follows.
\end{proof}
Note that $\wtil m_{j,k}$ are expressed as moments
\[\wtil m_{j,k}=\int_{S_1} z^{j-k}\wtil w(z)\frac{\ud z}{z}\]
where the weight function $\wtil w(z)$ is given in \eqref{weight:wtil}.
This proves \eqref{Toeplitz:ratio}, which ends the proof of Theorem~\ref{theorem:toeplitz}.

\subsection{Proof of Theorem~\ref{theorem:inner:product}}
\label{subsection:proof:theorem:BO}

We prove Theorem~\ref{theorem:inner:product} in three steps.
First we apply a Borodin--Okounkov type formula for the Toeplitz determinants, then we rewrite the ratio of determinants,
finally we use the rank one structure to obtain the assertion.

\paragraph{Step 1.}
We express both Toeplitz determinants in \eqref{Toeplitz:ratio} via a Borodin--Okounkov type formula, in the formulation of B\"ottcher--Widom \cite{BW}.

Let $n$ be a fixed integer.
Recall the semi-infinite matrix $A$ in \eqref{A:def}, viewed as an operator on $l^2(\zet_{\geq 0})$, and similarly define $\wtil A$ by \eqref{Atil:def}.
Both operators $A$ and $\wtil A$ are a trace class perturbation of the identity operator on $l^2(\zet_{\geq 0})$
so their Fredholm determinants $\det A$, $\det \wtil A$ are well-defined.
Then we obtain the following factorizations.

\begin{lemma}\label{thm:Borodin:Okounkov}
With the above notations, the following is true.
\begin{enumerate}
\item
The Toeplitz determinant in the numerator of \eqref{Toeplitz:ratio} can be factorized as
\begin{equation}\label{BO:formula1}
\det\left(\int_{S_1} z^{j-k}\wtil w(z)\frac{\ud z}{z}\right)_{j,k=1}^{n-1} = G^{n+\al-1} \wtil E \wtil F_{n-1,\al}\det \wtil A
\end{equation}
where
\[G:=\frac 12\exp\left(-\frac{a+b}{2}\right),\qquad \wtil E := \frac{\eta}{\eta-\xi}\exp\left(\frac{ab}{4}-\frac{a\eta^{-1}}{2}-\frac{b\xi}{2}\right)\]
and $\wtil F_{n-1,\al}$ will be defined in \eqref{Fn:def}--\eqref{Fn:def:2}.

\item
The Toeplitz determinant in the denominator of \eqref{Toeplitz:ratio} can be factorized as
\begin{equation}
\det\left(\int_{S_1} z^{j-k} w(z)\frac{\ud z}{z}\right)_{j,k=1}^{n} = G^{n+\al} E F_{n,\al}\det A
\end{equation}
where
\[G:=\frac 12\exp\left(-\frac{a+b}{2}\right),\qquad E := \exp\left(\frac{ab}{4}\right)\]
and $F_{n,\al}$ will be defined in the proof.
\end{enumerate}
\end{lemma}

\begin{proof}
The second part follows from the first by taking the limit $\xi\to 0$, $\eta\to\infty$ and replacing $n$ by $n+1$.
So it suffices to prove the first part of the lemma.
It will be convenient for us to replace the weight function $\wtil w(z)$ in the left hand side of \eqref{BO:formula1} by $\wtil w(-z)$.
Note that this is equivalent to a conjugation of the Toeplitz matrix by the diagonal matrix $\diag(1,-1,1,-1,\dots)$,
which does not affect the value of the determinant.
Moreover, we multiply the weight function $\wtil w(-z)$ by the prefactor $(-1)^{\alpha}$;
this prefactor multiplies the Toeplitz determinant in \eqref{BO:formula1} by the factor $(-1)^{(n-1)\al}$, but this factor will cancel out later.

We use the setting in B\"ottcher--Widom \cite{BW}. With the notations in that paper, the weight function of the Toeplitz determinant is
\begin{equation}\label{BO:w}
a(z) := \frac12(1+\xi z^{-1})(1+\eta^{-1}z)\exp\left(-\frac{z^{-1}+1}{2}a\right)\exp\left(-\frac{z+1}{2}b\right)
\end{equation}
on the unit circle $|z|=1$.
Note the changes of signs compared to \eqref{weight:wxyt}--\eqref{weight:wtil} by the replacement of $\wt w(z)$ with $\wt w(-z)$,
and note that $z$ turned into $z^{-1}$ due to the conventions in \cite{BW}.
We have the Wiener--Hopf factorization $a(z)=a_+(z)a_-(z)$ with
\begin{align}
a_+(z) &= \frac 12(1+\eta^{-1}z)\exp\left(-\frac{z+1}{2}b\right)\exp\left(-\frac{1}{2}a\right),\\
a_-(z) &= (1+\xi z^{-1})\exp\left(-\frac{z^{-1}}{2}a\right)
\end{align}
where $a_{\pm}(z)$ are analytic inside and outside the unit circle respectively and we normalized them as in B\"ottcher--Widom \cite[Eq.~(1)]{BW}.
Following \cite[Page~2]{BW}, set
\begin{align}
b(z) &= a_-(z)a_+^{-1}(z),\\
\wtil c(z) &= a_+(z^{-1})a_-^{-1}(z^{-1}) = b^{-1}(z^{-1}),
\end{align}
hence
\begin{equation}\label{BO:b}
b(z) = 2(1+\xi z^{-1})(1+\eta^{-1}z)^{-1}\exp\left(\frac{z+1}{2}b\right)\exp\left(\frac{-z^{-1}+1}{2}a\right).
\end{equation}
Let $H(b)$, $H(\wtil c)$ be the semi-infinite Hankel matrices associated to the symbols $b(z)$ and $\wtil c(z)$ respectively.
The row and column indices of these matrices are labelled by $\zet_{>0}$.

From \cite[Theorems~1.1 and 1.3]{BW} (with $n$ replaced by $n-1$ and with $\kappa=\alpha$),
we obtain the required formula \eqref{BO:formula1} for appropriately defined $G,\wtil E,\wtil F_{n-1,\al}$ (see below),
but with the final factor $\det\wtil A$ in \eqref{BO:formula1} replaced by the Fredholm determinant
\begin{equation}\label{Fred1} 
\det(\id-Q_{n+\al-1}H(b)H(\wtil c)Q_{n+\al-1})
\end{equation}
where $Q_k$ denotes the projection matrix $Q_k:=\diag(0_k,\id_\infty)$ with $0_k$ being the sequence of $0$s of length $k$.
(Note that the prefactor $(-1)^{(n-1)\alpha}$ in \cite[Theorem~1.3]{BW} cancels out with the one at the end of the first paragraph of the present proof, see above.) 
We calculate the product of Hankel matrices $H(b)H(\wtil c)$ in \eqref{Fred1}.
Its $(k,l)$th entry, $k,l\in\zet_{>0}$, is given by $1/(2\pi \ii)^2$ times
\begin{align*}
&\sum_{m=0}^{\infty} \left(\int_{S_1} w^{-k-m} b(w)\frac{\ud w}{w}\right)\left(\int_{S_1} z^{-l-m} \wtil c(z)\frac{\ud z}{z}\right)\\
&\qquad=\sum_{m=0}^{\infty} \left(\int_{S_1} w^{-k-m} b(w)\frac{\ud w}{w}\right)\left(\int_{S_1} z^{l+m} b^{-1}(z)\frac{\ud z}{z}\right)\\
&\qquad=\int_{S_{\rho^{-1}}}\frac{\ud z}{z}\int_{S_{\rho}} \frac{\ud w}{w} z^l w^{-k}\left(\sum_{m=0}^{\infty}z^{m} w^{-m}\right)b^{-1}(z)b(w)\\
&\qquad=\int_{S_{\rho^{-1}}}\frac{\ud z}{z}\int_{S_{\rho}} \frac{\ud w}{w} z^l w^{-k} (1-z/w)^{-1}b^{-1}(z)b(w).
\end{align*}
Hence the $(k,l)$th entry of $\id-H(b)H(\wtil c)$ is given by
\[\delta_{k-l}-\frac{1}{(2\pi \ii)^2}\int_{S_{\rho^{-1}}}\ud z\int_{S_{\rho}}\ud w\frac{z^{l-1}}{w^{k}} \frac{1}{w-z} b^{-1}(z)b(w)\]
which corresponds with our definition of $\wtil A$ in \eqref{Atil:def} up to a shift of the indices $k,l$ by $n+\al$.
Therefore the Fredholm determinant \eqref{Fred1} equals
\[ 
\det(\id-Q_{n+\al-1}H(b)H(\wtil c)Q_{n+\al-1}) = \det(\wtil A)_{l^2(\zet_{\geq 0})}\]
as desired.

Next, the expressions
for $G$ and $\wtil E$ are easily calculated from B\"ottcher--Widom \cite[p.~2--3]{BW} (keeping in mind formula \eqref{BO:w} above):
\[G = \exp(\log a(z))_0 = \frac 12\exp\left(-(a+b)/2\right)\]
and
\begin{align*}
\wtil E &=\exp\sum_{k=1}^{\infty} k(\log a(z))_k (\log a(z))_{-k} \\
&=\exp\sum_{k=1}^{\infty} k\left(-\frac b2\delta_{k-1}-\frac{(-\eta)^{-k}}{k}\right)\left(-\frac a2\delta_{k-1}-\frac{(-\xi)^{k}}{k}\right)\\
&=\exp\left(\frac{ab}{4}-\frac{a\eta^{-1}}{2}-\frac{b\xi}{2}\right)\exp\sum_{k=1}^{\infty}\frac{\xi^{k}\eta^{-k}}{k}\\
&=(1-\xi\eta^{-1})^{-1}\exp\left(\frac{ab}{4}-\frac{a\eta^{-1}}{2}-\frac{b\xi}{2}\right)
\end{align*}
as desired. Finally, we discuss the quantity $\wtil F_{n-1,\al}$.
From \cite[Eq.~(4)]{BW}, we have
\begin{equation}\label{Fn:def}
\wtil F_{n,\al} = \det(\FF\EE^{-1}\GG)
\end{equation}
where
\begin{equation}\label{Fn:def:2}
\EE = \id-H(b)Q_{n}H(\wtil c)Q_\al,\quad \FF=\begin{pmatrix} \id_\al & 0\end{pmatrix},\quad \GG=T(z^{-n}b)\begin{pmatrix} \id_\al \\ 0\end{pmatrix}
\end{equation}
where again $Q_n$ denotes the projection matrix $Q_n:=\diag(0_n,\id_\infty)$ with $0_n$ being the sequence of $0$s of length $n$.
The quantity $F_{n,\alpha}$ can be defined similarly.
It is also obtained as the limit of $\wt F_{n,\alpha}$ as $\xi\to0$ and $\eta\to\infty$.
\end{proof}

From Lemma~\ref{thm:Borodin:Okounkov} we obtain

\begin{corollary}\label{theorem:BO}
The ratio of Toeplitz determinants in \eqref{Toeplitz:ratio} can be written as
\begin{equation}
C_{n}(\xi,\eta) = 2 \frac{\xi^{-(n-1)}\eta^{n}}{\eta-\xi}\exp\left(\frac{1-\eta^{-1}}{2}a+\frac{1-\xi}{2}b\right) \frac{\det\wtil A}{\det A} \frac{\wtil F_{n-1,\al}}{F_{n,\al}}.
\end{equation}
\end{corollary}

\paragraph{Step 2.}
We express the last factor in Corollary~\ref{theorem:BO} via the determinants of certain block matrices, i.e.\ determinants on the space $L$.
Fix $n\in\zet_{>0}$. We define a block matrix operator
\[\begin{array}{rl} & \hspace{-4mm}\begin{array}{rr} \infty & \al\end{array} \\
\begin{array}{r} \infty\\ \al\end{array} & \hspace{-5mm}\begin{pmatrix} \wtil A & \wtil C \\\wtil B & \wtil D\end{pmatrix}\end{array}\]
of size $(\infty+\al)\times (\infty+\al)$.
Let us describe the four blocks.
Denote again by $S_\rho$ the circle with radius $\rho$ which we choose such that $\rho\in(1,\min\{|\xi|^{-1},|\eta|\})$.
\begin{itemize}
\item
$\wtil A$ is a semi-infinite matrix with $(k,l)$th entry
\begin{multline}\label{Atil:def}
\wtil A_{k,l}=\delta_{k-l}- \frac{1}{(2\pi \ii)^2}\int_{S_{\rho^{-1}}}\ud z\int_{S_{\rho}}\ud w\frac{z^{l+n+\al}}{w^{k+n+\al+1}} \frac{1}{w-z} \frac{w+\xi}{w+\eta}\frac{z+\eta}{z+\xi}\\
\times \exp\left(\frac{w-z}{2}b\right)\exp\left(\frac{z^{-1}-w^{-1}}{2}a\right)
\end{multline}
for $k,l\in\zet_{\geq 0}$.

\item
$\wtil B$ is a matrix of size $\al\times\infty$ with $(k,l)$th entry given by \eqref{Atil:def}, but with the $\delta_{k-l}$ term removed and the factor $\frac{z^{l+n+\al}}{w^{k+n+\al+1}}$ replaced by
$\frac{z^{l+n+\al}}{w^{k+n+1}}$ for $k=0,\ldots,\al-1$ and $l\in\zet_{\geq 0}$.

\item
$\wtil C$ is a Toeplitz matrix of size $\infty\times\al$ with $(k,l)$th entry
\begin{equation}\label{Ctil:def}
\wtil C_{k,l} = \frac{\eta}{2\pi \ii}\int_{S_1}\ud w\,w^{l-k-n-\al-1}\frac{w+\xi}{w+\eta}\exp\left(\frac{bw-aw^{-1}}{2}\right)
\end{equation}
for $k\in\zet_{\geq 0}$ and $l=0,\ldots,\al-1$.

\item
$\wtil D$ is a Toeplitz matrix of size $\al\times\al$ with $(k,l)$th entry given by \eqref{Ctil:def}, but with the factor $w^{l-k-n-\al-1}$ replaced by $w^{l-k-n-1}$ for $k,l=0,\ldots,\al-1$.
\end{itemize}

By taking the limit $\xi\to 0,\eta\to\infty$ and replacing $n$ by $n+1$ in the above formulas for $\wtil A,\wtil B,\wtil C,\wtil D$, we retrieve the formulas for the matrices $A,B,C,D$.

\begin{proposition}\label{prop:block:matrix}
With the above notations, the last factor in Corollary \ref{theorem:BO} can be written as
\begin{align}
\det(\wtil A)\,\wtil F_{n-1,\al} &= c\,\det\begin{pmatrix} \wtil A & \wtil C \\ \wtil B & \wtil D\end{pmatrix},\label{Ftil:blockdet}\\
\det(A)\, F_{n,\al} &= c\,\det\begin{pmatrix} A & C \\ B & D\end{pmatrix}\label{F:blockdet}
\end{align}
where the constant $c$ is the same in both formulas.
\end{proposition}

\begin{proof}
We only prove \eqref{Ftil:blockdet}, the other identity being similar.
We recall the general matrix identity
\begin{equation}\label{Schur:complement}
\det(\FF\EE^{-1}\GG-\HH) = -\det\begin{pmatrix} \EE & \GG \\ \FF & \HH \end{pmatrix}/ \det \EE
\end{equation}
for arbitrary matrices $\EE,\FF,\GG,\HH$ with compatible dimensions and $\EE$ invertible.
We will apply this identity with the matrices $\EE,\FF,\GG$ in \eqref{Fn:def:2} and $\HH=0$.
Then the left-hand side of \eqref{Schur:complement} reduces to $\wtil F_{n,\al}$ by virtue of \eqref{Fn:def}.
On the other hand, the right-hand side of \eqref{Schur:complement} can be reduced to\
\[ c\det\begin{pmatrix} \wtil A & \wtil C \\ \wtil B & \wtil D \end{pmatrix}/ \det \wtil A\]
with $c=(-1)^{\al+1}2 e^{(a+b)/2}$.
\end{proof}

\paragraph{Step 3.}
We relate the block matrices via a rank 1 formula. Let $T_\al$ be the upper triangular Toeplitz matrix
\begin{equation}\label{Talpha:inverse:def}
T_\al = \frac{1}{\eta}\begin{pmatrix}
1 & -(\xi-\eta) & \xi (\xi-\eta) & -\xi^{2} (\xi-\eta) & \ldots \\
0 & 1 & -(\xi-\eta) & \xi (\xi-\eta) & \ldots \\
0 & 0 & 1 & -(\xi-\eta)& \ldots \\
0&0&0&1&\ldots  \\
\vdots & \vdots & \vdots & \vdots& \ddots
\end{pmatrix}_{\al\times\al}.
\end{equation}

\begin{proposition}\label{prop:rank1:alpha}
The block matrices in Proposition~\ref{prop:block:matrix} satisfy the identity
\begin{equation}\label{rank1:alpha}
\begin{pmatrix} \wtil A & \wtil C \\ \wtil B & \wtil D\end{pmatrix} \begin{pmatrix} I_\infty & 0 \\ 0 & T_\al \end{pmatrix} - \begin{pmatrix} A & C \\ B & D\end{pmatrix}
= (\xi-\eta)\begin{pmatrix} \bf h \\ \bf\what h \end{pmatrix}\begin{pmatrix} \bf g & \boldsymbol{\xi} \end{pmatrix}
\end{equation}
where the right-hand side of \eqref{rank1:alpha} is a rank $1$ matrix and we used the notations \eqref{hk:def}--\eqref{xibold:def}.
\end{proposition}

\begin{proof}
Recall the matrices $\wtil C$ and $T_\al$ in \eqref{Ctil:def} and \eqref{Talpha:inverse:def} respectively.
The $(k,l)$th entry of $\wtil C T_\al$ is given by
\begin{align}
(\wtil C T_\al)_{k,l} &= \frac{1}{2\pi \ii}\int_{S_1}\ud w\,w^{-(k+n+\al+1)}\frac{w+\xi}{w+\eta}\exp\left(\frac{bw-aw^{-1}}{2}\right)\notag\\
&\qquad\qquad\times\left[w^l-(\xi -\eta)(w^{l-1}-w^{l-2}\xi+\ldots+(-\xi)^{l-1})\right]\notag\\
&= \frac{1}{2\pi \ii}\int_{S_1}\ud w\,\frac{w^{-(k+n+\al+1)}}{w+\eta}\exp\left(\frac{bw-aw^{-1}}{2}\right)\notag\\
&\qquad\qquad\times\left[(w+\xi)w^l-(\xi-\eta)(w^{l}-(-\xi)^l)\right]\label{proof:rank1:1}
\end{align}
for $k\in\zet_{\geq 0}$ and $l=0,\ldots,\al-1$.
On the other hand, the $(k,l)$th entry of $C$ is given by
\begin{equation}\label{proof:rank1:2}
C_{k,l} = \frac{1}{2\pi \ii}\int_{S_1}\ud w\,\frac{w^{-(k+n+\al+1)}}{w+\eta}\exp\left(\frac{bw-aw^{-1}}{2}\right)\left[(w+\eta)w^l\right].
\end{equation}
Subtracting \eqref{proof:rank1:2} from \eqref{proof:rank1:1} and cancelling terms, we find
\begin{equation*}
(\wtil C T_\al-C)_{k,l} = \frac{1}{2\pi \ii}\int_{S_1}\ud w\,\frac{w^{-(k+n+\al+1)}}{w+\eta}\exp\left(\frac{bw-aw^{-1}}{2}\right)\left[(\xi-\eta)(-\xi)^l\right].
\end{equation*}
This proves the matrix equality
\[\wtil C T_\al-C = (\xi-\eta)\mathbf{h}\boldsymbol{\xi}\]
on account of \eqref{hk:def} and \eqref{xibold:def}.
In a similar way, one shows that
\[\begin{pmatrix}\wtil C\\ \wtil D\end{pmatrix} T_\al-\begin{pmatrix}C\\ D\end{pmatrix} = (\xi-\eta)\begin{pmatrix}\bf h\\ \bf\what h\end{pmatrix}\boldsymbol{\xi}\]
which yields the second block column of \eqref{rank1:alpha}.
Finally, from the identity
\begin{equation}
\frac{w+\xi}{w+\eta}\frac{z+\eta}{z+\xi} \frac{1}{w-z} = \frac{1}{w-z} - \frac{\xi-\eta}{(w+\eta)(z+\xi)},
\end{equation}
one obtains the matrix equation
\[\begin{pmatrix}\wtil A\\ \wtil B\end{pmatrix} -\begin{pmatrix}A\\ B\end{pmatrix} = (\xi-\eta)\begin{pmatrix}\bf h\\ \bf\what h\end{pmatrix}\mathbf{g}\]
on account of \eqref{hk:def}--\eqref{gk:def}. This yields the first block column of \eqref{rank1:alpha}.
\end{proof}

\begin{corollary}\label{lemma:rankone:doublesum}  The last factor in Corollary \ref{theorem:BO} can be written as
\begin{equation}\label{rank:one:doublesum} \frac{\det\wtil A}{\det A} \frac{\wtil F_{n-1,\al}}{F_{n,\al}} =
\eta^{\al}\left(1+(\xi-\eta)\begin{pmatrix} \mathbf{g} & \boldsymbol{\xi} \end{pmatrix}\begin{pmatrix} A & C \\
B & D\end{pmatrix}^{-1}\begin{pmatrix} \bf h\\ \bf\what h
\end{pmatrix}\right).
\end{equation}
\end{corollary}

\begin{proof}
This follows from Propositions~\ref{prop:block:matrix} and \ref{prop:rank1:alpha}.
\end{proof}

Theorem~\ref{theorem:inner:product} is now an easy consequence of Corollary~\ref{theorem:BO} and Corollary~\ref{lemma:rankone:doublesum}. This ends the proof of Theorem~\ref{theorem:inner:product}.

\subsection{Proof of Theorem~\ref{theorem:kernel:n}}

In what follows, we show that
\begin{multline}\label{K_nkernelxieta:old}
K_n(s,x;t,y)=-p_{t-s}(x,y)\id_{t>s}+\frac1{2(2\pi\ii)^2}\int_{\Gamma_\eta'}\d\eta\int_{\Gamma_\xi'}\d\xi\frac{(t\xi^{-1}+1-t)^{\alpha-1}}{((1-s)\eta+s)^{\alpha+1}}\\
\frac{\eta^{n+\alpha}}{\xi^{n+1}}\exp\left(\frac{\eta-\xi}2b+\frac{\xi^{-1}-\eta^{-1}}2a-\frac{x(\eta-1)}{2(1-s)\eta+2s}-\frac{y(\xi^{-1}-1)}{2t\xi^{-1}+2(1-t)}\right)M(\xi,\eta)
\end{multline}
where $M(\xi,\eta)$ is defined in \eqref{M:xi:eta} where in the definitions of $\mathbf{g},\mathbf{h},\mathbf{\what h}$ in \eqref{hk:def}--\eqref{gk:def} we now take the circle radius $\rho$ such that
\begin{equation}\label{radius:rho:big}
\rho>\max\left\{\frac{2-t}{t},\frac{1+s}{1-s}\right\}\geq\max\{|\xi|^{-1},|\eta|\}>1
\end{equation}
and where the contours $\Gamma_\xi'$ and $\Gamma_\eta'$ are as follows.
Let $\Gamma_\xi'$ and $\Gamma_\eta'$ denote counterclockwise oriented closed contours such that $|\xi|<1$ and $|\eta|>1$ holds for all $\xi\in\Gamma_\xi'$ and $\eta\in\Gamma_\eta'$
and such that $\Gamma_\xi'$ encircles $0$ but it does not encircle the points $-s/(1-s)$ and $-t/(1-t)$ and $\Gamma_\eta'$ does surround the points $-s/(1-s)$ and $-t/(1-t)$.
See Figure~\ref{fig:primecontours} for illustration.

\begin{figure}
\begin{center}
\def\svgwidth{200pt}
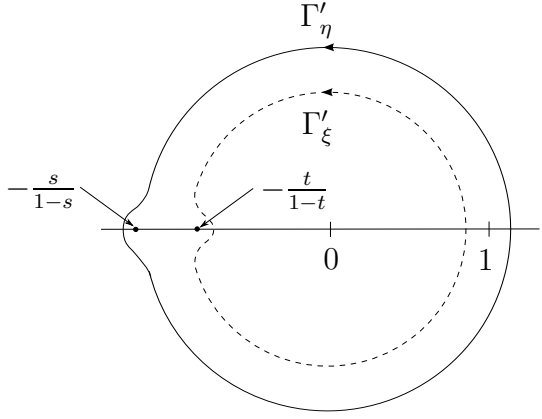
\end{center}
\caption{The integration contours $\Gamma_\xi'$ and $\Gamma_\eta'$.\label{fig:primecontours}}
\end{figure}

As a first step for proving \eqref{K_nkernelxieta:old}, we combine Theorems~\ref{theorem:toeplitz} and \ref{theorem:inner:product} along with \eqref{weight:wxyt} and \eqref{xi:eta} to get
\begin{multline}\label{extendedkernel:uv} K_n(s,x;t,y) = -p_{t-s}(x,y) \id_{t>s} + \frac{1}{(2\pi \ii)^2}\frac{1}{2t(1-s)}\int_{S_1}\int_{S_1} \frac{\ud u}{u}\frac{\ud v}{v}
u^{\alpha}v^{\alpha} \\ \times\exp\left(\frac{\xi^{-1}-\eta^{-1}}{2}a+\frac{\eta-\xi}{2}b+\frac{u^{-1}-1}{2t}y+\frac{v-1}{2(1-s)}x\right)\frac{\eta^{n+\al}}{\xi^{n-1}}M(\xi,\eta)
\end{multline}
where we recall that $\xi,\eta$ are defined in terms of $u,v$ by \eqref{xi:eta}.

Next we show that the circle radius $\rho$ in the definitions of $\mathbf{g},\mathbf{h},\mathbf{\what h}$ can be taken as in \eqref{radius:rho:big}.
Recall that we had $\rho\in(1,\min\{|\xi|^{-1},|\eta|\})$ in the original definitions of $\mathbf{g},\mathbf{h},\mathbf{\what h}$ in \eqref{hk:def}--\eqref{gk:def}. If we enlarge $\rho$ as in \eqref{radius:rho:big}, then a residue term will be picked up in each of these definitions.
These residue terms (ignoring the signs) are given by
\begin{align}
\label{hk:res}h_{k,\res} &= (-\eta)^{-(k+n+\al+1)}\exp\left(\frac{a\eta^{-1}-b\eta}{2}\right),\\
\label{whhk:res}\wh h_{k,\res} &= (-\eta)^{-(k+n+1)}\exp\left(\frac{a\eta^{-1}-b\eta}{2}\right),\\
\label{gl:res}g_{l,\res} &= (-\xi)^{l+n+\al}\exp\left(\frac{b\xi-a\xi^{-1}}{2}\right).
\end{align}
We claim that the contributions from these residue terms all drop out when inserting them in the double integral \eqref{extendedkernel:uv} via \eqref{M:xi:eta}.
Note that the dependence on $\xi$ and $\eta$ (that is, on $u$ and $v$) factorizes in the second term of \eqref{M:xi:eta} if we substitute it in \eqref{extendedkernel:uv}.
Hence in the second term, we can separate the $u$ and the $v$ integrals.
First, we show that the contribution coming from the residue $h_{k,\res}$ vanishes.
Note that there are essential cancellations when we multiply \eqref{hk:res} by the $v$ (and $\eta$) dependent factors in \eqref{extendedkernel:uv}.
The resulting contribution written as an $\eta$ integral is
\begin{equation}
\int\d\eta\frac{1-s}{((1-s)\eta+s)^{\alpha+1}}\exp\left(-\frac{x(\eta-1)}{2(1-s)\eta+2s}\right)\eta^{-(k+1)}
\end{equation}
where the integration contour is the image of the unit circle under \eqref{xi:eta}, that is the circle of radius $1/(1-s)$ around $-s/(1-s)$.
Note that the integrand is analytic outside this circle (also at $\infty$), hence we can blow up the contour and conclude that the integral is $0$ by Cauchy's theorem.
The argument for $\wt h_{k,\res}$ is similar, for $g_{l,\res}$ after the change from $u$ to $\xi$ according to \eqref{xi:eta}, we can shrink the contour to $0$.

Next we change the integration variables $u,v$ in the double integral in \eqref{extendedkernel:uv} to $\xi,\eta$ as defined in \eqref{xi:eta}.
The integration contours also transform under \eqref{xi:eta},
i.e.\ the contour for $\xi$ becomes the circle of radius $1/(2-t)$ around $1-1/(2-t)$ and the contour for $\eta$ will be the circle of radius $1/(1-s)$ around $1-1/(1-s)$.
Note that for any $t\in(0,1)$, the point $-t/(1-t)$ is outside the $\xi$ contour and for any $s\in(0,1)$, the point $-s/(1-s)$ is inside the $\eta$ contour.
Since the $\xi$ integral has no singularity at $-s/(1-s)$ and the $\eta$ integral has no singularity at $-t/(1-t)$,
the contours can be deformed to $\Gamma_\xi'$ and $\Gamma_\eta'$ as shown in Figure~\ref{fig:primecontours} by Cauchy's theorem
and \eqref{extendedkernel:uv} reduces to \eqref{K_nkernelxieta:old}.

Now we are ready to derive Theorem~\ref{theorem:kernel:n} from \eqref{K_nkernelxieta:old}.
Note that the order of the $\xi$ and $\eta$ integration in the double integral \eqref{K_nkernelxieta:old} can be interchanged.
This is due to Fubini's theorem and the fact that the contours $\Gamma_\xi'$ and $\Gamma_\eta'$ are disjoint.
We deform the integration contours in formula \eqref{K_nkernelxieta:old} from $\Gamma_\xi'$ and $\Gamma_\eta'$ to the new contours $\Gamma_\xi$ and $\Gamma_\eta$ in the following steps.
First we enlarge $\Gamma_\eta'$ to $\Gamma_\eta''$ if necessary and we move the contour $\Gamma_\xi'$ to $\Gamma_\xi$ inside $\Gamma_\eta''$.
Then, for any fixed $\xi\in\Gamma_\xi$, we pick up a residue contribution when deforming $\Gamma_\eta''$ to $\Gamma_\eta$ due to the $1/(\eta-\xi)$ term in $M(\xi,\eta)$.
The residue term can be written as
\begin{equation}\label{residueterm}
\frac1{4\pi\ii}\int_{\Gamma_\xi}\d\xi
\frac{((1-t)\xi+t)^{\alpha-1}}{((1-s)\xi+s)^{\alpha+1}}\exp\left(-\frac{x(\xi-1)}{2(1-s)\xi+2s}-\frac{y(\xi^{-1}-1)}{2t\xi^{-1}+2(1-t)}\right).
\end{equation}
The integrand above is analytic away from $-s/(1-s)$ and $-t/(1-t)$, in particular inside $\Gamma_\xi$, hence it vanishes thanks to Cauchy's theorem.

\section{Asymptotic analysis for the hard edge tacnode}\label{s:analysis_tacnode}

In this section, we prove Theorem~\ref{thm:hardedgetacnode}.
The proof of the first part follows the lines of the classical method of steep descent analysis.
The leading terms in the exponent of the kernel $K_n$ can be expressed by the function
\begin{equation}\label{deff0}
f_0(z)=\ln(-z)+\frac{z-z^{-1}}2=\frac16(z+1)^3+\Or((z+1)^4).
\end{equation}
We define the following contours with counterclockwise orientation.
If $r<1$ and close enough to $1$, then let
\begin{equation}\label{defGammar_small}
\Gamma_r=\{-1+e^{\ii\pi/3}t,0\le t\le t^*\}\cup\{-1+e^{-\ii\pi/3}t,0\le t\le t^*\}\cup\{re^{\ii\theta},-\theta^*\le\theta\le\theta^*\}
\end{equation}
where $t^*=(1-\sqrt{4r^2-3})/2$ and $\theta^*=\pi-\arcsin(\sqrt3(1-\sqrt{4r^2-3})/4r)$.
If $r>1$ and close to $1$, then let
\begin{equation}\label{defGammar_big}
\Gamma_r=\{-1+e^{\ii2\pi/3}t,0\le t\le t^*\}\cup\{-1+e^{-\ii2\pi/3}t,0\le t\le t^*\}\cup\{re^{\ii\theta},-\theta^*\le\theta\le\theta^*\}
\end{equation}
where $t^*=(\sqrt{4r^2-3}-1)/2$ and $\theta^*=\pi-\arcsin(\sqrt3(\sqrt{4r^2-3}-1)/4r)$.
The next lemma is about the steep descent properties of these paths.

\begin{lemma}\label{lemma:steep}
\begin{enumerate}
\item There is an $r_1<1$ close enough to $1$ for which $\Gamma_{r_1}$ defines a steep descent contour for the function $\Re(f_0)$ with respect to the critical point at $-1$.
\item There is an $r_2>1$ close enough to $1$ for which $\Gamma_{r_2}$ defines a steep descent contour for the function $-\Re(f_0)$ with respect to the critical point at $-1$.
\end{enumerate}
\end{lemma}

\begin{proof}
To show the steep descent property, we consider the following derivatives.
In the neighborhood of $-1$, we have
\[\frac{\d}{\d t}\Re f_0\left(-1+e^{\pm\ii\pi/3}t\right)=\frac{-t^2(2-2t-t^2)}{(2-2t-t^2)^2+3(2t-t^2)^2}\le0\]
if $0\le t\le\sqrt3-1$.
Furthermore,
\[-\frac{\d}{\d t}\Re f_0\left(-1+e^{\pm2\ii\pi/3}t\right)=\frac{-t^2(2+2t-t^2)}{(2+2t-t^2)^2+3(2t+t^2)^2}\le0\]
if $0\le t\le\sqrt3+1$.

On the circular part of the contour,
\begin{equation}\label{circularder}
\frac{\d}{\d\theta}\Re f_0\left(re^{\ii\theta}\right)=\frac{r^{-1}-r}2\sin\theta.
\end{equation}
If $r<1$, \eqref{circularder} is negative for $\theta\in[-\pi,0]$ and positive for $\theta\in[0,\pi]$,
whereas for $r>1$, the sign of \eqref{circularder} is opposite.
This proves the required property of the paths.
\end{proof}

\begin{proof}[Proof of Theorem~\ref{thm:hardedgetacnode}]
We use the method of steep descent analysis.
If we substitute \eqref{defa}--\eqref{defb} and the rescaled space and time variables \eqref{timespacescaling} into the double integral in \eqref{K_nkernelxieta},
one sees that the leading term in the exponent is $2N(f_0(q^{-1}\eta)-f_0(q^{-1}\xi))$.
We also substitute \eqref{M:xi:eta} for $M(\xi,\eta)$.
Next we specify the initial contours $\Gamma_\eta$ and $\Gamma_\xi$ to be the steep descent contours $\Gamma_{\eta,\text{steep}}$ and $\Gamma_{\xi,\text{steep}}$.
$\Gamma_{\xi,\text{steep}}$ coincides with $q\Gamma_{r_2}$ except for an $N^{-1/3}$ neighbourhood of $-q$ and there it crosses the real axis to the right of $-s/(1-s)$ and $-t/(1-t)$.
$\Gamma_{\eta,\text{steep}}$ consists of two parts: the first one coincides with $q\Gamma_{r_1}$ except for an $N^{-1/3}$ neighbourhood of $-q$ and it remains inside $\Gamma_{\xi,\text{steep}}$.
The other part is a small (order $N^{-1/3}$) circle around the singularity at $\eta=-s/(1-s)$ which is also in the $N^{-1/3}$ neighbourhood of $-q$ due to the scaling \eqref{timespacescaling}.
See Figure~\ref{fig:xietacontours}.

\begin{figure}
\begin{center}
\def\svgwidth{200pt}
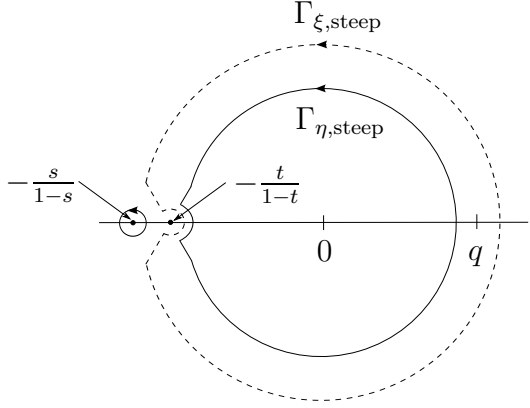
\end{center}
\caption{The contours $\Gamma_{\eta,\text{steep}}$ and $\Gamma_{\xi,\text{steep}}$ provide a possible steep descent choice of the integration contours $\Gamma_\eta$ and $\Gamma_\xi$.\label{fig:xietacontours}}
\end{figure}

Now we are ready for the steep descent analysis.
Consider the kernel \eqref{K_nkernelxieta} with the integration contours replaced by $\Gamma_{\eta,\text{steep}}$ and $\Gamma_{\xi,\text{steep}}$.
Let us choose a small but fixed $\delta>0$ and denote $\Gamma_\eta^\delta=\Gamma_{\eta,\text{steep}}\cap\{\eta:|\eta+q|\le\delta\}$ and
$\Gamma_\xi^\delta=\Gamma_{\xi,\text{steep}}\cap\{\xi:|\xi+q|\le\delta\}$.
For $\delta$ small enough, these contours do not contain the circular part of the original steep descent contour coming from the last part in \eqref{defGammar_small} and \eqref{defGammar_big},
but they completely contain the order $N^{-1/3}$ modifications if $N$ is large enough.
By neglecting the integral over $\Gamma_{\eta,\text{steep}}\setminus\Gamma_\eta^\delta$ and $\Gamma_{\xi,\text{steep}}\setminus\Gamma_\xi^\delta$, we make an error of order $\Or(\exp(-c\delta^3N))$.

In the $\delta$ neighborhood of $-q$, we can replace $f_0$ by its Taylor series.
Now we can do the change of variables given by
\begin{equation}
\eta=q\left(-1+\vv N^{-1/3}\right),\qquad\xi=q\left(-1+\uu N^{-1/3}\right).\label{xietascaling}
\end{equation}
The reader should not be confused by the fact that $\eta$ and $\xi$ were initially defined by $u$ and $v$ in \eqref{xi:eta}
where $u$ and $v$ played different role than in the rescaling \eqref{xietascaling}.
For the factor giving the main exponential term in \eqref{K_nkernelxieta}, we have
\begin{equation}\begin{aligned}
&(\eta\xi^{-1})^{2N}\exp\left(\frac{\eta-\xi}2 b+\frac{\xi^{-1}-\eta^{-1}}2 a\right)\\
&\qquad=(\eta\xi^{-1})^{\sigma N^{1/3}}\exp\left(2N\left(1-\frac\sigma{2N^{2/3}}\right)(f_0(q^{-1}\eta)-f_0(q^{-1}\xi))\right)\\
&\qquad=\exp\left(\frac{\vv^3}3-\frac{\uu^3}3-\sigma(\vv-\uu)+\Or(N^{-1/3})\right)
\end{aligned}\end{equation}
using also \eqref{timespacescaling}.
The other factors after substituting the scaling on the left-hand side of \eqref{kernelconv} become
\begin{multline}
\exp\left(-\frac{2rx(\eta-1)}{2(1-(\frac q{1+q}+rs))\eta+2(\frac q{1+q}+rs)}+\frac{2ry(\xi-1)}{2(1-(\frac q{1+q}+rt))\xi+2(\frac q{1+q}+rt)}\right)\\
=\exp\left(\frac x{\vv+s}-\frac y{\uu+t}\right)(1+\Or(N^{-1/3}))
\end{multline}
and
\begin{equation}
\frac{((1-(\frac q{1+q}+rt))\xi+(\frac q{1+q}+rt))^{\alpha-1}}{((1-(\frac q{1+q}+rs))\eta+(\frac q{1+q}+rs))^{\alpha+1}}
=\frac{(N^{-1/3}q(1+q)^{-1}(\uu+t))^{\alpha-1}}{(N^{-1/3}q(1+q)^{-1}(\vv+s))^{\alpha+1}}(1+\Or(N^{-1/3})).
\end{equation}
Finally, the change of variables \eqref{xietascaling} gives a factor $q^2N^{-2/3}$ and $(\eta/\xi)^{\alpha}$ goes to $1$.
Together with Proposition~\ref{prop:detdetlimit}, this shows that the double integral in \eqref{K_nkernelxieta} multiplied by the prefactor $2r$ is equal to
\begin{equation}
\frac1{(2\pi\ii)^2}\int_{N^{1/3}(q^{-1}\Gamma_\eta^\delta+1)}\d\vv\int_{N^{1/3}(q^{-1}\Gamma_\xi^\delta+1)}\d\uu
\frac{\exp\left(\frac{\vv^3}3+\frac x{\vv+s}\right)}{\exp\left(\frac{\uu^3}3+\frac y{\uu+t}\right)} \frac{(\uu+t)^{\alpha-1}}{(\vv+s)^{\alpha+1}}\Mm(\uu,\vv)
\end{equation}
up to an error of order $\Or(N^{-1/3})$ in the exponent coming from the Taylor approximation and the additive $\Or(\exp(-c\delta^3N))$ error.
By the identity $|e^x-1|\le|x|e^{|x|}$, the error in the exponent can be removed by an additive $\Or(N^{-1/3})$ error.
Finally, we extend the integral over $\vv$ to infinity in the $e^{\pm\ii\pi/3}$ directions and the integral over $\uu$ to infinity in the $e^{\pm\ii2\pi/3}$ directions
on the cost of an $\Or(\exp(-c\delta^3N))$ error again.
This proves the convergence towards the double integral on the right-hand side of \eqref{deflimkernel}.
The term $-p_{\frac{t-s}2}(x,y)\id_{t>s}$ appears by simple substitution.

If $x$ and $y$ are confined in a compact subset of $\R$, then all the above error estimates can be given uniformly in $x$ and $y$ yielding the stated uniform convergence.
This completes the proof of the first part of the theorem.
The second part follows directly from the first one since the uniform convergence of the kernel on bounded intervals
implies the convergence of Fredholm determinants on compact sets by the dominated convergence theorem.
\end{proof}

\begin{proposition}\label{prop:detdetlimit}
Under the scaling \eqref{xietascaling}, we have
\begin{equation}\label{rank:alphaplusone:det:tris}
\lim_{N\to\infty} N^{-1/3} q M(\xi,\eta)=e^{\sigma(\vv-\uu)}\Mm(\uu,\vv)
\end{equation}
as $N\to\infty$, where $M(\xi,\eta)$ is given by \eqref{M:xi:eta} and $\Mm(\uu,\vv)$ is defined in \eqref{defM}.
Furthermore,
\begin{equation}\label{Mbound}
|N^{-1/3}q M(\xi,\eta)|\le Ce^{c(|\vv|+|\uu|)}
\end{equation}
for some $c,C>0$.
\end{proposition}

\begin{proof}
First we will show that $qM(\xi,\eta)$ is independent of the
parameter $q$. To see this, define the diagonal matrices
\begin{align}
(U_1)_{k,k}&=q^{2N+k+\alpha+1}, & k&=0,1,2,\dots,\\
(U_2)_{k,k}&=q^{2N+k+1}, & k&=0,1,\dots,\alpha-1,\\
(V_1)_{l,l}&=q^{-(2N+l+\alpha+1)}, & l&=0,1,2,\dots,\\
(V_2)_{l,l}&=q^{-(l+1)}, & l&=0,1,\dots,\alpha-1.
\end{align}
In equation \eqref{M:xi:eta}, the matrix product on the right-hand side can be conjugated by the invertible matrices $U_1,U_2,V_1,V_2$.
This yields
\begin{equation}\label{rank:alphaplusone:det:further}
q M(\xi,\eta)=\frac q{\eta-\xi}-\begin{pmatrix} q{\bf g}V_1 & q\boldsymbol{\xi}V_2 \end{pmatrix}
\begin{pmatrix} U_1AV_1 & U_1CV_2 \\ U_2BV_1 & U_2DV_2\end{pmatrix}^{-1}\begin{pmatrix} U_1{\bf h} \\ U_2{\bf\what h} \end{pmatrix}.
\end{equation}
Let us show that the right-hand side of \eqref{rank:alphaplusone:det:further} is independent of $q$.
The idea is to perform a change of variables $w\to qw$ and $z\to qz$ in the defining integrals of $h_k$, $\wh h_k$, $g_l$,
$A_{k,l}$, $B_{k,l}$, $C_{k,l}$ and $D_{k,l}$, see \eqref{hk:def}--\eqref{gk:def} and \eqref{A:def}--\eqref{D:def}.
It is not hard to see that after this change of variables, $(U_1\mathbf{h})_k$ becomes $h_k$ with $\eta$ replaced by $\eta/q$
and with $a$ and $b$ defined for $q=1$, i.e.\ $a=b=2N-\sigma N^{1/3}$ in \eqref{hk:def}.
$(U_2\mathbf{\wh h})_k$ will be $\wh {h}_k$ with $\eta$ replaced $\eta/q$ and with $a$ and $b$ defined for $q=1$ in \eqref{hktil:def}.
Similarly, $(q{\bf g}V_1)_l$ is the same as $g_l$ with $\xi$ replaced by $\xi/q$ and with $a$ and $b$ defined for $q=1$ in \eqref{gk:def}
and $(q\boldsymbol{\xi}V_2)_l$ is equal to $(\boldsymbol{\xi})_l$ with $\xi$ replaced by $\xi/q$ in \eqref{xibold:def}.
By the same change of variables, one sees that the block matrix on the right-hand side of \eqref{rank:alphaplusone:det:further} is equal to
\[\begin{pmatrix} A & C \\ B & D \end{pmatrix}\]
with $a$ and $b$ defined for $q=1$ in \eqref{A:def}--\eqref{D:def}.
Note that by \eqref{xietascaling}, replacing $\eta$ and $\xi$ by $\eta/q$ and $\xi/q$ means the same scaling as taking $\eta$ and $\xi$ for $q=1$.

Hence we have shown that $qM(\xi,\eta)$ is independent of $q$.
Therefore we can and do assume that $q=1$ in what follows.
Denote by $K$ (resp.\ $L$) the lower triangular (resp.\ upper triangular) Pascal triangle type matrix
\begin{equation}
K = \begin{pmatrix}(-1)^{k-l}N^{k/3}\binom{k}{l}\end{pmatrix}_{k,l=0}^{\al-1},\qquad L = \begin{pmatrix}(-1)^{k}N^{l/3}\binom{l}{k}\end{pmatrix}_{k,l=0}^{\al-1}.
\end{equation}
In equation \eqref{M:xi:eta}, the matrix product on the right-hand side can be further conjugated by the invertible matrices $K,L$.
This yields (recall $q=1$)
\begin{equation}\label{rank:alphaplusone:det:KL}
M(\xi,\eta)=\frac 1{\eta-\xi}-\begin{pmatrix} {\bf g} & \boldsymbol{\xi}L \end{pmatrix}
\begin{pmatrix} A & CL \\ KB & KDL\end{pmatrix}^{-1}\begin{pmatrix} \bf h \\ K\bf{\what h} \end{pmatrix}.
\end{equation}
Now \eqref{rank:alphaplusone:det:tris} follows from Propositions~\ref{prop:functionlimits}, \ref{prop:operatorlimits}, \ref{prop:functopbounds} below and the dominated convergence theorem.
Note that the right-hand side of \eqref{rank:alphaplusone:det:KL} can be rewritten by using the block matrix inversion formula \eqref{schur:inversion}.
The matrix inverse in the limit is well-defined since the operator $\id-\Kaa_{\Ai,\sigma}$ is invertible for any $\sigma\in\R$,
and for $\si$ large enough the limiting Schur complement $\Ss_\sigma$ is also invertible by Proposition~\ref{prop:limSchur}.

Finally, the second assertion \eqref{Mbound} follows from the estimates given in Proposition~\ref{prop:functopbounds}.
\end{proof}

\begin{proposition}\label{prop:functionlimits}
Assume that $q=1$. Let $k,l\in\{0,1,\dots,\alpha-1\}$ be fixed and choose $x,y\ge0$.
We have the following limits for the functions in \eqref{rank:alphaplusone:det:KL}:
\begin{align}
h_{N^{1/3}x}&\to e^{\sigma\vv} h(x;\vv),\label{hconvergence}\\
(K\mathbf{\wh h})_k&\to e^{\sigma\vv} \frac{\partial^k}{\partial x^k}h(x;\vv)|_{x=0},\label{whhconvergence}\\
g_{N^{1/3}y}&\to -e^{-\sigma\uu} h(y;-\uu),\label{gconvergence}\\
(\boldsymbol\xi L)_l&\to \uu^l\label{xiconvergence}
\end{align}
as $N\to\infty$ pointwise.
\end{proposition}

\begin{proposition}\label{prop:operatorlimits}
Assume that $q=1$. Fix the integers $k,l\in\{0,1,\dots,\alpha-1\}$ and let $x,y\ge0$.
Then the operators in \eqref{rank:alphaplusone:det:KL} converge as follows:
\begin{align}
N^{1/3}A_{N^{1/3}x,N^{1/3}y}&\to \Aa_\sigma(x,y),\label{Alimit}\\
N^{1/3}(KB)_{k,N^{1/3}y}&\to \Bb_\sigma(k,y),\label{Blimit}\\
N^{1/3}(CL)_{N^{1/3}x,l}&\to \Cc_\sigma(x,l),\label{Climit}\\
N^{1/3}(KDL)_{k,l}&\to \Dd_\sigma(k,l)\label{Dlimit}
\end{align}
as $N\to\infty$ pointwise.
\end{proposition}

We have the following decay properties of the functions and operators.
\begin{proposition}\label{prop:functopbounds}
Assume that $q=1$. For any $k,l\in\{0,1,\dots,\alpha-1\}$ and any constant $d>0$, there are constants $c,C>0$ such that
\begin{align}
\left|h_{N^{1/3}x}\right|&\le Ce^{c|\vv|-dx},\\
\left|(K\mathbf{\wh h})_k\right|&\le Ce^{c|\vv|},\\
\left|g_{N^{1/3}y}\right|&\le Ce^{c|\uu|-dy},\\
\left|(\boldsymbol\xi L)_l\right|&\le Ce^{c|\uu|},\\
\left|N^{1/3}A_{N^{1/3}x,N^{1/3}y}\right|&\le Ce^{-d(x+y)},\\
\left|N^{1/3}(KB)_{k,N^{1/3}y}\right|&\le Ce^{-dy},\\
\left|N^{1/3}(CL)_{N^{1/3}x,l}\right|&\le Ce^{-dx},\\
\left|N^{1/3}(KDL)_{k,l}\right|&\le C
\end{align}
uniformly for $N>N_0$ and as $x,y>X_0$ for some thresholds $N_0$ and $X_0$.
\end{proposition}

The strategy for proving Propositions~\ref{prop:functionlimits} and \ref{prop:operatorlimits} is the following.
To keep the computations simpler, we first prove for $\sigma=0$, and we point out the changes for general $\sigma$.
We first rewrite our functions and operators using the Bessel function $J_n(t)$ defined by
\begin{equation}
J_n(t)=\frac1{2\pi\ii}\int_{S_1}\d z\,z^{-(n+1)}\exp\left(\frac t2\left(z-z^{-1}\right)\right).\label{defJ}
\end{equation}
It turns out that some of the operators and functions can be written even more conveniently using the function
\begin{equation}\label{defJk}
J^{(k)}(s,t)=t^{\frac{k+1}3}\sum_{p=0}^k \binom kp (-1)^{k-p} J_{2t+st^{1/3}+p}(2t)
\end{equation}
which is an approximation of the $k$th derivative of the Airy function as seen from Proposition~\ref{prop:Jconv}.
We remark that one can also rewrite $J^{(k)}(s,t)$ by its definition \eqref{defJ} and the binomial theorem as
\begin{equation}\label{Jdifference}\begin{aligned}
J^{(k)}(s,t)
&=t^{\frac{k+1}3}\sum_{p=0}^k \binom kp (-1)^{k-p} \frac1{2\pi\ii}\int_{S_1}\d z\, z^{-(2t+st^{1/3}+p+1)}\exp\left(t(z-z^{-1})\right)\\
&=\frac{t^{\frac{k+1}3}}{2\pi\ii}\int_{S_1}\d z\, \exp\left(2t\left(-\ln z+\frac{z-z^{-1}}2\right)-st^{1/3}\ln z\right)\frac{\left(z^{-1}-1\right)^k}z.
\end{aligned}\end{equation}
Then the limit statement of Proposition~\ref{prop:Jconv} and the bound of Proposition~\ref{prop:Jbound} imply the claimed convergence results
in Propositions~\ref{prop:functionlimits} and \ref{prop:operatorlimits} by dominated convergence.
The bounds in Proposition~\ref{prop:functopbounds} follow from the bounds given for the dominated convergence argument
in the proofs of Propositions~\ref{prop:functionlimits} and \ref{prop:operatorlimits}.

\begin{proof}[Proof of Proposition~\ref{prop:functionlimits}]
Suppose first that $\sigma=0$.
In the definition of $h_k$ \eqref{hk:def} we write
\begin{equation}\label{geoseries}
\frac1{w+\eta} = \sum_{j\ge0}\frac{(-\eta)^j}{w^{j+1}}
\end{equation}
and use \eqref{defa}--\eqref{defb} (with $q=1$ and $\sigma=0$) to obtain
\begin{equation}\label{h-hres}\begin{aligned}
h_k&= \sum_{j\ge0}(-\eta)^j\frac1{2\pi\ii}\int_{S_{\rho}}\d w\,w^{-(2N+k+j+\alpha+2)}\exp\left(N\left(w-w^{-1}\right)\right)\\
&=\sum_{j\ge0}\left(-\eta\right)^j J_{2N+k+j+\alpha+1}(2N)
\end{aligned}\end{equation}
on account of \eqref{defJ}.
If we substitute the scaling of $\eta$ from \eqref{xietascaling} (with $q=1$) and $k=N^{1/3}x$ and replace the summation by an integral for $\lambda=N^{-1/3}j$,
then the above expression up to $N^{-1/3}$ error is
\begin{multline}
\int_0^\infty\d\lambda\left(1-\vv N^{-1/3}\right)^{N^{1/3}\lambda}N^{1/3}J_{2N+N^{1/3}(x+\lambda)+\alpha+1}(2N)\\
=\int_0^\infty\d\lambda\left(1-\vv N^{-1/3}\right)^{N^{1/3}\lambda}J^{(0)}(x+\lambda+\Or(N^{-1/3}),N).
\end{multline}
The factor $(1-\vv N^{-1/3})^{N^{1/3}\lambda}$ converges to $e^{-\vv\lambda}$ and it is bounded by $e^{|\vv|\lambda}$ uniformly in $N$.
It follows from Propositions~\ref{prop:Jconv} and \ref{prop:Jbound} for $k=0$ that the factor $J^{(0)}(x+\lambda+\Or(N^{-1/3}),N)$ converges to $\Ai(x+\lambda)$ pointwise
and that its decay in $\lambda$ (and in $x$) is faster than any exponential, in particular, faster than $e^{-2|\vv|\lambda}$.
This shows that the dominated convergence theorem applies and it yields \eqref{hconvergence} for $\sigma=0$.

For general $\sigma$, we have $J_{2N+k+j+\alpha+1}(2N-N^{1/3}\sigma)$ instead of $J_{2N+k+j+\alpha+1}(2N)$ on the right-hand side of \eqref{h-hres}.
Then we use Proposition~\ref{prop:Jconv} as
\begin{equation}
J_{2N+N^{1/3}(x+\lambda)+\alpha+1}(2N-N^{-1/3}\sigma)\to\Ai(x+\lambda+\sigma).
\end{equation}
After the change of variable $\lambda+\sigma\to\lambda$ in the integral obtained as the limit of $h_{N^{1/3}x}$, one gets \eqref{hconvergence}.

Suppose again $\sigma=0$.
In the definition of $\wh h_k$ \eqref{hktil:def}, we write $1/(w+\eta)$ as a geometric series as in \eqref{geoseries} to get
\begin{equation}
\wh h_k=\sum_{j\ge0}(-\eta)^j J_{2N+j+k+1}(2N).
\end{equation}
By the definition of the matrix $K$ and by that of the function $J^{(k)}$, we have
\begin{equation}\label{whhcomputations}\begin{aligned}
(K\mathbf{\wh h})_k &=N^{-1/3}\sum_{j\ge0}(-\eta)^j J^{(k)}(N^{-1/3}(j+k+1),N)\\
&=\int_0^\infty\d\lambda\left(1-\vv N^{-1/3}\right)^{N^{1/3}\lambda}J^{(k)}(\lambda+\Or(N^{-1/3}),N)+\Or(N^{-1/3}).
\end{aligned}\end{equation}
By the same dominated convergence argument as above using Propositions~\ref{prop:Jconv} and \ref{prop:Jbound} for $k=0,1,\dots,\alpha-1$,
one can conclude the proof of \eqref{whhconvergence} for $\sigma=0$.
The general case follows as before.

The proof of \eqref{gconvergence} is similar.
We write $1/(z+\xi)=\sum_{j\ge0}(-z)^j\xi^{-(j+1)}$ as a geometric series. Next we substitute $z\to z^{-1}$.
This allows for expressing $g_l$ with Bessel functions as
\begin{equation}
g_l=-\sum_{j\ge0}(-\xi)^{-(j+1)}J_{2N+l+j+\alpha+1}(2N).
\end{equation}
The dominated convergence argument and the extension to general $\sigma$ is the same as for $h_k$. Finally, the proof of \eqref{xiconvergence} is straightforward.
\end{proof}

\begin{proof}[Proof of Proposition~\ref{prop:operatorlimits}]
Since the steps of this proof are rather similar to those of Proposition~\ref{prop:functionlimits}, we only give a sketch here, from which the reader can complete the details.
We also restrict ourselves to the $\sigma=0$ case.

By rewriting $1/(w-z)=\sum_{j\ge0}z^j/w^{j+1}$ as a geometric series,
substituting $z\to z^{-1}$ and using \eqref{defa}--\eqref{defb} (with $q=1$), we get
\begin{equation}
A_{k,l}=\delta_{k,l}-\sum_{j\ge0}J_{2N+l+j+\alpha+1}(2N)J_{2N+k+j+\alpha+1}(2N).
\end{equation}
For $k=N^{1/3}x$ and $l=N^{1/3}y$, an application of Propositions~\ref{prop:Jconv} and \ref{prop:Jbound} for $k=0$ and the dominated convergence theorem gives \eqref{Alimit}.

Similarly,
\begin{equation}
B_{k,l}=-\sum_{j\ge0}J_{2N+l+j+\alpha+1}(2N)J_{2N+k+j+1}(2N).
\end{equation}
As in \eqref{whhcomputations} by applying the matrix $K$, one gets
\[ N^{1/3}(KB)_{k,l}=-\sum_{j\ge0} J_{2N+l+j+\alpha+1}(2N) J^{(k)}(N^{-1/3}(j+1),N).\]
A dominated convergence argument as before yields \eqref{Blimit}.

It is easy to see that
\[C_{k,l}=J_{2N+k-l+\alpha}(2N).\]
If we apply the matrix $L$ from the right, the convergence \eqref{Climit} follows from the equation
\[N^{1/3}(CL)_{k,l}=J^{(l)}(N^{-1/3}(k+\alpha-l),N).\]

Similarly,
\begin{equation}\label{DBessel}
D_{k,l}=J_{2N+k-l}(2N).
\end{equation}
After applying the matrix $K$ from the left and the matrix $L$ from the right, one can rewrite \eqref{DBessel} as
\begin{align*}
N^{1/3}(KDL)_{k,l}&=N^{\frac{k+l+1}3}\sum_{p=0}^k\sum_{r=0}^l \binom kp \binom lr (-1)^{k+p+r}J_{2N+p-r}(2N)\\
&=N^{\frac{k+l+1}3}\sum_{s=0}^{k+l} \binom {k+l}s (-1)^{k+l+s}J_{2N-l+s}(2N)\\
&=J^{(k+l)}(-N^{-1/3}l,N).
\end{align*}
Then \eqref{Dlimit} follows.
This completes the proof of Proposition~\ref{prop:operatorlimits}.
\end{proof}

\begin{proof}[Proof of Proposition~\ref{prop:functopbounds}]
The assertions easily follow from the bounds given in the proofs of Propositions~\ref{prop:functionlimits} and \ref{prop:operatorlimits} for the dominated convergence argument,
which were obtained from Proposition~\ref{prop:Jbound} and from obvious estimates.
\end{proof}

\section{Asymptotic analysis for the hard edge Pearcey process}\label{s:analysis_Pearcey}

In this section, we first prove Proposition~\ref{prop:confluent:a0} which is the starting formula for the asymptotic analysis for the hard edge Pearcey process.
Then, we turn to the actual asymptotic analysis, i.e.\ we prove Theorem~\ref{thm:hardedgePearcey}.
We use the following lemma for the proof of Proposition~\ref{prop:confluent:a0}.

\begin{lemma}\label{lemma:kernel:Joh}
The kernel $K_n$ for $n$ non-intersecting squared Bessel paths in \eqref{extendedkernel:start} can be written as
\begin{equation}\label{kernel:Joh}
K_n(s,x;t,y) = -p_{t-s}(x,y) \id_{t>s}- \int_{-\infty}^0 p_{t-1}(w,y) \sum_{k=1}^n p_{1-s}(x,b_k)\frac{\det\left(p_1(a_i,b_j^{(k)})\right)_{i,j=1}^n}{\det\left(p_1(a_i,b_j)\right)_{i,j=1}^n}\ud w
\end{equation}
for any $x,y>0$ and $s,t\in (0,1)$ where we use the notation
\begin{equation}\label{bjk:def}
b_j^{(k)} = \left\{\begin{array}{ll} b_j, & \mbox{if $j\neq k$},\\ w, & \mbox{if $j= k$}.\end{array}\right.
\end{equation}
\end{lemma}

\begin{proof}
We follow Johansson \cite[(2.7)--(2.10)]{Joh11}.
Define the column vector
\begin{equation}
\pee = \begin{pmatrix} p_t(a_1,y) & \ldots & p_t(a_n,y) \end{pmatrix}^T.
\end{equation}
Applying Cramer's rule to \eqref{extendedkernel:start}, we find
\begin{equation}\label{Cramer1}
K_n(s,x;t,y) = -p_{t-s}(x,y) \id_{t>s}+ \sum_{k=1}^n p_{1-s}(x,b_k)\frac{\det(A|\pee)_k}{\det A}
\end{equation}
where $(A|\pee)_k$ denotes the matrix $A$ in \eqref{extendedkernel:Amatrix} with its $k$th column replaced by the vector $\pee$.
The following Markov type formula holds for the transition probabilities of the squared Bessel paths given in \eqref{transitionprob:sqB} for \lq negative time spans\rq:
\begin{equation}\label{Markov:sqB}
p_{t}(x,y) = -\int_{-\infty}^0 p_s(x,z) p_{t-s}(z,y)\ud z = -\int_{-\infty}^0 p_{t-s}(x,z) p_s(z,y)\ud z
\end{equation}
for any $x,y>0$ and $s>t>0$.
Here we define the branches of the powers in \eqref{transitionprob:sqB}--\eqref{BesselI:def}
by $(\sqrt{z})^{\alpha} \equiv z^{\alpha/2} \equiv e^{\alpha\pi \ii/2}|z|^{\alpha/2}$ and $(1/z)^{\alpha/2} \equiv e^{-\alpha\pi \ii/2}|z|^{-\alpha/2}$ when $z$ is negative.
The Markov type formula \eqref{Markov:sqB} is the analogue of \cite[(2.5)--(2.6)]{KT11}.
From \eqref{Markov:sqB}, we get
\begin{equation}\label{Cramer2}
\frac{\det(A|\pee)_k}{\det A} = -\int_{-\infty}^0 \frac{\det\left(p_1(a_i,b_j^{(k)})\right)_{i,j=1}^n}{\det\left(p_1(a_i,b_j)\right)_{i,j=1}^n}p_{t-1}(w,y)\ud w
\end{equation}
with the notation \eqref{bjk:def}.
By inserting \eqref{Cramer2} into \eqref{Cramer1} and interchanging the sum and the integral, we obtain \eqref{kernel:Joh}.
This proves the lemma.
\end{proof}

\begin{proof}[Proof of Proposition~\ref{prop:confluent:a0}]
Consider the ratio of determinants in the right hand side of \eqref{kernel:Joh}. From the definition of the transition probability and using the linearity of the determinant, we can write it as
\begin{equation}\label{ratio:determinants}
\frac{\det\left(p_1(a_i,b_j^{(k)})\right)_{i,j}}{\det\left(p_1(a_i,b_j)\right)_{i,j}} = \left(\frac{w}{b_k}\right)^{\alpha/2} \exp\left(\frac{b_k-w}{2} \right) \frac{\det\left(a_i^{-\alpha/2}
I_{\alpha}\left(\sqrt{a_i b_j^{(k)}}\right)\right)_{i,j}}{\det\left(a_i^{-\alpha/2}I_{\alpha}\left(\sqrt{a_i b_j}\right)\right)_{i,j}}
\end{equation}
where we also used \eqref{bjk:def}.
The factors $a_i^{-\alpha/2}$ in the above determinants depend only on the row index $i$, therefore they can be extracted from the determinants cancelling each other completely.
But it will be convenient to keep these factors inside the determinants.

When taking the confluent limit $a_i\to 0$ of \eqref{ratio:determinants}, certain derivatives appear.
This is discussed by Tracy--Widom \cite[pages~7--8]{TW3} in a similar context.
If we apply this idea to \eqref{ratio:determinants}, we obtain
\begin{equation}\label{confluent:limit:a0}
\lim_{a_i\to 0}\frac{\det\left(p_1(a_i,b_j^{(k)})\right)_{i,j}}{\det\left(p_1(a_i,b_j)\right)_{i,j}}
= \left(\frac{w}{b_k}\right)^{\alpha/2} \exp\left(\frac{b_k-w}{2} \right) \frac{\det\left(\frac{\partial^{i-1}}{\partial a^{i-1}}\left[a^{-\alpha/2}
I_{\alpha}\left(\sqrt{a b_j^{(k)}}\right)\right]_{a=0}\right)_{i,j}}
{\det\left(\frac{\partial^{i-1}}{\partial a^{i-1}}\left[ a^{-\alpha/2} I_{\alpha}\left(\sqrt{a b_j}\right)\right]_{a=0}\right)_{i,j}}
\end{equation}
where again $i,j$ run from $1$ to $n$.

We evaluate the derivatives.
From the series expansion  \eqref{BesselI:def} of the modified Bessel function, we find
\begin{equation}
\frac{\partial^{i-1}}{\partial a^{i-1}}\left[a^{-\alpha/2} I_{\alpha}(\sqrt{a b})\right]_{a=0}=\frac{1}{\Gamma(i+\alpha)}\left(\frac{b}{4}\right)^{i-1+\alpha/2}.
\end{equation}
Using this for both matrices on the right-hand side of \eqref{confluent:limit:a0} and cancelling again the common factors, we obtain
\begin{equation}\label{confluent:limit:a0:2}\begin{aligned}
\lim_{a_i\to 0} \frac{\det\left(p_1(a_i,b_j^{(k)})\right)_{i,j}}{\det\left(p_1(a_i,b_j)\right)_{i,j}}
&= \left(\frac{w}{b_k}\right)^{\alpha} \exp\left(\frac{b_k-w}{2} \right)\frac{\det\left((b_j^{(k)})^{i-1}\right)_{i,j}}{\det\left(b_j^{i-1}\right)_{i,j}}\\
&= \left(\frac{w}{b_k}\right)^{\alpha} \exp\left(\frac{b_k-w}{2} \right) \prod_{j\neq k} \frac{w-b_j}{b_k-b_j}\\
&= \left(\frac{w}{b_k}\right)^{\alpha} \exp\left(\frac{b_k-w}{2} \right) \Res\left(\frac{1}{w-z}\prod_{j=1}^n\frac{w-b_j}{z-b_j},z=b_k\right)
\end{aligned}\end{equation}
where the second line follows from the evaluation of the Vandermonde determinants and cancellation of the common factors recalling \eqref{bjk:def}.,
By inserting \eqref{confluent:limit:a0:2} into \eqref{kernel:Joh} and using the residue theorem, we obtain the double integral formula \eqref{doubleintegral:confluent:a0}.
This proves the proposition.
\end{proof}

Next we prove Theorem~\ref{thm:hardedgePearcey} which immediately follows from the following convergence statement for the kernels.

\begin{proposition}
Let $s,t,x,y\in\R$ be fixed numbers. It holds that
\begin{multline}
\lim_{n\to\infty}\frac1{2qn^{1/2}}K_n^{\Pea}\left(\frac1{1+q}+\frac q{(1+q)^2}sn^{-1/2},\frac1{2q}xn^{-1/2};\frac1{1+q}+\frac q{(1+q)^2}tn^{-1/2},\frac1{2q}yn^{-1/2}\right)\\
=L^\alpha(s,x;t,y)
\end{multline}
and the convergence is uniform for $x$ and $y$ from compact subsets of $\R$.
\end{proposition}

\begin{proof}
We use again the method of steep descent analysis.
First we substitute the scaled parameter values of the Pearcey scaling to \eqref{doubleintegral:confluent:a0} and use the definition of the transition kernel \eqref{transitionprob:sqB} to get
\begin{equation}\label{rescaledKPea}\begin{aligned}
\frac1{2qn^{1/2}}&K_n^{\Pea}\left(\frac1{1+q}+\frac q{(1+q)^2}sn^{-1/2},\frac1{2q}xn^{-1/2};\frac1{1+q}+\frac q{(1+q)^2}tn^{-1/2},\frac1{2q}yn^{-1/2}\right)\\
=&-p_{\frac{t-s}2}(x,y)\id_{t>s}\\
&-\left(\frac yx\right)^{\alpha/2}\frac1{4q\sqrt n\pi\ii}\int_\Gamma\d z\int_{-\infty}^0\d w\left(\frac wz\right)^{\alpha/2}
\exp\left(\frac{w+\frac y{2q\sqrt n}}{\frac{2q}{1+q}-\frac{2qt}{(1+q)^2\sqrt n}}-\frac{\frac x{2q\sqrt n}+z}{\frac{2q}{1+q}-\frac{2qs}{(1+q)^2\sqrt n}}\right)\\
&\quad\times I_\alpha\left(-2\sqrt{w\frac y{2q\sqrt n}}(1+\Or(n^{-1/2}))\right)I_\alpha\left(2\sqrt{\frac x{2q\sqrt n}z}(1+\Or(n^{-1/2}))\right)\\
&\quad\times\exp\left(\frac{z-w}2\right)\left(\frac{1-\frac w{2qn}}{1-\frac z{2qn}}\right)^n\frac1{z-w}+\Or(n^{-1/2}).
\end{aligned}\end{equation}
We perform the change of variables $w\to2qnw$ and $z\to2qnz$ in the double integral above.
Since $\Re(z-w)>0$, the following exponential integral representation holds:
\begin{equation}\label{intrepr}
\frac1{z-w}=\sqrt n\int_0^\infty\d\lambda\,e^{-\lambda\sqrt n(z-w)}.
\end{equation}
The advantage of this formula is that it separates the dependence of the integrand of \eqref{rescaledKPea} on the variables $w$ and $z$,
so the $n\to\infty$ limit of the $w$-integral and the $z$-integral can be taken separately, as we will proceed.
At the end, we will redo the operation \eqref{intrepr}.
In fact, one could perform the asymptotic analysis without using the identity \eqref{intrepr}, but then one would have even longer formulas in the proof.
Using \eqref{intrepr} in \eqref{rescaledKPea}, we arrive at
\begin{equation}\label{rescaledKPea2}\begin{aligned}
\frac1{2qn^{1/2}}&K_n^{\Pea}\left(\frac1{1+q}+\frac q{(1+q)^2}sn^{-1/2},\frac1{2q}xn^{-1/2};\frac1{1+q}+\frac q{(1+q)^2}tn^{-1/2},\frac1{2q}yn^{-1/2}\right)\\
=&-p_{\frac{t-s}2}(x,y)\id_{t>s}\\
&-\left(\frac yx\right)^{\alpha/2}\frac n{2\pi\ii}\int_\Gamma\d z\int_{-\infty}^0\d w\int_0^\infty\d\lambda\left(\frac wz\right)^{\alpha/2}\exp\left(n(\log(1-w)-\log(1-z))\right)\\
&\quad\times\exp\left(\frac{2qnw+\frac y{2q\sqrt n}}{\frac{2q}{1+q}\left(1-\frac t{(1+q)\sqrt n}\right)}-\frac{2qnz+\frac x{2q\sqrt n}}{\frac{2q}{1+q}\left(1-\frac s{(1+q)\sqrt n}\right)}+qn(z-w)
-\lambda\sqrt n(z-w)\right)\\
&\quad\times I_\alpha\left(-2n^{1/4}\sqrt{yw}(1+\Or(n^{-1/2}))\right)I_\alpha\left(2n^{1/4}\sqrt{xz}(1+\Or(n^{-1/2}))\right).
\end{aligned}\end{equation}

We start with the asymptotic analysis for the $w$-integral.
We are to take the limit of
\begin{equation}\label{wintegral}
n^{1/2+\alpha/4}\int_{-\infty}^0\d w\,w^{\alpha/2}e^{n(\log(1-w)+w)+\sqrt ntw+\sqrt n\lambda w+\Or(w)}I_\alpha\left(-2n^{1/4}\sqrt{yw}(1+\Or(n^{-1/2}))\right)
\end{equation}
where $\Or(w)$ means an error term which is bounded in absolute value by constant times $|w|$ for $n$ large enough.
The constant may depend on other parameters: $y$ and $t$ here.
The term $\Or(n^{-1/2})$ has a similar meaning in \eqref{wintegral}.

We use the method of steep descent analysis to determine the limit of \eqref{wintegral} as explained also in \cite{BF} and \cite{FV}.
Note that by the exponential asymptotics $I_\alpha(z)=e^{z+\Or(1)}$ as $z\to\infty$, see e.g.\ 9.7.1 in \cite{AS},
the Bessel function in \eqref{wintegral} corresponds to a term of order $n^{1/4}\sqrt w$ in the exponent.
The first step of the analysis is to consider the function that is multiplied by the highest power of $n$ in the exponent
\begin{equation}\label{f0Taylor}
f_0(w)=\log(1-w)+w=-\frac{w^2}2+\Or(w^3).
\end{equation}
We show that $\R_-$ is a steep descent path for the function $\Re(f_0(w))$ which means that it reaches its maximum at $w=0$ and that $\Re(f_0(w))$ is monotone along $\R_-$.
It is easily seen by taking the derivative
\begin{equation}
\frac{\d}{\d t}\Re(f_0(-t))=-\frac t{1+t}\le0
\end{equation}
for all $t\ge0$.

For a $\delta>0$ small enough which will be given later, restricting the integral in \eqref{wintegral} to $(-\delta,0)$
gives an error of order $\Or(\exp(-c\delta^2n))$ as $n\to\infty$ because of the steep descent property of the path.
We substitute $f_0$ by its Taylor expansion given in \eqref{f0Taylor},
and now we choose $\delta$ small enough such that the error of the Taylor expansion is much smaller than the main terms in the exponent.
We do a change of variables $\sqrt nw\to w$.
As $n\to\infty$, the error term of the Taylor expansion is $\Or(n^{-1/2})$.
In order to compare the integral with and without this error term in the exponent,
we use the inequality $|e^x-1|\le|x|e^{|x|}$ and we get that the difference of the two integrals is $\Or(n^{-1/2})$.
In the integral without the error term in the exponent, we extend the steepest descent path to $\R_-$ again where an error of $\Or(\exp(-c\delta^2n))$ is made.

This procedure of steep descent analysis proves that \eqref{wintegral} converges to
\begin{equation}
\int_{-\infty}^0\d w\,w^{\alpha/2} e^{-\frac{w^2}2+tw+\lambda w}I_\alpha\left(-2\sqrt{yw}\right)
\end{equation}
as $n\to\infty$. It is also seen from the above argument that the convergence is uniform if $y$ is in a compact interval.

With a similar method for the $z$-integral, one can show that
\begin{multline}\label{zintegral}
\frac{n^{1/2-\alpha/4}}{2\pi\ii}\int_\Gamma\d z\,z^{-\alpha/2} e^{n(-\log(1-z)-z)-\sqrt nsz-\sqrt n\lambda z+\Or(z)}I_\alpha\left(2n^{1/4}\sqrt{xz}(1+\Or(n^{-1/2}))\right)\\
\to\frac1{2\pi\ii}\int_{-\ii\R}\d z\,z^{-\alpha/2} e^{\frac{z^2}2-sz-\lambda z}I_\alpha\left(2\sqrt{xz}\right)
\end{multline}
as $n\to\infty$ uniformly as $x$ is in a finite interval where $-\ii\R$ is $\ii\R$ with opposite orientation.
The error terms here have similar meaning as in \eqref{wintegral}.
In this argument, one has to first specify the contour $\Gamma$.
Before taking the $n\to\infty$ limit, let $\Gamma$ to be the counterclockwise oriented sector contour
$\{e^{-\ii(\pi/2-\varepsilon)}t,t\in[0,R]\}\cup\{e^{\ii\theta}R,\theta\in[-\pi/2+\varepsilon,\pi/2-\varepsilon]\}\cup\{e^{\ii(\pi/2-\varepsilon)}t,t\in[R,0]\}$
where $t\in[R,0]$ refers to opposite orientation. Here $\varepsilon>0$ is a small value and $R$ is sufficiently large such that the contour surrounds the point $b=2n$.
As $R\to\infty$, the integrand on the left-hand side of \eqref{zintegral} along $\{e^{\ii\theta}R,\theta\in[-\pi/2+\varepsilon,\pi/2-\varepsilon]\}$
becomes exponentially small in $R$ for large enough but fixed $n$,
hence we can replace $\Gamma$ by the contour $\Gamma_\varepsilon$ consisting of two rays one coming from $e^{\ii(\pi/2-\varepsilon)}\infty$ to $0$ and one from $0$ to $e^{-\ii(\pi/2-\varepsilon)}\infty$.

The reason why we use the contour $\Gamma_\varepsilon$ for some small $\varepsilon>0$ for the analysis instead of $\ii\R=\Gamma_0$ is that
by deforming the contour on the left-hand side of \eqref{zintegral}, we do not have enough control on the decay of the integrand along $\{e^{\pm\ii\theta}R,\theta\in[\pi/2-\varepsilon,\pi/2]\}$.

The path $\Gamma_\varepsilon$ is a steep descent contour for the function $-\Re(f_0(z))$, since
\[-\frac{\d}{\d t}\Re(f_0(e^{\pm\ii(\pi/2-\varepsilon)}t))
=\frac{\cos^2\left(\frac\pi2-\varepsilon\right)-\sin^2\left(\frac\pi2-\varepsilon\right)-\cos\left(\frac\pi2-\varepsilon\right)t}
{\left(1-\cos\left(\frac\pi2-\varepsilon\right)t\right)^2+\left(\sin\left(\frac\pi2-\varepsilon\right)t\right)^2}\le0\]
for all $t\ge0$ if $\varepsilon>0$ is small enough.
One can perform a similar asymptotic analysis as for the $w$-integral using the steep descent contour $\Gamma_\varepsilon$ for the $z$-integral.
One gets the convergence \eqref{zintegral} with $-\ii\R$ replaced by $\Gamma_\varepsilon$ on the right-hand side,
but these two integrals are the same due to the decay coming from the quadratic term in the  exponent.

To show that the second term in the right-hand side of \eqref{rescaledKPea} without the minus sign converges to
\begin{equation}\label{ntoinftylimit}
\left(\frac yx\right)^{\alpha/2}\frac1{2\pi\ii}\int_{\ii\R}\d z\int_{-\infty}^0\d w\left(\frac wz\right)^{\alpha/2}
\frac1{z-w}\frac{e^{z^2/2-sz}}{e^{w^2/2-tw}}I_\alpha\left(-2\sqrt{yw}\right)I_\alpha\left(2\sqrt{xz}\right),
\end{equation}
we use a dominated convergence argument for the $\lambda$ integral of \eqref{rescaledKPea2}.
Following the lines of the proof of Proposition~\ref{prop:Jbound}, but actually easier as in Lemma 4.1 of \cite{FV},
one can give an exponential bound in $\lambda$ on the expression in \eqref{wintegral} and on the left-hand side of \eqref{zintegral} uniformly in $n$.
We omit the proof here.
This justifies that the second term on the right-hand side of \eqref{rescaledKPea} converges to \eqref{ntoinftylimit}.

What is left to show is that \eqref{ntoinftylimit} can be rewritten as the double integral in \eqref{Lalphadef}.
To this end, it is convenient to split the integration contour $\ii\R$ in \eqref{ntoinftylimit} into the half-lines $\ii\R_+$ and $\ii\R_-$.
Let us first consider the contribution from $\ii\R_-$.
By substituting $z$ by $-z$ and $w$ by $-w$ and taking into account the branches of the powers, we get the contribution
\begin{equation}\label{halfline1}
\left(\frac yx\right)^{\alpha/2}\frac1{2\pi\ii}\int_{\ii\R_+}\d z\int_{0}^\infty\d w\left(\frac wz\right)^{\alpha/2}
\frac1{w-z}\frac{e^{z^2/2+sz}}{e^{w^2/2+tw}}J_\alpha\left(2\sqrt{yw}\right)J_\alpha\left(2\sqrt{xz}\right)
\end{equation}
where we also used the identity
\[e^{\ii\alpha\pi/2}I_\alpha\left(-2\ii\sqrt z\right)=J_\alpha\left(2\sqrt z\right)\]
for any $z\in\cee$ for the Bessel functions (see e.g.\ 9.6.3 in \cite{AS}).
Similarly, the contribution from the half-line $\ii\R_+$ in \eqref{ntoinftylimit} can be written in exactly the same form as \eqref{halfline1}, but with $\ii\R_+$ replaced by $\ii\R_-$.
To get this, we also use the identity
\[e^{-\ii\alpha\pi}J_\alpha\left(-2\sqrt{xz}\right)=J_\alpha\left(2\sqrt{xz}\right)\]
for $x\in\R_+$ and $z\in\ii\R_-$ which is a consequence of the series expansion of the Bessel function (see e.g.\ 9.1.10 in \cite{AS}).

Combining these two contributions, we conclude that \eqref{ntoinftylimit} equals
\begin{equation}\label{halfline2}
\left(\frac yx\right)^{\alpha/2}\frac1{2\pi\ii}\int_{\ii\R}\d z\int_{0}^\infty\d w\left(\frac wz\right)^{\alpha/2}
\frac1{w-z}\frac{e^{z^2/2+sz}}{e^{w^2/2+tw}}J_\alpha\left(2\sqrt{yw}\right)J_\alpha\left(2\sqrt{xz}\right).
\end{equation}
Making the change of variables $u^2=w$ and $v^2=z$,
we rewrite this as
\begin{equation}
-\left(\frac yx\right)^{\alpha/2}\frac1{2\pi\ii}\int_{C}\d v\int_{0}^\infty\d u\, 4uv\left(\frac uv\right)^{\alpha}
\frac1{v^2-u^2}\frac{e^{v^4/2+sv^2}}{e^{u^4/2+tu^2}}J_\alpha\left(2\sqrt{y}u\right)J_\alpha\left(2\sqrt{x}v\right)
\end{equation}
where we recall the definition of the contour $C$. Hence we obtain the double integral on the right-hand side of \eqref{Lalphadef}.
\end{proof}

\section{Proofs of alternative hard edge Pearcey formulations and corollaries}\label{s:alternative_Pearcey}

In this section, we prove Proposition~\ref{thm:Pearcey:alternative} and Corollaries~\ref{cor:KMW} and \ref{cor:BK}.
The strategy for proving the proposition is the following.
If $\alpha$ is an integer, we can give another formulation of the kernel $L^\alpha$ in Proposition~\ref{prop:Pearcey:alternative2} below via the approach used to derive the hard edge tacnode kernel.
Then we show in Proposition~\ref{prop:pearcey:rank1} that the temperature derivative of the kernel has a rank one structure.
Finally, in Lemma~\ref{lemma:rank1:eq}, we see that the functions which appear in the rank one derivative can be written in a simpler form given in \eqref{g1:Pea}--\eqref{g2:Pea}.
The proof of Proposition~\ref{thm:Pearcey:alternative} is easy after this.

Let
\begin{equation}\label{Bx:def}
B(x) = e^{-x^2/2} = \frac{1}{\sqrt{2\pi}\ii}\int_{\ii\er}\d w e^{w^2/2-xw}.
\end{equation}
Let $\DD\alpha$ be the matrix of size $\al\times\al$ with $(k,l)$th entry
\begin{equation}\label{D:Pea}
(\DD\alpha)_{k,l} = B^{(k+l)}(-\si),\qquad k,l=0,\ldots,\al-1
\end{equation}
where the superscript denotes the $(k+l)$th derivative.
Let the row vector ${\boldsymbol \UU}$ of size $\al$ be defined as before by \eqref{defXi}, and let ${\bf h}$ be the column vector of size $\al$ with $k$th entry
\begin{equation}\label{hk:Pea}
h_k = \int_0^\infty \d\lam\, e^{-\vv\lam} B^{(k)}(\lam-\si),\qquad k=0,\ldots,\al-1.
\end{equation}
Then we put
\begin{equation}\label{M:Pea}
\Mm(\uu,\vv) = \frac{1}{\vv-\uu}-\boldsymbol{\UU}(\DD\alpha)^{-1}{\bf h}.
\end{equation}
Note that we use the same symbols ${\bf h}$ and $\Mm(\uu,\vv)$ in this section for different functions as in the sections about the hard edge tacnode process, but this should not lead to any confusion.
The invertibility of the matrix $\DD\alpha$ follows from the identity
\begin{equation}\label{detDDalpha}
\det\DD\alpha = \left(\prod_{j=0}^{\al-1} (-1)^j j!\right) e^{-\al\si^2/2}
\end{equation}
which is easily proved by induction on $\al$, see \eqref{Hermite:id2:step1bis} below.
Now we are ready to state

\begin{proposition}\label{prop:Pearcey:alternative2}
If $\alpha$ is a non-negative integer, then we can write $L^\alpha$ in \eqref{Lalphadef} as
\begin{multline}\label{Lalpharewrite}
L^\alpha(s,x;t,y)=-p_{\frac{t-s}{2}}(x,y)\id_{t>s}\\
+\frac1{(2\pi\ii)^2}\int_{\Gamma_{-s}}\d\vv\int_{\delta+\ii\R}\d\uu\,
\frac{e^{-\frac{\vv^2}2+\si\vv+\frac x{\vv+s}}}{e^{-\frac{\uu^2}2+\si\uu+\frac y{\uu+t}}}\frac{(\uu+t)^{\al-1}}{(\vv+s)^{\al+1}}\Mm(\uu,\vv)
\end{multline}
where $\Gamma_{-s}$ is a clockwise oriented circle surrounding the singularity at $-s$ and $\delta>0$ is chosen such that the contour $\delta+\ii\R$
passes to the right of the singularity at $-t$ and to the right of the contour $\Gamma_{-s}$.
\end{proposition}

\begin{proposition}\label{prop:pearcey:rank1}
The derivative of the kernel \eqref{Lalpharewrite} with respect to $\si$ is the rank 1 kernel
\[\frac{\pa}{\pa\si} L^\alpha(s,x;t,y) = g_1(s,x)g_2(t,y)\]
where
\begin{align}
\label{g1:Pea:Schur} g_1(s,x) &= \frac1{2\pi\ii}\int_{\Gamma_{-s}}\d\vv\, e^{-\frac{\vv^2}2+\si\vv+\frac x{\vv+s}}(\vv+s)^{-(\al+1)}
\left[\begin{pmatrix} 0 & \cdots & 0 & 1 \end{pmatrix}\DD\alpha^{-1}{\bf h}\right],\\
\label{g2:Pea:Schur} g_2(t,y) &= \frac1{2\pi\ii}\int_{\delta+\ii\R}\d\uu\, e^{\frac{\uu^2}2-\si\uu-\frac y{\uu+t}}(\uu+t)^{\al-1}
\left[\uu^{\al}-{\boldsymbol \UU}\DD\alpha^{-1}\left( B^{(k+\al)}(-\si) \right)_{k=0}^{\al-1}\right].
\end{align}
If $\al=0$ then the expressions between square brackets above are understood to be $1$.
\end{proposition}

The proofs of Propositions~\ref{prop:Pearcey:alternative2} and \ref{prop:pearcey:rank1} are very similar to the ones for the analogous statements for the hard edge tacnode kernel.
We leave this to the interested reader.

Contrary to the tacnode case, it turns out that the above formulas can be considerably simplified.

\begin{lemma}\label{lemma:rank1:eq} The functions given in \eqref{g1:Pea:Schur}--\eqref{g2:Pea:Schur} can be written in the simplified form
\begin{align}
\label{g1:Pea} g_1(s,x) &= \frac{(-x)^{-\al}}{2\pi\ii}\int_{\Gamma_{-s}}\d\vv\, e^{-\frac{\vv^2}2+\si\vv+\frac x{\vv+s}}(\vv+s)^{\al-1},\\
\label{g2:Pea} g_2(t,y) &= \frac{(-y)^{\al}}{2\pi\ii}\int_{\delta+\ii\R}\d\uu\, e^{\frac{\uu^2}2-\si\uu-\frac y{\uu+t}}(\uu+t)^{-\al-1}.
\end{align}
\end{lemma}

An important ingredient for the proof of Lemma~\ref{lemma:rank1:eq} is the following elementary analysis lemma which we give with a draft of the proof.

\begin{lemma}[Repeated integration by parts]\label{lemma:int:parts:repeated}
For any $\al\in\zet_{\geq 0}$ and any function $f$ with sufficient decay at infinity in the direction of the contour $\Gamma_\vv$, we have
\begin{equation}\label{int:parts:repeated}
\int_{\Gamma_\vv}\ud\vv\, e^{\frac{x}{\vv+s}}(\vv+s)^{-\al-1}f(\vv) = x^{-\al}\int_{\Gamma_\vv}\ud\vv\, e^{\frac{x}{\vv+s}}(\vv+s)^{\al-1}\frac{\pa^{\al} f}{\pa\vv^{\al}}(\vv).
\end{equation}
\end{lemma}

\begin{proof}[Sketch of proof]
This is trivial for $\al=0$.
For $\al=1$, it follows by applying integration by parts to the left-hand side of \eqref{int:parts:repeated}.
For higher values of $\al$, it follows from a repeated ($\al$-fold) use of integration by parts
writing each time $e^{\frac{x}{\vv+s}}(\vv+s)^{k}\ud\vv = -x^{-1}(\vv+s)^{k+2}\ud e^{\frac{x}{\vv+s}}$ and then integrating by parts.
This allows to express the left-hand side of \eqref{int:parts:repeated} as $x^{-\al}$ times a sum of several terms.
Then one notes that these terms cancel out each other completely except for one surviving term which is precisely the right-hand side of \eqref{int:parts:repeated}.
\end{proof}

\begin{proof}[Proof of Lemma~\ref{lemma:rank1:eq}]
This is obvious for $\al=0$.
For general $\al$, the first step in the proof is to use Lemma~\ref{lemma:int:parts:repeated}.
Then it remains to prove two identities, one for $g_1$ and one for $g_2$.

We start with the identity for $g_1$ which reads
\begin{equation}\label{Hermite:id1}
\begin{pmatrix} 0 & \cdots & 0 & 1 \end{pmatrix}\DD\al^{-1} \frac{\pa^{\al}}{\pa\vv^\al}\left(e^{-\vv^2/2+\si\vv}{h_l}\right)_{l=0}^{\alpha-1} = (-1)^\al e^{-\vv^2/2+\si\vv}.
\end{equation}
To prove this identity, we first derive the formula
\begin{equation}\label{h:derivative}
\frac{\pa}{\pa\vv}h_l=(\vv-\sigma)h_l-B^{(l)}(-\sigma)+lh_{l-1}
\end{equation}
as follows.
Note that the derivatives of the function $B$ satisfy the recursion relation
\begin{equation}\label{B:recursion}
\lambda B^{(l)}(\lambda-\sigma)=\sigma B^{(l)}(\lambda-\sigma)-B^{(l+1)}(\lambda-\sigma)-lB^{(l-1)}(\lambda-\sigma)
\end{equation}
for all $l=0,1,2,\dots$ which is closely related to the recursion for the Hermite polynomials.
By taking the $\vv$ derivative of \eqref{hk:Pea} and using \eqref{B:recursion}, one gets
\begin{equation}\label{h:derivative2}
\frac{\pa}{\pa\vv}h_l=h_{l+1}+lh_{l-1}-\sigma h_l.
\end{equation}
Integration by parts yields the equation
\begin{equation}
h_{l+1}=\vv h_l-B^{(l)}(-\sigma)
\end{equation}
which substituted into \eqref{h:derivative2} exactly gives \eqref{h:derivative}.

If we apply one derivation on the left-hand side of \eqref{Hermite:id1} to the vector of entries $e^{-\vv^2/2+\si\vv}h_l$, then by \eqref{h:derivative}, we get
\begin{equation}\label{h:first:deriv}\begin{aligned}
\frac{\pa}{\pa\vv}\left(e^{-\vv^2/2+\si\vv}{h_l}\right)_{l=0}^{\alpha-1}&=e^{-\vv^2/2+\sigma\vv}\left(-B^{(l)}(-\sigma)+lh_{l-1}\right)_{l=0}^{\alpha-1}\\
&=e^{-\vv^2/2+\sigma\vv}\left(-\DD\alpha \begin{pmatrix} 1 & 0 & \dots & 0 \end{pmatrix}^T+\left(lh_{l-1}\right)_{l=0}^{\alpha-1}\right).
\end{aligned}\end{equation}
If $\alpha=1$, then $(lh_{l-1})_{l=0}^{\alpha-1}$ is $0$ and \eqref{Hermite:id1} follows easily.
Otherwise (if $\al>1$) after substituting \eqref{h:first:deriv} into the left-hand side of \eqref{Hermite:id1} and cancelling $\DD\alpha$ with its inverse, the scalar product of the two vectors is $0$,
hence the left-hand side of \eqref{Hermite:id1} can be written as
\begin{equation}\label{Hermite:id1:step1}
\begin{pmatrix} 0 & \cdots & 0 & 1 \end{pmatrix}\DD\al^{-1}  \frac{\pa^{\al-1}}{\pa\vv^{\al-1}}\left(e^{-\vv^2/2+\si\vv}{l h_{l-1}}\right)_{l=0}^{\alpha-1}.
\end{equation}

At this point we apply the following elementary column operations to the matrix $\DD\al$.
For $l=\al,\al-1,\ldots,1$ we update column $l$ by subtracting $\si$ times column $l-1$ and adding $l-1$ times column $l-2$.
By \eqref{B:recursion} with $\lambda=0$, the resulting matrix has zeros in its zeroth row, except for the $(0,0)$ entry, which is $e^{-\si^2/2}$.
We should apply the same operations to the row vector $\begin{pmatrix} 0 & \cdots & 0 & 1 \end{pmatrix}$, but this vector does not change under the operations.
Then \eqref{Hermite:id1:step1} reduces to
\begin{equation}\label{Hermite:id1:step2}
\begin{pmatrix} 0 & \cdots & 0 & 1 \end{pmatrix}\begin{pmatrix} e^{-\si^2/2} & 0 \\ * & \left( -(k+1)B^{(k+l)}(-\si) \right)_{k,l=0}^{\al-2} \end{pmatrix}^{-1}
\frac{\pa^{\al-1}}{\pa\vv^{\al-1}}\left(e^{-\vv^2/2+\si\vv}{l h_{l-1}}\right)_{l=0}^{\alpha-1}
\end{equation}
where $*$ denotes an unimportant column vector.
Taking into account the sparsity pattern of the involved matrices and vectors and using the block inversion formula \eqref{schur:inversion}, this in turn can be reduced to
\begin{equation}\label{Hermite:id1:step3}
\begin{pmatrix} 0 & \cdots & 0 & 1 \end{pmatrix}
\left[\begin{pmatrix} -(k+1)B^{(k+l)}(-\si) \end{pmatrix}_{k,l=0}^{\al-2}\right]^{-1} \frac{\pa^{\al-1}}{\pa\vv^{\al-1}}\left(e^{-\vv^2/2+\si\vv}{l h_{l-1}}\right)_{l=1}^{\alpha-1}.
\end{equation}
Finally, we note that the factors $(k+1)$ in the matrix rows and $l$ in the entries of the column vector cancel each other.
So we get
\begin{equation}\label{Hermite:id1:step4}
(-1)\begin{pmatrix} 0 & \cdots & 0 & 1 \end{pmatrix} \DD{\al-1}^{-1}  \frac{\pa^{\al-1}}{\pa\vv^{\al-1}}\left(e^{-\vv^2/2+\si\vv}{h_{l}}\right)_{l=0}^{\alpha-2}.
\end{equation}
But this is just what we get from the identity for $\al-1$ in \eqref{Hermite:id1},
so the proof of \eqref{Hermite:id1} ends by induction yielding the equivalence of $g_1$ in \eqref{g1:Pea} and in \eqref{g1:Pea:Schur}.

Now we turn to the identity for $g_2$. Using Lemma~\ref{lemma:int:parts:repeated}, we will prove the identity
\begin{equation}\label{Hermite:id2}
\left(\frac{\pa^{\al}}{\pa \uu^\al}e^{\uu^2/2-\si\uu}\right)e^{-\uu^2/2+\si\uu} = \uu^\al- {\boldsymbol \UU}\DD\al^{-1}\left( B^{(k+\al)}(-\si) \right)_{k=0}^{\al-1}
\end{equation}
by induction on $\al$. For $\al=0$ the identity is trivial since both sides are $1$.
For $\al\geq 1$, denote the left-hand side of \eqref{Hermite:id2} by $H_\al$ and the right-hand side by $\wtil H_\al$.
Note that $H_\al$ is basically a Hermite polynomial and it satisfies the recursion
\begin{equation}\label{Hermite:id2:rec}
H_{\al} = (\uu-\si)H_{\al-1}+H_{\al-1}',\qquad \al\in\zet_{>0}
\end{equation}
where the prime denotes the $\uu$-derivative.
We will show that $\wtil H_\al$ satisfies the same recursion.
The expression $\wtil H_\al$ has a Schur complement form and therefore it can be written as a ratio of two nested determinants:
\begin{equation}\label{Hermite:id2:Schur} \wtil H_\al = \det E_\al/\det \DD\al \qquad \textrm{with}\qquad E_\al := \begin{pmatrix} \DD\al & \left( B^{(k+\al)}(-\si) \right)_{k=0}^{\al-1} \\
{\boldsymbol \UU} & \uu^\al\end{pmatrix}.
\end{equation}
Note that the last column of $E_\al$ nicely follows the pattern of the previous columns.
We note that the entries in the topmost rows of $E_\al$ are basically the Hermite polynomials and they satisfy the recursion \eqref{B:recursion} with $\lambda=0$.
We use the following column operations for $E_\al$: for $l=\al,\al-1,\ldots,1$ we update column $l$ by subtracting $\si$ times column $l-1$ and adding $l-1$ times column $l-2$.
The resulting matrix has zeros in its zeroth row, except for the $(0,0)$ entry, which is $e^{-\si^2/2}$, and so we can reduce the size of the determinant. This gives us
\[\det E_\al = e^{-\si^2/2}\det\begin{pmatrix} \left( -(k+1)B^{(k+l)}(-\si) \right)_{k=0}^{\al-2} \\ (\uu-\si+l\uu^{-1})\uu^{l} \end{pmatrix}_{l=0}^{\al-1}.\]
We take out the factor $-(k+1)$ from the $k$th row yielding a total contribution
\[(-1)^{\al-1}(\al-1)!.\]
Next we split the determinant by linearity with respect to its last row, recognizing the smaller matrix $E_{\al-1}$ in one of the terms, and its $\uu$-derivative $E_{\al-1}'$ in the other term.
This yields
\begin{equation}\label{Hermite:id2:step1}
\det E_\al = (-1)^{\al-1}(\al-1)!e^{-\si^2/2}\left[ (\uu-\si)\det E_{\al-1}+\det E_{\al-1}'\right].
\end{equation}
Similar (but actually simpler) manipulations yield
\begin{equation}\label{Hermite:id2:step1bis}
\det \DD\al = (-1)^{\al-1}(\al-1)!e^{-\si^2/2}\det \DD{\al-1}.
\end{equation}
Using these relations in \eqref{Hermite:id2:Schur}, we find that $\wtil H_\al$ satisfies the required recursion relation \eqref{Hermite:id2:rec}.
This ends the proof of the identity \eqref{Hermite:id2}, hence the two definitions \eqref{g2:Pea} and \eqref{g2:Pea:Schur} of $g_2$ coincide.
\end{proof}

\begin{proof}[Proof of Proposition~\ref{thm:Pearcey:alternative}]
By Proposition~\ref{prop:pearcey:rank1}, the derivative of the kernel $L^\alpha$ with respect to $\sigma$ can be written
in a rank one form with $g_1$ and $g_2$ given by \eqref{g1:Pea:Schur}--\eqref{g2:Pea:Schur}.
Note that the derivative of the right-hand side of \eqref{Lalpha:KMW:step1} is also rank one but with $g_1$ and $g_2$ given by \eqref{g1:Pea}--\eqref{g2:Pea}.
Hence Lemma~\ref{lemma:rank1:eq} shows that the two derivatives are the same.
Moreover, both double integrals converge to zero in the $\si\to +\infty$ limit.
This ends the proof.
\end{proof}

Now we turn to the proofs of the two corollaries.

\begin{proof}[Proof of Corollary~\ref{cor:KMW}]
Similarly to Remark~\ref{rem:cafasso} for the tacnode kernel, by introducing the new integration variables $u:=-(w+s)^{-1}$ and $v:=-(z+t)^{-1}$ in \eqref{Lalpha:KMW:step1},
the kernel can be written alternatively as
\begin{equation}\label{Lalpha:KMW:step2}\begin{aligned}
L^\alpha(s,x;t,y) = &-p_{\frac{t-s}{2}}(x,y)\id_{t>s}\\
&+ \left(\frac yx\right)^\alpha \frac1{(2\pi\ii)^2}\int_{\Gamma_u}\d u\, \int_{\Gamma_v}\d v\, \frac{1}{u-v+uv(t-s)}
\frac{e^{(v^{-1}+t)^2/2+\si (v^{-1}+t)+yv}}{e^{(u^{-1}+s)^2/2+\si (u^{-1}+s)+xu}}\frac{v^{\al}}{u^{\al}}
\end{aligned}\end{equation}
where $\Gamma_v$ is a clockwise circle in the left half-plane touching zero along the imaginary axis, and $\Gamma_u$ is a counterclockwise loop surrounding $\Gamma_v$.
It is now straightforward to get \eqref{KMW:reduction}.
\end{proof}

\begin{remark}\label{remark:KMW:noninteger}
Similarly to Remark~\ref{rem:Pearcy:nonint}, \eqref{Lalpha:KMW:step2} also extends to any real $\al>-1$.
In this case, we take $\Gamma_v$ as above but we replace $\Gamma_u$ by a counterclockwise contour coming from $+\infty$ from above, encircling $\Gamma_v$, and returning to $+\infty$ from below.
We then take the powers $u^\al$, $v^\al$ with a branch cut along the positive half-line, i.e., with $0<\arg u,\arg v<2\pi$, see also \cite[Eq.\ (1.19)]{KMW2}.
\end{remark}

\begin{proof}[Proof of Corollary~\ref{cor:BK}]
For $\alpha=-1/2$, we have
\begin{align*}
J_{-1/2}\left(2\sqrt yu\right)&=\frac1{\sqrt\pi y^{1/4}\sqrt u}\cos\left(2\sqrt yu\right)
\end{align*}
if $u>0$ or $u\in C$. Substituting this, we get the following expression for the kernel
\begin{equation}\begin{aligned}
L^{-1/2}(s,x;t,y)= &-p_{\frac{t-s}2}(x,y)\id_{t>s}\\
&+\frac2{\pi^2\ii\sqrt y}\int_C\d v\int_{\R_+}\d u\frac v{v^2-u^2}\frac{e^{v^4/2+sv^2}}{e^{u^4/2+tu^2}}\cos\left(2\sqrt yu\right)\cos\left(2\sqrt xv\right).
\end{aligned}\end{equation}
To conclude the proof, one has to do the change of variables $u\to2^{1/4}u$ and $v\to2^{1/4}v$ and substitute the new variables from \eqref{kernelchangeofvariable}
along with the derivative $\frac{\d y}{\d\sigma_1}=\frac{\sigma_1}{2\sqrt2}$, the rest is straightforward.
\end{proof}

\section{Asymptotic invertibility}\label{s:asymptinv}

In this section, we prove Proposition~\ref{prop:limSchur}.
It relies on Lemmas~\ref{lemma:detAirewrite}, \ref{lemma:fklasympt} and \ref{lemma:detAn}.
In order to state them, we introduce the following notation.
Then we give the proof of the proposition, finally, we prove the lemmas.
Inspired by the asymptotics of the Airy function
\begin{equation}\label{Aiasympt}
\Ai(x)=\frac1{2\sqrt\pi}x^{-1/4}\exp\left(-\frac23 x^{3/2}\right)\left(1+\Or(x^{-3/2})\right)
\end{equation}
as $x\to\infty$, see e.g.\ 10.4.59 in \cite{AS}, let us denote
\begin{equation}\label{deffk}
f_k(x)=2\sqrt\pi\exp\left(\frac23 x^{3/2}\right)\frac{\pa^k}{\pa x^k}\Ai(x)
\end{equation}
for $k=0,1,2,\dots$.

\begin{lemma}\label{lemma:detAirewrite}
For any positive integer $n$, one has the equality
\begin{equation}\label{detAirewrite}
\det\left(\frac{\pa^{k+l}}{\pa x^{k+l}}\Ai(x)\right)_{k,l=0}^{n-1}=\frac{e^{-\frac 23 n x^{3/2}}}{2^n \pi^{n/2}}\det\left(f_k^{(l)}(x)\right)_{k,l=0}^{n-1}
\end{equation}
where $f_k^{(l)}$ is the $l$th derivative of $f_k$.
\end{lemma}

By the definition $\Ai''(x)=x\Ai(x)$ of the Airy function, all higher derivatives of $f_k(x)$ can be expressed in terms $\Ai(x)$ and $\Ai'(x)$.
By \eqref{Aiasympt} and the asymptotics of $\Ai'(x)$, see 10.4.61 in \cite{AS}, it follows that
\begin{equation}\label{fklasympt}
f_k^{(l)}(x)\sim(-1)^k 4^{-l} x^{k/2-1/4-l}a_{k,l}
\end{equation}
as $x\to\infty$ with
\begin{equation}\label{defakl}
a_{k,l}=(2k-1)(2k-5)\dots(2k+3-4l)
\end{equation}
where the empty product is defined to be $1$.

Note that \eqref{fklasympt} already yields the following

\begin{lemma}\label{lemma:fklasympt}
We have the asymptotics
\begin{equation}\label{detfklasympt}
\det\left(f_k^{(l)}(x)\right)_{k,l=0}^{n-1}\sim(-1)^{\binom n2}4^{-\binom n2}x^{-n^2/4}\det\left(\left(a_{k,l}\right)_{k,l=0}^{n-1}\right)
\end{equation}
as $x\to\infty$.
\end{lemma}

The matrix in the lemma has the form
\begin{equation}
\left(a_{k,l}\right)_{k,l=0}^{n-1} =\begin{pmatrix} 1 & -1 & (-1)(-5) & (-1)(-5)(-9) & \hdots \\ 1 & 1 & 1(-3) & 1(-3)(-7) & \hdots \\ 1 & 3 & 3(-1) & 3(-1)(-5) & \hdots \\
1 & 5 & 5\cdot1 & 5\cdot1(-3) & \hdots \\ \vdots & \vdots & \vdots & \vdots & \ddots \end{pmatrix}.
\end{equation}

\begin{lemma}\label{lemma:detAn}
The determinant of the $n\times n$ matrix defined by \eqref{defakl} is
\begin{equation}\label{detAn}
\det\left(\left(a_{k,l}\right)_{k,l=0}^{n-1}\right)=2^{\binom n2}\left(\prod_{j=0}^{n-1} j!\right).
\end{equation}
\end{lemma}

\begin{proof}[Proof of Proposition~\ref{prop:limSchur}]
\begin{enumerate}
\item The asymptotics \eqref{Airy:der:matrix} readily follows if one puts together \eqref{detAirewrite}, \eqref{detfklasympt} and \eqref{detAn}.
\item The assertion is a consequence of the first part and the discussion in the paragraph before the statement of the proposition.
\end{enumerate}
\end{proof}

\begin{proof}[Proof of Lemma~\ref{lemma:detAirewrite}]
By definition \eqref{deffk}, one immediately gets \eqref{detAirewrite} with $f_k^{(l)}(x)$ replaced by $f_{k+l}(x)$ in the determinant on the right-hand side.
Hence one only needs to show that
\begin{equation}\label{fdetequality}
\det\left(f_{k+l}(x)\right)_{k,l=0}^{n-1}=\det\left(f_k^{(l)}(x)\right)_{k,l=0}^{n-1}.
\end{equation}

It follows from \eqref{deffk} that the recursion relation
\begin{equation}\label{recursion_fk}
f_{k+1}(x)=-x^{1/2}f_k(x)+f_k'(x)
\end{equation}
holds for any $k\ge0$ integer.
In \eqref{recursion_fk}, $f_{k+1}$ is expressed in terms of $f_k$ and its derivative where the coefficient of the latter is $1$.
By applying \eqref{recursion_fk} and its derivative on the right-hand side of \eqref{recursion_fk} again,
one can express $f_{k+2}$ in terms of $f_k$ and its first and second derivative where the coefficient of the second derivative is $1$.
In general, one gets by induction that
\begin{equation}\label{generalrecursion_fk}
f_{k+l}(x)=c_{l,0}(x)f_k(x)+c_{l,1}(x)f_k'(x)+\dots+c_{l,l-1}(x)f_k^{(l-1)}(x)+f_k^{(l)}(x)
\end{equation}
for all $k,l$ non-negative integers where $c_{l,0}(x),\dots,c_{l,l-1}(x)$ are coefficients that depend on $x$ but not on $k$.

The equality of the determinants \eqref{fdetequality} now follows by simple column transformations.
We first represent the entries on the left-hand side of \eqref{fdetequality} as given in \eqref{generalrecursion_fk}.
Then we perform the following step for $l=1,\dots,n-1$:
from the $l$th column of the matrix, we subtract $c_{l,0}$ times the $0$ column, $c_{l,1}$ times the first column, up to $c_{l,l-1}$ times the $l-1$st column.
After these transformations, we recover the right-hand side of \eqref{fdetequality}, which completes the proof.
\end{proof}

\begin{proof}[Proof of Lemma~\ref{lemma:detAn}]
We apply elementary column operation on the matrix $\left(a_{k,l}\right)_{k,l=0}^{n-1}$ in order to get $0$ everywhere in the first row except for the first position as follows.
Starting from the last column and proceeding from right to left until the second column,
we add an appropriate multiple of the previous column to each of the columns such that we get $0$ in the top position of the column.
With this procedure, we get that
\begin{equation}\label{Anexpand}
\det\left(\left(a_{k,l}\right)_{k,l=0}^{n-1}\right)=\det\begin{pmatrix} 1 & 0 & 0 & 0 & \hdots \\ 1 & 2 & 1\cdot2 & 1(-3)\cdot2 & \hdots \\ 1 & 4 & 3\cdot4 & 3(-1)\cdot4 & \hdots \\
1 & 6 & 5\cdot6 & 5\cdot1\cdot6 & \hdots \\ \vdots & \vdots & \vdots & \vdots & \ddots\end{pmatrix}
=\det\begin{pmatrix} 2 & 1\cdot2 & 1(-3)\cdot2 & \hdots \\ 4 & 3\cdot4 & 3(-1)\cdot4 & \hdots \\ 6 & 5\cdot6 & 5\cdot1\cdot6 & \hdots \\
\vdots & \vdots & \vdots & \ddots\end{pmatrix}
\end{equation}
where we expanded the determinant along the first row in the last step.
By dividing the $k$th row on the right-hand side of \eqref{Anexpand} by $2k$,
we get that the determinant of the matrix $\left(a_{k,l}\right)_{k,l=0}^{n-1}$ is $2^{n-1}(n-1)!$ times a similar $(n-1)\times(n-1)$ determinant as that of $\left(a_{k,l}\right)_{k,l=0}^{n-1}$.
The statement of the lemma now follows by induction.
\end{proof}

\appendix

\section{Limit of Bessel functions}\label{s:Bessel_limit}

\begin{proposition}\label{prop:Jconv}
Fix a non-negative integer $k$. Then for any $s>0$,
\begin{equation}\label{Jconv}
J^{(k)}(s,t) \to \frac{\partial^k}{\partial s^k} \Ai(s)
\end{equation}
pointwise as $t\to\infty$.
\end{proposition}

\begin{proof}
In the integral representation of $J^{(k)}$ in \eqref{Jdifference}, the coefficient of the leading term in the exponent as $t\to\infty$ is the function
\[f(z)=-\ln z+\frac{z-z^{-1}}2=\frac16(z-1)^3+\Or((z-1)^4)\]
for which we introduce the following steep descent path.
It follows from the more general derivative calculation in \eqref{steepdescwtf} that the circle $S_{1+\varepsilon}$ is a steep descent path for $\Re(f(z))$ for any $\varepsilon>0$.
Let us choose a small $\varepsilon$ such that $S_{1+\varepsilon}$ replaced in a small neighborhood of $1$ by a segment of the path $\{1+e^{\ii\pi/3}s,s>0\}\cup\{1+e^{-\ii\pi/3}s,s>0\}$
such that they form a closed loop around the origin.

For the contour given above, we perform the usual steep descent analysis, so we first neglect the integral on the right-hand side of \eqref{Jdifference} except for a small neighborhood of $1$.
After the change of variable $z=1+Zt^{-1/3}$, we get that
\begin{equation}\label{Zkint}
J^{(k)}(s,t)\to\frac1{2\pi\ii}\int_{e^{-\ii\pi/3}\infty}^{e^{\ii\pi/3}\infty} \d Z e^{Z^3/3-sZ}(-Z)^k.
\end{equation}
To show that the right-hand side of \eqref{Zkint} and that of \eqref{Jconv} are the same, we apply a dominated convergence argument.
With the notation
\[f(s,Z)=e^{Z^3/3-sZ},\]
we can write
\begin{align*}
\frac1{2\pi\ii}\int_{e^{-\ii\pi/3}\infty}^{e^{\ii\pi/3}\infty}\d Z \frac{\partial^k}{\partial s^k} f(s,Z)
&=\frac1{2\pi\ii}\int_{e^{-\ii\pi/3}\infty}^{e^{\ii\pi/3}\infty}\d Z\lim_{h\to 0}h^{-k}\sum_{p=0}^k\binom kp(-1)^{k-p}f(s+ph,Z)\\
&=\lim_{h\to 0}h^{-k}\sum_{p=0}^k\binom kp(-1)^{k-p}\frac1{2\pi\ii}\int_{e^{-\ii\pi/3}\infty}^{e^{\ii\pi/3}\infty}\d Zf(s+ph,Z)\\
&=\frac{\partial^k}{\partial s^k}\frac1{2\pi\ii}\int_{e^{-\ii\pi/3}\infty}^{e^{\ii\pi/3}\infty}\d Z f(s,Z)
\end{align*}
where we used the dominated convergence theorem in the second step above.
The integrand can be dominated uniformly in $h$ because of the fast decay of $f(s,Z)$ along the $Z$ contour.
\end{proof}

\begin{proposition}\label{prop:Jbound}
Fix a non-negative integer $k$.
For any $c>0$, there is a finite constant $C$ and thresholds $t_0$, $s_0$ such that
\begin{equation}\label{Jbound}
\left|J^{(k)}(s,t)\right|\le Ce^{-cs}
\end{equation}
holds uniformly for $t>t_0$ if $s>s_0$.
\end{proposition}

\begin{proof}
We adapt the method of proof of Proposition 5.3 in~\cite{BF07} or that of Lemma 4.1 in~\cite{FV} to our setting.
In order to investigate the large values of $s$, we rescale it as
\begin{equation}\label{defwts}
\wt s=t^{-2/3}s
\end{equation}
and we define the function
\begin{equation}\label{defwtf}
\wt f(z)=-\ln z+\frac{z-z^{-1}}2-\frac12\wt s\ln z
\end{equation}
which gives the leading contribution in the exponent of the integrand on the right-hand side of \eqref{Jdifference} if $s$ is of order $t^{2/3}$.

For small values of $\wt s$, this function has two critical points at $1\pm\wt s^{1/2}$ at first order, and we will pass through the larger one.
Hence define
\begin{equation}\label{defalpha}
\alpha=\left\{\begin{array}{ll} 1+\wt s^{1/2} & \mbox{if}\ \wt s\le\varepsilon,\\ 1+\varepsilon^{1/2} & \mbox{if}\ \wt s>\varepsilon\end{array}\right.
\end{equation}
for some small $\varepsilon>0$ to be chosen later.

On the right-hand side of \eqref{Jdifference}, we can change the integration contour by the Cauchy theorem to a circle of radius $\alpha$.
Using the definition \eqref{defwtf}, this yields
\begin{equation}\label{Jkwithwtf}
J^{(k)}(s,t)=t^{\frac{k+1}3}\frac1{2\pi\ii}\int_{S_\alpha}\d z \exp(2t\wt f(z))z^{-1}\left(z^{-1}-1\right)^k.
\end{equation}
The path $S_\alpha$ is steep descent for the function $\Re(\wt f(z))$, since
\begin{equation}\label{steepdescwtf}
\frac{\partial}{\partial\theta}\Re(\wt f(\alpha e^{\ii\theta}))=-\frac{\alpha-\alpha^{-1}}2\sin\theta
\end{equation}
which is negative for $\alpha>1$ and $\theta\in(0,\pi)$.

We show that the integral in \eqref{Jkwithwtf} can be bounded by the value of the integrand at $\alpha$.
To this end, define
\begin{equation}
Q(\alpha)=t^{k/3}\exp\left(\Re\left(2t\wt f(\alpha)\right)\right)\left|\alpha^{-1}(\alpha^{-1}-1)^k\right|.
\end{equation}
Let $S_\alpha^\delta=\{\alpha e^{\ii\theta},|\theta|\le\delta\}$ for some small $\delta>0$.
By the steep descent property of $S_\alpha$, the contribution of the integral over $S_\alpha\setminus S_\alpha^\delta$ in \eqref{Jkwithwtf}
is bounded by $Q(\alpha)\Or(e^{-\gamma t})$ where $\gamma>0$ does not depend on $t$.

To bound the integral over $S_\alpha^\delta$, we first observe by series expansion that
\begin{equation}\label{wtfTaylor}
\Re(\wt f(\alpha e^{\ii\theta})-\wt f(\alpha))=-\frac{\alpha-\alpha^{-1}}4\theta^2(1+\Or(\theta)).
\end{equation}
We also use that for any $z\in S_\alpha^\delta$
\begin{equation}\label{ratio1}
\left|\frac{z^{-1}(z^{-1}-1)^k}{\alpha^{-1}(\alpha^{-1}-1)^k}\right|\le\wt K\wt s^{-k/2}\le Kt^{k/3}
\end{equation}
holds for some $\wt K$ and $K$ constants by \eqref{defalpha} if $s$ is larger than some $s_0$.
If $z=\alpha e^{\ii\theta}$ is such that $|\theta|\le t^{-1/3}$, then the stronger bound
\begin{equation}\label{ratio2}
\left|\frac{z^{-1}(z^{-1}-1)^k}{\alpha^{-1}(\alpha^{-1}-1)^k}\right|\le K
\end{equation}
applies.
Putting the estimates \eqref{wtfTaylor}, \eqref{ratio1} and \eqref{ratio2} together,
we get that the integral on the right-hand side of \eqref{Jkwithwtf} taken only over $S_\alpha^\delta$ can be bounded by
\begin{multline}\label{maincontr}
Q(\alpha)\left|\frac{t^{1/3}}{2\pi\ii}\int_{S_\alpha^\delta}\d z\exp\left(2t(\wt f(z)-\wt f(\alpha))\right)\frac{z^{-1}(z^{-1}-1)^k}{\alpha^{-1}(\alpha^{-1}-1)^k}\right|\\
\le Q(\alpha)\frac{t^{1/3}}{2\pi}\alpha\bigg[\int_{[-t^{-\frac13},t^{-\frac13}]} \d\theta e^{-\frac{\alpha-\alpha^{-1}}2 \theta^2 t(1+\Or(\theta))}K
+\int_{\substack{[-\delta,\delta]\setminus\\ [-t^{-\frac13},t^{-\frac13}]}} \d\theta e^{-\frac{\alpha-\alpha^{-1}}2 \theta^2 t(1+\Or(\theta))}Kt^{k/3}\bigg]
\end{multline}
after the change of variable $z=\alpha e^{\ii\theta}$.
For $t$ large enough, the error terms in the exponents on the right-hand side of \eqref{maincontr} and factor $t^{k/3}$ in the second integral can be removed
by replacing $(\alpha-\alpha^{-1})/2$ by $(\alpha-\alpha^{-1})/4$.
Thus we get that
\begin{equation}
\eqref{maincontr}\le Q(\alpha)\frac{t^{1/3}}{2\pi}\alpha\int_{-\delta}^\delta \d\theta\exp\left(-\frac{\alpha-\alpha^{-1}}4t\theta^2\right)
\le Q(\alpha)\frac\alpha{\sqrt{2\pi\frac{\alpha-\alpha^{-1}}4 t^{1/3}}}
\end{equation}
by bounding the Gaussian integral.
The estimate above is the largest if $\alpha$ is close to $1$.
Note that, by \eqref{defalpha} and \eqref{defwts},
\begin{equation}
(\alpha-\alpha^{-1})t^{1/3}\sim2\wt s^{1/2}t^{1/3}\sim2s^{1/2}
\end{equation}
which is large if $s$ is large enough, therefore \eqref{maincontr} is at most constant times $Q(\alpha)$.

Hence, it remains to bound $Q(\alpha)$ exponentially in $s$.
For this end, we use the Taylor expansion
\begin{equation}
\wt f(z)=\left(\frac{(z-1)^3}6-\frac12\wt s(z-1)\right)(1+\Or(z-1)).
\end{equation}

If $\wt s\le\varepsilon$, then
\begin{equation}\begin{aligned}
Q(\alpha)&=\exp\left(-\frac23t\wt s^{3/2}\left(1+\Or\left(\sqrt\varepsilon\right)\right)\right)t^{k/3}\wt s^{k/2}\left(1+\Or\left(\wt s^{-1/2}\right)\right)\\
&=\exp\left(-\frac23s^{3/2}\left(1+\Or\left(\sqrt\varepsilon\right)\right)\right)s^{k/2}\left(1+\Or\left(t^{-1/3}\right)\right)
\end{aligned}\end{equation}
which is even stronger than what we had to prove.

If $\wt s>\varepsilon$, then
\begin{equation}
Q(\alpha)=\exp\left(t\sqrt\varepsilon\left(\frac{\varepsilon}3-\wt s\right)\left(1+\Or\left(\sqrt\varepsilon\right)\right)\right)t^{k/3}\varepsilon^{k/2}\left(1+\Or\left(\sqrt\varepsilon\right)\right)
\le\exp\left(-\frac13\sqrt\varepsilon t^{1/3} s\right)
\end{equation}
since $\varepsilon/3-\wt s\le-\frac23\wt s=-\frac23 t^{-2/3}s$, and the error terms can be removed by replacing $\frac23$ by $\frac13$ for any given $\varepsilon>0$ and for $t$ large enough.
This finishes the proof.
\end{proof}

\paragraph{Acknowledgements.}
The authors thank Patrik Ferrari for many illuminating discussions at different stages of the project and Mattia Cafasso for the discussion which led to Remark~\ref{rem:cafasso}.
B.\ V.\ thanks the invitation of Lun Zhang and the hospitality of the KU Leuven where the present research was initiated.
He is grateful for the generous support of the Humboldt Research Fellowship for Postdoctoral Researchers during his stay at the University of Bonn
and for the Postdoctoral Fellowship of the Hungarian Academy of Sciences.
His work was partially supported by OTKA (Hungarian National Research Fund) grant K100473.
S.\ D.\ is a Postdoctoral Fellow of the Fund for Scientific Research -- Flanders (Belgium).


\begin{thebibliography}{99}

\bibitem{AS}
M.\ Abramowitz and I.A.\ Stegun,
Handbook of Mathematical Functions.
New York: Dover Publications, 1968.

\bibitem{AFvM11}
M.\ Adler, P.L.\ Ferrari and P.\ van Moerbeke,
Non-intersecting random walks in the neighborhood of a symmetric tacnode,
Ann.\ Probab.\ {\bf 41} (2013), 2599--2647.

\bibitem{AJvM}
M.~Adler, K.~Johansson and P.~van Moerbeke,
Double Aztec diamonds and the tacnode process,
Adv.\ Math.\ {\bf 252} (2014), 518--571. 

\bibitem{BC11}
M.~Bertola and M.~Cafasso,
The Transition between the Gap Probabilities from the Pearcey to the Airy Process -- a Riemann--Hilbert Approach
Int.\ Math.\ Res.\ Not.\ IMRN {\bf7} (2012), 1519--1568.

\bibitem{BF07}
A.~Borodin and P.L.~Ferrari,
Large time asymptotics of growth models on space-like paths I: PushASEP,
Electron.\ J.\ Probab.\ \textbf{13} (2008), 1380--1418.

\bibitem{BF}
A.~Borodin and P.~L.~Ferrari,
Anisotropic growth of random surfaces in $2+1$ dimensions,
arXiv:0804.3035.

\bibitem{BK}
A.\ Borodin and J.\ Kuan,
Random surface growth with a wall and Plancherel measures for $O(\infty)$,
Comm.\ Pure Appl.\ Math.\ {\bf 63} (2010), 831--894.

\bibitem{BW}
A.\ B\"ottcher and H.\ Widom,
Szeg\H o via Jacobi,
Lin.\ Alg.\ Appl.\ {\bf 419} (2006), 656--667.

\bibitem{D}
S.~Delvaux,
The tacnode kernel: equality of Riemann-Hilbert and Airy resolvent formulas,
arXiv:1211.4845.

\bibitem{Dhet}
S.~Delvaux:
Non-intersecting squared Bessel paths at a hard-edge tacnode,
Commun.\ Math.\ Phys.\ {\bf 324} (2013), 715--766.

\bibitem{DKZ}
S.~Delvaux, A.B.J.~Kuijlaars and L.~Zhang,
Critical behavior of non-intersecting Brownian motions at a tacnode,
Comm.\ Pure Appl.\ Math.\ {\bf 64} (2011), 1305--1383.

\bibitem{DF}
P.\ Desrosiers and P.\ Forrester,
A note on biorthogonal ensembles,
J.\ Approx.\ Theory {\bf 152} (2008), 167--187.

\bibitem{FV}
P.~L.~Ferrari and B.~Vet\H o,
Non-colliding Brownian bridges and the asymmetric tacnode process,
Electron.\ J.\ Probab.\ {\bf 17} (2012), no. 44, 1--17.

\bibitem{Joh11}
K.\ Johansson,
Noncolliding Brownian motions and the extended tacnode process,
Commun.\ Math.\ Phys.\ {\bf 319} (2013), 231--267.

\bibitem{KT11}
M.~Katori and H.~Tanemura,
Noncolliding squared Bessel processes,
J.\ Stat.\ Phys.\ {\bf 142} (2011), 592--615.

\bibitem{KO01}
W.\ K\"onig and N.\ O'Connell,
Eigenvalues of the Laguerre process as non-colliding squared Bessel processes,
Elec.\ Comm.\ Probab.\ {\bf 6} (2001), no.\ 11, 107--114.

\bibitem{KMW2}
A.B.J.\ Kuijlaars, A.\ Mart\'inez-Finkelshtein and F.\ Wielonsky,
Non-intersecting squared Bessel paths: critical time and double scaling limit,
Comm.\ Math.\ Phys.\ {\bf 308} (2011), 227--279.

\bibitem{TW1}
C.\ Tracy and H.\ Widom,
Level-spacing distributions and the Airy kernel,
Comm.\ Math.\ Phys.\ {\bf 159} (1994), 151--174.

\bibitem{TW3}
C.\ Tracy and H.\ Widom,
The Pearcey process,
Commun.\ Math.\ Phys.\ {\bf 263} (2006), 381--400.

\end{thebibliography}
\end{document}